\documentclass[12pt,reqno]{amsart}
\usepackage{amsmath,amssymb,extarrows}
\usepackage{hyperref}
\usepackage{url}
\usepackage{tikz,enumerate}
\usepackage{diagbox}
\usepackage{appendix}
\usepackage{epic}
\usepackage{comment}
\usepackage{bm}
\usepackage{adjustbox}

\usepackage{booktabs}
\usepackage{array}

\usepackage{float} 

\usepackage{cite}

\usepackage{hyperref}
\usepackage{array}

\usepackage{booktabs}

\allowdisplaybreaks[4]

\newtheorem{theorem}{Theorem}
\newtheorem{corollary}[theorem]{Corollary}

\newtheorem{lemma}[theorem]{Lemma}
\newtheorem{prop}[theorem]{Proposition}
\theoremstyle{definition}

\theoremstyle{remark}
\newtheorem{rem}{Remark}
\numberwithin{equation}{section}
\numberwithin{theorem}{section}
\numberwithin{defn}{section}

\newcommand{\s}{\mathcal{S}}

\newcommand{\qbinom}[2]{\genfrac{[}{]}{0pt}{}{#1}{#2}}

\newcommand{\ii}{\mathrm{i}}

\newcommand{\padedvphantom}[3]{%
	\vtop{%
		\vbox{%
			\vspace*{#2}%
			\hbox{\vphantom{#1}}%
		}%
		\vspace*{#3}%
	}%
}

\begin{document}
\title[Nahm Sums Dual to Zagier's Rank Three Examples and Related Identities]
 {Nahm Sums Dual to Zagier's Rank Three Examples and Related Identities}

\author{Changsong Shi and Liuquan Wang}

\address[C.\ Shi]{School of Mathematics and Statistics, Wuhan University, Wuhan 430072, Hubei, People's Republic of China}
\email{changsong@whu.edu.cn}

\address[L.\ Wang]{School of Mathematics and Statistics, Wuhan University, Wuhan 430072, Hubei, People's Republic of China}
\email{wanglq@whu.edu.cn;mathlqwang@163.com}



\begin{abstract}
In 2007, Zagier identified twelve families of candidates for rank three modular Nahm sums and established the modularity of three of them. The remaining cases were subsequently confirmed by Wang. Zagier's duality observation indicates that the Nahm sums associated with the duals of these examples might still be modular. We confirm that this is indeed true. The modularity of the dual of the sixth example was previously established by Milas and Wang, while that of the dual of  the eleventh example was partially established by Shi and Wang. We settle all remaining cases, thereby establishing the modularity of every defined dual in Zagier's rank three list. We achieve this by proving new Rogers--Ramanujan type identities that express the relevant Nahm sums as finite sums of infinite products.  Along the way we  discover two new families of modular rank three Nahm sums. As applications, we prove some Nahm sum identities conjectured by Cao--Wang, Shi--Wang and Li.
\end{abstract}

\maketitle

\tableofcontents

\section{Introduction}\label{sec-intro}
The study of the modularity of $q$-series originates in the 19th-century theories of theta and elliptic functions. The classical Rogers--Ramanujan identities~\cite{Rogers1894} are among the most prominent examples of this interplay:
\begin{align}\label{RR}
\sum_{n=0}^{\infty}\frac{q^{n^2+an}}{(q;q)_n}=\frac{1}{(q^{1+a},q^{4-a};q^5)_\infty},
\quad a=0,1.
\end{align}
For $a=0$ and $a=1$, these series become modular upon multiplication by $q^{-1/60}$ and $q^{11/60}$, respectively. Here and throughout, we assume that $q=e^{2\pi \ii \tau}$ with $\mathrm{Im}~\tau>0$ so that $|q|<1$ and use the standard $q$-series notation:
\begin{align}
(a;q)_n&:=\prod_{k=0}^{n-1}(1-aq^k), \quad n\in\mathbb{Z}_{\geq0}, \\
(a;q)_\infty&:=\prod_{k=0}^{\infty}(1-aq^k), \\
(a_1,\dots,a_m;q)_n&:=\prod_{k=1}^{m}(a_k;q)_n, \quad n\in\mathbb{Z}_{\geq0}\cup \{\infty\}.
\end{align}

A particularly intriguing class of $q$-hypergeometric series, known as \emph{Nahm sums}, was introduced by Nahm in the 1990s in connection with two-dimensional conformal field theory (CFT) and the representation theory of affine Lie algebras~\cite{Nahm1994,Nahmconf,Nahm2007}. A Nahm sum of rank $r$ is a $q$-hypergeometric series of the form
\begin{equation}\label{eq-Nahm}
f_{A,B,C}(q):=\sum_{n=(n_1,\dots,n_r)\in\mathbb{Z}_{\geq 0}^r}\frac{q^{\frac12{n}^{\mathrm{T}}A{n}+B^{\mathrm{T}}n+C}}{(q;q)_{n_1}\cdots(q;q)_{n_r}},
\end{equation}
where $A$ is a symmetric positive-definite $r\times r$ matrix with rational entries, $B\in\mathbb{Q}^r$, and $C\in\mathbb{Q}$.  Nahm posed the problem of classifying all triples $(A,B,C)$ for which $f_{A,B,C}(q)$ is modular. Any such triple is commonly referred to as a \emph{modular triple}. Nahm sums provide a natural framework encompassing many classical $q$-series identities. For example, the Rogers--Ramanujan identities \eqref{RR}  give rise to the two rank one modular triples $(A,B,C)=(2,0,-1/60)$ and $(A,B,C)=(2,1,11/60)$, respectively.

Motivated by fermionic character formulas arising in rational conformal field theory, Nahm \cite{Nahm2007} observed that certain characters admit $q$-hypergeometric representations of the form \eqref{eq-Nahm} and related the expected modular behavior of such series to torsion phenomena in Bloch groups. A precise formulation of this connection was given by Zagier~\cite{Zagier} and became known as \emph{Nahm's conjecture}. In its original form, the conjecture asserts that a rational symmetric positive-definite matrix $A$ admits vectors $B\in\mathbb{Q}^r$ and constants $C\in\mathbb{Q}$ for which $f_{A,B,C}(q)$ is modular if and only if every solution of the algebraic system determined by $A$ gives rise to a torsion class in the appropriate Bloch group. This conjectural framework has stimulated extensive research and revealed deep connections among Nahm sums, dilogarithm identities, torsion phenomena in algebraic $K$-theory, and the character theory of vertex operator algebras. Although the proposed equivalence holds in rank one, Vlasenko and Zwegers~\cite{VZ} exhibited counterexamples in higher rank, showing that its original formulation requires modification. More recently, Calegari, Garoufalidis, and Zagier~\cite{CGZ} established an important necessary implication: if $f_{A,B,C}(q)$ is modular, then the Bloch-group class associated with the distinguished positive real solution of the algebraic system determined by $A$ vanishes.

In 2007, Zagier~\cite{Zagier} made substantial progress on Nahm's problem. He completely resolved the rank one case, proving that there are exactly seven modular triples of rank one, namely,
\begin{equation}\label{eq-rank1}
\begin{split}
   &(1/2, 0, -1/40), ~(1/2, 1/2, 1/40), ~ (1, 0, -1/48), ~(1, 1/2, 1/24),  \\
   & (1,-1/2,1/24), ~(2,0,-1/60), ~ (2,1,11/60).
\end{split}
\end{equation}
For ranks two and three, he also produced lists consisting of eleven and twelve sets of conjectural modular triples, respectively. The modularity of all these examples has since been established through work of Zagier~\cite{Zagier}, Vlasenko and Zwegers~\cite{VZ}, Cherednik and Feigin~\cite{Feigin}, Wang~\cite{Wang-rank2,Wang-rank3}, and Cao, Rosengren, and Wang~\cite{CRW}. It is also worth emphasizing that Vlasenko and Zwegers~\cite{VZ} proved that every modular Nahm sum necessarily has weight zero.

A common strategy in~\cite{VZ,Wang-rank2,Wang-rank3,CRW} is to establish Rogers--Ramanujan type identities that express $f_{A,B,0}(q)$ as finite sums of infinite-product quotients built from
\begin{align}\label{J-defn}
    &J_m:=(q^m;q^m)_\infty, \quad J_{a,m}:=(q^a,q^{m-a},q^m;q^m)_\infty, \\
     &\overline{J}_{a,m}:=(-q^a,-q^{m-a},q^m;q^m)_\infty=\frac{J_{2a,2m}J_m^2}{J_{a,m}J_{2m}}.
\end{align}
It is well known that $q^{\frac{m}{24}}J_m$ and $q^{\frac{m}{24}+\frac{m}{2}P_2(\frac{a}{m})} J_{a,m}$ are modular forms of weight $1/2$ on suitable congruence subgroups, where $P_2(t)=t^2-t+\frac16$ for $0\leq t<1$ and is extended periodically. These product representations therefore make it possible to determine an appropriate constant $C$ and verify the modularity of $f_{A,B,C}(q)=q^C f_{A,B,0}(q)$. The main challenge is therefore to establish the requisite Rogers--Ramanujan type identities, a task that can be highly nontrivial even for low-rank Nahm sums. A notable example is provided by the conjectural Kanade--Russell identities modulo~9~\cite{KR2015,Kursungoz2019}, which equate certain double sums with infinite products and remain open despite their deceptively simple form.

An important source of modular triples is provided by Cartan matrices attached to Dynkin diagrams. For a Dynkin diagram $X$, let $\mathcal{C}(X)$, $r(X)$, and $h(X)$ denote its Cartan matrix, rank, and Coxeter number, respectively. It has long been expected that, for every pair $(X,Y)$ of ADET-type Dynkin diagrams, the Nahm sum associated with the triple $\bigl(M(X,Y),0,C(X,Y)\bigr)$ is modular, where
\begin{align}\label{eq-folklore}
M(X,Y)=\mathcal{C}(X)\otimes\mathcal{C}(Y)^{-1}, \quad 
C(X,Y)=-\frac{1}{24}\frac{r(X)r(Y)h(X)}{h(X)+h(Y)}.
\end{align}
This folklore conjecture has been verified for several infinite families. For example, the $(A_1,T_r)$ case follows from the classical Andrews--Gordon identities~\cite{Gordon1961,Andrews1974}, while the $(T_1,T_r)$ case follows from an identity of Stembridge~\cite[Corollary~1.5(b)]{Stembridge}. The $(A_1,D_r)$ case is a consequence of identities conjectured by Kedem, Klassen, McCoy, and Melzer~\cite[(2.9)]{Kedem} and subsequently proved by Warnaar~\cite{Warnaar2007}. These identities were recently generalized, with new proofs, by Wang and Wang~\cite{Wang-Wang2025}.

The $(T_r,T_1)$ case was investigated by Calinescu, Milas, and Penn~\cite{CMP} in their study of twisted modules for principal-subspace vertex algebras. They established this case for $r=2$; the cases $r=3$ and $r=4,5$ were subsequently proved by Milas and Wang~\cite{MW24} and by Shi and Wang ~\cite{Shi-Wang}, respectively. In particular, Shi and Wang proved several rank four tadpole Nahm sum identities in~\cite[Theorem~1.5]{Shi-Wang}, but left the following identity as a conjecture:
\begin{align}\label{id-SW}
\sum_{i,j,k,l\geq0} \frac{q^{\frac12\left(i^2+(i-j)^2+(j-k)^2+(k-l)^2\right)
+i-j+k-\frac12l}}{(q;q)_i(q;q)_j(q;q)_k(q;q)_l}=6\frac{J_2^2J_3^3}{J_1^5}.
\end{align}
Conjectural formulas for rank $r$ tadpole Nahm sums associated with the zero vector for arbitrary $r$ were proposed in~\cite[Conjecture~4.5]{MW24}, and they were confirmed quite recently by Chern, Tran and Wakhare \cite{Chern} based on the work of Bartlett and Warnaar \cite{Bartlett-Warnaar}. Sun and Wang~\cite{SW} proposed a conjectural identity for the Nahm sum associated with $(T_1,D_r)$. This identity was subsequently proved by Wang and Wang~\cite{Wang-Wang2026}, thereby establishing the $(T_1,D_r)$ case. We refer the reader to~\cite{SW} for a more detailed introduction to Nahm sums associated with Dynkin diagrams.

Zagier~\cite{Zagier} also identified an important structural pattern among modular triples. For a modular triple $(A,B,C)$ of rank $r$, define its dual by
\begin{align}\label{eq-dual}
   \mathcal{D}(A,B,C)=(A^\star, B^\star, C^\star):=(A^{-1},A^{-1}B,\frac{1}{2}B^\mathrm{T} A^{-1}B-\frac{r}{24}-C).
\end{align}
Zagier~\cite[p.~50, (f)]{Zagier} suggested that ``if $f_{A,B,C}$ is modular then
$f_{A^\star,B^\star,C^\star}$ is also modular''.  This principle can be verified directly for the seven rank one modular triples in \eqref{eq-rank1}. Moreover, the dual of every rank two modular triple listed in~\cite[Table~2]{Zagier} also appears in the same table. The duality principle therefore holds for all of these rank two examples as well~\cite{Wang-rank2}. More recently, Wang~\cite{Wang-2026} constructed rank four counterexamples to this duality principle, which also disprove Mizuno's duality conjecture for generalized Nahm sums~\cite[Conjecture~4.1]{Mizuno}. Rank three counterexamples were subsequently obtained by Cao and Wang~\cite{Cao-Wang-rank3}. A common feature of these counterexamples is that the original Nahm sums $f_{A,B,C}(q)$ are weight-zero modular functions, whereas the corresponding dual Nahm sums decompose into the sum of two modular forms, one of weight zero and the other of weight one. This phenomenon suggests that an appropriately modified version of Zagier's duality principle may nevertheless remain valid.

Motivated by Zagier's duality principle, Cao and Wang~\cite{Cao-Wang-rank3,Cao-Wang-rank4,Cao-Wang2026} introduced the \emph{lift-dual operation} and used it to construct new modular triples of ranks three and four. They also proposed several conjectural identities for rank four tadpole Nahm sums in~\cite[Conjecture~3.2]{Cao-Wang-rank4}:
\begin{align}
    \sum_{i,j,k,l\ge 0}\frac{q^{\frac{1}{2}(i^2+(i-j)^2+(j-k)^2+(k-l)^2)+2i-2j+k-\frac{1}{2}l}}{(q;q)_i(q;q)_j(q;q)_k(q;q)_l}&=8q^{-1}\frac{J_2^6}{J_1^6}, \label{id-CW-1}\\
        \sum_{i,j,k,l\ge 0}\frac{q^{\frac{1}{2}(i^2+(i-j)^2+(j-k)^2+(k-l)^2)+\frac{3}{2}i-\frac{3}{2}j+\frac{1}{2}k}}{(q;q)_i(q;q)_j(q;q)_k(q;q)_l}&=2q^{-\frac{1}{2}}\frac{J_1^3}{J_{\frac{1}{2}}^3}, \label{id-CW-2}\\
        \sum_{i,j,k,l\ge 0}\frac{q^{\frac{1}{2}(i^2+(i-j)^2+(j-k)^2+(k-l)^2)+i-j+\frac{1}{2}l}}{(q;q)_i(q;q)_j(q;q)_k(q;q)_l}&=4\frac{J_2^6}{J_1^6}. \label{id-CW-3}
\end{align}

As announced in~\cite[p.~44]{Wang-rank3}, the main purpose of this paper is to investigate the modularity of the Nahm sums dual to Zagier's rank three examples listed in~\cite[Table~3]{Zagier}. Following~\cite{Wang-rank3}, we label these entries Examples~1--12 according to their order of appearance in Zagier's table. The matrices associated with Examples~4 and~5 are singular; hence the dual \eqref{eq-dual} is undefined for these two examples. Consequently, only the duals of Examples~1--3 and~6--12 are worthy of consideration.

Some of these duals have already been studied. The dual of Example~6 gives rise to the tadpole Nahm sums associated with the Cartan matrix $\mathcal{C}(T_3)$, whose modularity was established by Milas and Wang~\cite[Theorem~1.2]{MW24}. Example~11 gives rise to five dual Nahm sums, four of which were proved to be modular by Shi and Wang~\cite{Shi-Wang} through the following identities:
\begin{align}
&\sum_{i,j,k\geq 0} \frac{q^{i^2+3j^2+4k^2-2ij+4jk}}{(q^4;q^4)_i(q^4;q^4)_j(q^4;q^4)_k}
=\frac{J_{12}J_{24}^4}{J_{1,24}J_{3,24}J_{8,24}J_{9,24}J_{11,24}}, \label{dual-1} \\
&\sum_{i,j,k\geq 0} \frac{q^{i^2+3j^2+4k^2-2ij+4jk+2i-2j}}{(q^4;q^4)_i(q^4;q^4)_j(q^4;q^4)_k}=\frac{J_{24}^{9/2}J_{12,24}^{1/2}}{J_{1,24}J_{5,24}J_{7,24}J_{8,24}J_{11,24}}, \label{dual-3} \\
&\sum_{i,j,k\geq 0} \frac{q^{i^2+3j^2+4k^2-2ij+4jk+2i}}{(q^4;q^4)_i(q^4;q^4)_j(q^4;q^4)_k}=\frac{J_{24}^{9/2}J_{6,24}^3}{J_{3,24}^2J_{4,24}^2J_{8,24}J_{9,24}^2J_{12,24}^{1/2}},\label{dual-4} \\
&\sum_{i,j,k\geq 0} \frac{q^{i^2+3j^2+4k^2-2ij+4jk+2i+2j+4k}}{(q^4;q^4)_i(q^4;q^4)_j(q^4;q^4)_k}
=\frac{J_{24}^{9/2}J_{12,24}^{1/2}}{J_{3,24}J_{5,24}J_{7,24}J_{8,24}J_{9,24}}. \label{dual-5}
\end{align}
For the remaining case, Shi and Wang proposed the following conjectural identity in~\cite[Conjecture~1.4]{Shi-Wang}:
\begin{align}
    \sum_{i,j,k\geq 0} \frac{q^{i^2+3j^2+4k^2-2ij+4jk+2j+4k}}{(q^4;q^4)_i(q^4;q^4)_j(q^4;q^4)_k}
=\frac{J_{2,24}J_{24}^{9/2}J_{6,24}J_{10,24}}{J_{1,24}J_{4,24}^2J_{5,24}J_{7,24}J_{8,24}J_{11,24}J_{12,24}^{1/2}}. \label{dual-2}
\end{align} 

To the best of our knowledge, duals of the other examples have not previously been discussed. In this paper, we treat these examples individually and confirm the modularity of their duals. As we will see, Example 1 is dual to itself, Examples 2 and 3 are dual to each other. Hence the modularity of the duals of Examples 1--3 are known. We mainly work on Nahm sums dual to Examples 7--12 and confirm all their modularity. As an illustration, we present the main results for the dual of Examples 9, 11 and 12 here.

The dual of Example 9 is particularly intricate. Let
\begin{align}
    M(u,v,w;q):=\sum_{i,j,k\geq 0}\frac{q^{i^2+2j^2+2k^2-2ij-2jk}u^iv^jw^k}{(q^4;q^4)_i(q^4;q^4)_j(q^4;q^4)_k}. \label{M-defn-general}
\end{align}
The Nahm sums dual to Example 9 with $q$ replaced by $q^4$ are given by
\begin{equation}\label{M-defn}
\begin{split}
    M^{(1)}(q)&:=M(1,1,1;q), \quad M^{(2)}(q):=M(q^2,q^{-2},1;q), \\
    M^{(3)}(q)&:=M(q^2,1,q^{-2};q), \quad M^{(4)}(q):=M(q^2,1,1;q). 
\end{split}
\end{equation}
To evaluate them we split the sums into several parts with restrictions on the parity of the summation indices:
\begin{align}
    M_{x,y,z}(u,v,w;q):=\sum_{\substack{i,j,k\geq 0\\(i,j,k)\equiv (x,y,z)\ \mathrm{mod}\ 2}}\frac{q^{i^2+2j^2+2k^2-2ij-2jk}u^iv^jz^k}{(q^4;q^4)_i(q^4;q^4)_j(q^4;q^4)_k}.\label{def-Mxyz}
\end{align}
It suffices to evaluate the following series with $(x,y,z)\in \{0,1\}^3$:
\begin{equation}
\begin{split}
    M_{x,y,z}^{(1)}(q)&:=M_{x,y,z}(1,1,1;q), \quad  M_{x,y,z}^{(2)}(q):=M_{x,y,z}(q^2,q^{-2},1;q), \\
    M_{x,y,z}^{(3)}(q)&:=M_{x,y,z}(q^2,1,q^{-2};q), \quad  M_{x,y,z}^{(4)}(q):=M_{x,y,z}(q^2,1,1;q).
\end{split}
\end{equation} 
We will present two different approaches to evaluate these parts and express them as infinite products. Then we are able to confirm the modularity of these Nahm sums.
\begin{theorem}\label{thm-dual-exam9-1st}
For $h=1,2,3,4$ we have
\begin{align}\label{eq-thm-Mh}
    M^{(h)}(q)=-\mathcal{F}_0^{(h)}(q)+\mathcal{F}_1^{(h)}(q)+2\ii^{h-2}\Big(\mathcal{F}_0(\ii q)+\mathcal{F}_1^{(h)}(\ii q)\Big)
\end{align}
where $\mathcal{F}_k^{(h)}(q)$ ($k=0,1$) admit the explicit product representations in \eqref{eq-F1}--\eqref{eq-F4}.
\end{theorem}
It is worth noting that Theorem \ref{thm-dual-exam9-1st} confirms the $(T_3,A_1)$ case of the aforementioned folklore conjecture (see \eqref{eq-folklore}). 

We also settle the remaining case of Example~11, thereby completing the proof of its modularity and filling the gap left in~\cite{Shi-Wang}.
\begin{theorem}\label{thm-dual-2}
The conjectural identity \eqref{dual-2} holds.
\end{theorem}

As for the dual of Example 12, we establish the following identities to confirm their modularity. \begin{theorem}\label{thm-dual-exam12}
We have
\begin{align}
&\sum_{i,j,k\geq 0} \frac{q^{\frac{3}{2}i^2+\frac{7}{2}j^2+4k^2-4ij-3ik+4jk+\frac{1}{2}i-\frac{3}{2}j+2k}}{(q^5;q^5)_i(q^5;q^5)_j(q^5;q^5)_k} \label{dual-12-1} \\
&=2\frac{J_{50}^{11}}{J_{5,50}^4J_{10,50}J_{15,50}^3J_{20,50}^2J_{25}}+q\frac{J_{50}^{11}J_{10,50}}{J_{5,50}^5J_{15,50}^2J_{20,50}^4J_{25}}
+4q^2\frac{J_{50}^{12}}{J_{5,50}^3J_{10,50}^2J_{15,50}^3J_{20,50}J_{25}^3},  \nonumber\\
&\sum_{i,j,k\geq 0} \frac{q^{\frac{3}{2}i^2+\frac{7}{2}j^2+4k^2-4ij-3ik+4jk+\frac{5}{2}i-\frac{5}{2}j}}{(q^5;q^5)_i(q^5;q^5)_j(q^5;q^5)_k}  \label{dual-12-2} \\
&=\frac{J_{20,50}J_{50}^{11}}{J_{5,50}^2J_{10,50}^4J_{15,50}^5J_{25}}+2q\frac{J_{50}^{11}}{J_{5,50}^3J_{10,50}^2J_{15,50}^4J_{20,50}J_{25}}
+4q^4\frac{J_{50}^{12}}{J_{5,50}^3J_{10,50}J_{15,50}^3J_{20,50}^2J_{25}^3}. \nonumber
\end{align}
\end{theorem}
This result was initially discovered with the aid of Garvan's $q$-series Maple package \cite{Garvan}. Specifically, we computed the $5$-dissections of the relevant Nahm sums and observed that each component either vanishes or appears to admit a representation as a single infinite product. The proof, however, is considerably more involved: its key step is to reduce these rank three Nahm sums to a collection of single sums with known product evaluations.

In the course of our investigation of the duals of Zagier's examples, we obtain several additional results that are of independent interest.

As a by-product of our investigation of Example 9, we discover the following new family of rank three modular Nahm sums.
\begin{theorem}\label{thm-ex9-new}
We have
 \begin{align}
\sum_{i,j,k\geq 0}\frac{q^{4i^2+3j^2+2k^2+4ij-2jk+4i+2j}}{(q^4;q^4)_i(q^4;q^4)_j(q^4;q^4)_k}        &=\frac{J_4J_{4,72}J_{28,72}J_{32,72}J_{36,72}}{J_2J_8J_{14,72}J_{18,72}J_{22,72}}+\frac{2q^5J_8J_{8,72}}{J_4^2}, \label{eq-thm-T-1}\\
\sum_{i,j,k\geq 0}\frac{q^{4i^2+3j^2+2k^2+4ij-2jk-2j}}{(q^4;q^4)_i(q^4;q^4)_j(q^4;q^4)_k}  &=\frac{J_4J_{24}J_{12,24}}{J_2J_8J_{6,24}}+\frac{2qJ_8J_{24}}{J_4^2},  \label{eq-thm-T-2}\\
\sum_{i,j,k\geq 0}\frac{q^{4i^2+3j^2+2k^2+4ij-2jk-2k}}{(q^4;q^4)_i(q^4;q^4)_j(q^4;q^4)_k}  &=\frac{qJ_4J_{4,72}J_{16,72}J_{20,72}J_{36,72}}{J_2J_8J_{2,72}J_{18,72}J_{34,72}}+\frac{2J_8J_{32,72}}{J_4^2}, \label{eq-thm-T-3}\\
\sum_{i,j,k\geq 0}\frac{q^{4i^2+3j^2+2k^2+4ij-2jk}}{(q^4;q^4)_i(q^4;q^4)_j(q^4;q^4)_k}  &=\frac{J_4J_{8,72}J_{20,72}J_{28,72}J_{36,72}}{J_2J_8J_{10,72}J_{18,72}J_{26,72}}+\frac{2q^3J_8J_{16,72}}{J_4^2}.  \label{eq-thm-T-4}
\end{align}
\end{theorem}
Motivated by Zagier's duality observation, it is natural to conjecture that the duals of these Nahm sums are also modular (see Table \ref{tab-newexample}). This is indeed true and we establish the following theorem which implies its modularity.
\begin{theorem}\label{thm-dual-rank3}
We have
\begin{align}
 &\sum_{i,j,k\geq 0}\frac{q^{5i^2+8j^2+8k^2-8ij-4ik+8jk+4i-8j-4k}}{(q^{12};q^{12})_i(q^{12};q^{12})_j(q^{12};q^{12})_k}=\frac{2J_{18}^3J_{36}}{J_{9,36}^2J_{12,36}^2}+\frac{qJ_{18}J_{36}^2J_{6,36}}{J_{3,36}J_{12,36}^2J_{15,36}},  \label{eq-rank3-dual-1}\\
 &\sum_{i,j,k\geq 0}
\frac{
q^{5i^2+8j^2+8k^2-8ij-4ik+8jk+6i}
}{
(q^{12};q^{12})_i
(q^{12};q^{12})_j
(q^{12};q^{12})_k
}   \label{eq-rank3-dual-2}\\
&=\frac{J_6^{2}}{J_3J_{12}^{3}}
\left(\frac{J_9J_{12}J_{54}^{3}J_{108}^{2}J_{3,54}J_{12,54}}{J_{18}^{2}J_{27}^{2}J_{6,108}J_{24,108}^{2}}
-q^{21}\frac{J_{27}J_{15,54}J_{6,108}J_{12,108}^{2}}{J_{54}^{2}J_{30,108}}
+2q^8
\frac{J_9J_{36}J_{12,108}}{J_{18}}
\right), \nonumber
 \\
 &\sum_{i,j,k\geq 0}\frac{q^{5i^2+8j^2+8k^2-8ij-4ik+8jk+2i-4j-8k}
}{(q^{12};q^{12})_i(q^{12};q^{12})_j(q^{12};q^{12})_k}  \label{eq-rank3-dual-3} \\
&=2\frac{J_6^{2}J_9J_{36}J_{48,108}}{J_3J_{12}^{3}J_{18}}
+q^4\frac{J_6^{4}J_{18}J_{27}^{2}J_{108}^{4}J_{6,54}J_{21,54}J_{48,108}}{
J_3^{2}J_{12}^{3}J_{36}^{2}J_{54}^{6}J_{24,108}}
+q^{25}\frac{J_6J_{18}J_{108}^{2}J_{3,54}J_{6,108}^{2}}{J_{12}J_{36}^{3}J_{54}
J_{24,54}J_{24,108}},  \nonumber \\
&\sum_{i,j,k\geq 0}\frac{q^{5i^2+8j^2+8k^2-8ij-4ik+8jk}}{(q^{12};q^{12})_i
(q^{12};q^{12})_j(q^{12};q^{12})_k}  \label{eq-rank3-dual-4}
 \\
&=\frac{J_{18}^{2}J_{27}J_{54}J_{42,108}^{2}J_{48,108}}{
J_9J_{12}J_{36}^{2}J_{21,54}^{2}J_{30,108}}
+q^{24}\frac{J_6J_{18}J_{108}J_{15,54}J_{6,108}^{2}J_{24,108}
}{J_{12}^{2}J_{36}^{2}J_{54}^{2}J_{42,108}}
+2q^5\frac{J_6^{\,2}J_9J_{36}J_{24,108}}{J_3J_{12}^{3}J_{18}}. \nonumber 
 \end{align}
\end{theorem}

Furthermore, using the results we established for the dual of Zagier's rank three examples, we are able to prove some conjectural identities in the literature. 

First, we are able to prove some previously conjectured identities for rank four tadpole Nahm sums.
\begin{theorem}\label{thm-CSW-conj}
The conjectural identities \eqref{id-SW} and \eqref{id-CW-1}--\eqref{id-CW-3} hold.
\end{theorem}
Note that Shi and Wang \cite[Theorem 1.5]{Shi-Wang} proved \eqref{id-SW} by assuming \eqref{dual-2}. Hence Theorem \ref{thm-dual-2} already implies \eqref{id-SW}. We will present another proof which proves \eqref{id-SW} first and then \eqref{dual-2}. Similar arguments allow us to prove \eqref{id-CW-1}--\eqref{id-CW-3} based on some identities for the dual of Zagier's Example 8. We noted that Li \cite{Li-arXiv} recently proved \eqref{id-CW-1}--\eqref{id-CW-3} independently using a different approach.

Remarkably, the dual of Example 9 is connected to the new rank four Nahm sums in the work of Cao--Wang \cite[Table 4]{Cao-Wang-rank4}.  Let
\begin{align}\label{eq-N-defn}
    N(u,v,w;q):=\sum_{i,j,k,l\geq 0}\frac{q^{2i^2+2j^2+k^2+l^2-2ij-2jk-kl}u^{i-j}v^kw^l}{(q^2;q^2)_i(q^2;q^2)_j(q^2;q^2)_k(q^2;q^2)_l}.
\end{align}
Cao and Wang \cite[Table 4 and Conjecture 4.4]{Cao-Wang-rank4} conjectured that the following Nahm sums are modular: 
\begin{equation}
\begin{split}
N^{(1)}(q)&:=N(1,1,1;q), \quad N^{(2)}(q):=N(q^{-2},1,q^{-1};q),   \\
N^{(3)}(q)&:=N(q^{-4},q^{-1},1;q), \quad N^{(4)}(q):=N(q^{-2},1,1;q).
\end{split}
\end{equation}
 In particular, they \cite[Conjecture 3.3]{Cao-Wang-rank4} conjectured that
    \begin{align}\label{eq-CW-conj3.3}
        \sum_{i,j,k,l\geq 0}\frac{q^{2i^2+2j^2+k^2+l^2-2ij-2jk-kl-2i+2j-l}}{(q^2;q^2)_i(q^2;q^2)_j(q^2;q^2)_k(q^2;q^2)_l}=9\frac{J_3^3J_6^2}{J_1J_2^4}.
    \end{align}
They were not able to provide product expressions for the other three Nahm sums. These Nahm sums arise as the lift-dual of the Nahm sums in Zagier’s Example 9. However, it is not a priori clear whether there is any relation between them and the dual of Example 9. To our surprise, we find that they are closely related to each other. Recall one of Ramanujan's theta functions:
\begin{align}\label{eq-varphi}
    \varphi(q):=\sum_{n=-\infty}^\infty q^{n^2}=\frac{J_2^5}{J_1^2J_4^2}.
\end{align}

 \begin{theorem}\label{thm-CW-conj3.3}
For $h=1,2,3,4$ we have
\begin{align}\label{eq-thm-Nh}
    N^{(h)}(q^2)=\frac{1}{2}\frac{q^{c_h}}{J_4}\big(\varphi(q)M^{(h)}(q)-(-1)^h\varphi(-q)M^{(h)}(-q)\big)
\end{align}
where $c_0=0$, $c_1=-1$, $c_2=-4$ and $c_4=-1$. 
As a consequence, the conjectural identity \eqref{eq-CW-conj3.3} holds.
\end{theorem}
To work out a  proof of \eqref{eq-CW-conj3.3},  Li \cite[Conjecture 5.2]{Li-arXiv} conjectured two identities  which we rewrite as:
\begin{align}
    \sum_{i,j,k\geq 0} \frac{q^{2i^2+j^2+k^2-2ij-jk+2i-j}}{(q^2;q^2)_{2i}(q^2;q^2)_j(q^2;q^2)_{k}}&=3\frac{J_3^3\overline{J}_{4,12}}{J_1J_2^2J_6}, \label{eq-Li-1}\\
    \sum_{i,j,k\geq 0} \frac{q^{2i^2+j^2+k^2-2ij-jk+4i-2j+1}}{(q^2;q^2)_{2i+1}(q^2;q^2)_j(q^2;q^2)_{k}}&=3\frac{J_3^3\overline{J}_{2,12}}{J_1J_2^2J_6}. \label{eq-Li-2}
\end{align}
He proved that \eqref{eq-CW-conj3.3} follow from \eqref{eq-Li-1} and \eqref{eq-Li-2}. It turns out that the partial Nahm sums in \eqref{eq-Li-1} and \eqref{eq-Li-2} are essentially the even and odd power parts, respectively, of the Nahm sums $M^{(2)}(q)$. Hence a direct application of our computation of $M^{(2)}_{x,y,z}(q)$ yields a proof for \eqref{eq-Li-1} and \eqref{eq-Li-2} (see Corollary \ref{cor-Li-conj}).

We briefly describe the main methods and ideas used to prove the Rogers--Ramanujan type identities in this paper. For the dual of Examples 7, 8 and 10, we evaluate the corresponding Nahm sums primarily through direct $q$-series manipulations involving transformation and summation formulas.  The proof of Theorem \ref{thm-ex9-new} is based on the Bailey pair machinery. The dual of Example 11 (Theorem \ref{thm-dual-2}) and Theorem \ref{thm-CW-conj3.3} are more difficult and we need delicate applications of the constant-term method as some key steps. A key identity in the proof of Theorem \ref{thm-dual-2} and the identities for the dual of Example 8 lead to the proof of Theorem \ref{thm-CSW-conj}. The proofs of Theorems \ref{thm-dual-exam9-1st}, \ref{thm-dual-exam12} and \ref{thm-dual-rank3} are considerably even harder and require  several new techniques. For Theorems \ref{thm-dual-exam9-1st} and \ref{thm-dual-rank3} we begin by isolating the summations over two suitably chosen variables and introducing certain associated sequences. We then derive recurrence relations for these sequences and express them as linear combinations of explicit solutions involving $q$-binomial coefficients. From these expressions, the relevant Bailey pairs arise naturally. Substituting the resulting formulas back into the original Nahm sums, we obtain the desired product representations by applying the standard Bailey-pair method. For Theorem \ref{thm-dual-exam12}, we first decompose the relevant Nahm sums into linear combinations of double sums multiplied by theta products and then reduce these double sums to well-known single sums.

The rest of this paper is organized as follows. In Section \ref{sec-pre}, we collect and establish several auxiliary $q$-series identities and recall some background on Bailey pairs. In Section \ref{sec-dual}, we study the duals of Zagier's rank three examples individually. Section \ref{sec-rank3} is devoted to proofs of Theorems \ref{thm-ex9-new} and \ref{thm-dual-rank3}. Finally, in Section \ref{sec-proof}, we provide proofs of Theorems \ref{thm-CSW-conj} and \ref{thm-CW-conj3.3} and the conjectural identities \eqref{eq-Li-1}--\eqref{eq-Li-2}. 

\section{Preliminaries}\label{sec-pre}
\subsection{Some sum-to-product identities}
We first review some $q$-series identities which will be useful in our proofs.

The $q$-binomial theorem \cite[Theorem 2.1]{Andrews-book} states that
\begin{align}\label{q-binomial}
\sum_{n=0}^\infty \frac{(a;q)_n}{(q;q)_n}z^n=\frac{(az;q)_\infty}{(z;q)_\infty}, \quad |z|<1.
\end{align}
As its consequences, Euler's $q$-exponential identities \cite[Corollary 2.2]{Andrews-book} assert that
\begin{align}\label{Euler}
\sum_{n=0}^\infty \frac{z^n}{(q;q)_n}=\frac{1}{(z;q)_\infty},  \quad
\sum_{n=0}^\infty \frac{z^nq^{\frac{n^2-n}{2}}}{(q;q)_n}=(-z;q)_\infty.
\end{align}
Here we need $|z|<1$ in the first identity for convergence. 

The $q$-binomial coefficient is defined as
\begin{align}
    \qbinom{n}{m}_{q}:=\begin{cases}
         \frac{(q;q)_n}{(q;q)_m(q;q)_{n-m}}&0\leq m\leq n,\\
         0&\text{otherwise.}
    \end{cases}
\end{align}

As a finite version of \eqref{q-binomial}, we have\cite[Theorem 3.3]{Andrews-book}
\begin{align}
    (-z;q)_n=\sum_{i=0}^n\qbinom{n}{i}_qq^{i(i-1)/2}z^i.\label{q-binomial-finite}
\end{align}

Recall the Jacobi triple product identity \cite[Theorem 2.8]{Andrews-book}: 
\begin{align}\label{Jacobi}
j(z;q):=(z,q/z,q;q)_\infty=\sum_{n=-\infty}^{{\infty}} (-z)^n q^{\frac{n^2-n}{2}}.
\end{align}
We denote $D_z:=z\frac{\partial}{\partial z}$ and define
\begin{align}
 \mathcal{L}(z;q):=D_z\log j(z;q).
\end{align}
By the product definition, we have
\begin{align}
 j(q/z;q)&=j(z;q), \quad
 \mathcal{L}(q/z;q)=-\mathcal{L}(z;q),
 \label{eq-inverse}\\
 j(z^{-1};q)&=-z^{-1}j(z;q), \quad \mathcal{L}(z^{-1};q)=1-\mathcal{L}(z;q),
 \label{eq-reciprocal}\\
 j(qz;q)&=-z^{-1}j(z;q), \quad \mathcal{L}(qz;q)=\mathcal{L}(z;q)-1.
 \label{eq-shift}
\end{align}

We need the following theta series identities which appear to be new.
\begin{lemma}\label{lem-W-product}
Let 
\begin{align}\label{Wk-defn}
    W_k(q):=\sum_{n\in\mathbb{Z}}(-1)^n(9n+k)q^{9n^2+2kn}.
\end{align}
We have
\begin{align}
    W_1(q)&=\frac{J_{18}^{3}J_{7,18}J_{6,18}^{3}}{J_{9,18}J_{3,18}^{3}}
 -3q^{3}\frac{J_{18}^{3}J_{4,18}^{3}}{J_{3,18}J_{7,18}^{2}}, \label{W1-product} \\
 W_2(q)&=-\frac{J_{18}^{3}J_{5,18}J_{6,18}^{3}}{J_{9,18}J_{3,18}^{3}}
 +3\frac{J_{18}^{3}J_{8,18}^{3}}{J_{3,18}J_{5,18}^{2}}. \label{W2-product} \\
W_3(q)&= 3\frac{J_{18}^{3}J_{6,18}^{3}}{J_{3,18}^{2}J_{9,18}}
 =3\frac{J_6^5}{J_3^2}, \label{W3-product} \\
W_4(q)&=
 \frac{J_{18}^{3}J_{1,18}J_{6,18}^{3}}
      {J_{9,18}J_{3,18}^{3}}
 +3\frac{J_{18}^{3}J_{2,18}^{3}}
         {J_{3,18}J_{1,18}^{2}} \label{W4-product}
\end{align}
and 
\begin{align}
W_{k}(q)=q^{9-2k}W_{9-k}(q), \quad k\in \{5,6,7,8\}. \label{Wk-rec}
\end{align}
\end{lemma}
It follows from \eqref{Jacobi} that
\begin{align*}
 (k+9D_z)j(z;q^{18})
 =\sum_{n\in\mathbb{Z}}(-1)^n(9n+k)q^{9n(n-1)}z^n.
\end{align*}
Hence we have
\begin{align}\label{Wk-L}
    W_k(q)=(k+9\mathcal{L}(q^{9+2k};q^{18}))J_{9+2k,18}.
\end{align}
To evaluate $\mathcal{L}(q^{9+2k};q^{18})$ we need the following formula which is equivalent to an addition formula \cite[p.~64, (3.5.13)]{Lawden} for Jacobi theta functions:
\begin{align}\label{L-addition}
    \mathcal{L}(x;q^2)+\mathcal{L}(y;q^2)-\mathcal{L}(xy/q;q^2)=q^{-1}xy\frac{J_2^4}{J_1^2}\frac{j(qx;q^2)j(qy;q^2)j(xy;q^2)}{j(x;q^2)j(y;q^2)j(xy/q;q^2)}.
\end{align}

In particular, when $x=y$ we have
\begin{align}\label{eq-duplicate-L}
 2\mathcal{L}(x;q^2)-\mathcal{L}(x^2/q;q^2)
 =q^{-1}x^2\frac{J_2^4}{J_1^2}
 \frac{j(q/x;q^2)^2j(x^2;q^2)}
 {j(x;q^2)^2j(x^2/q;q^2)}.
\end{align}
\begin{proof}[Proof of Lemma \ref{lem-W-product}]
Replacing $n$ by $-n-1$ in \eqref{Wk-defn}, we obtain \eqref{Wk-rec} immediately. Hence it remains to prove \eqref{W1-product}--\eqref{W4-product}.

Replacing $q$ by $q^{9}$ and setting $x=q^{9-k}$, $X_k:=\mathcal{L}(q^{9-k};q^{18})$ in \eqref{eq-duplicate-L}, we obtain 
\begin{equation}\label{eq-X-recurrence}
 2X_k-X_{2k}=q^{9-2k}
 \frac{J_{18}^4J_{k,18}^2J_{2k,18}}
 {J_9^2J_{9-k,18}^2J_{9-2k,18}}=:P_k.
\end{equation}
Taking $k=2,4$ in \eqref{eq-X-recurrence} gives
\begin{align}\label{X248}
 2X_2-X_4=P_2, \quad  2X_4-X_8=P_4.
\end{align}
Furthermore, by \eqref{eq-shift} and \eqref{eq-inverse} we have
\begin{align*}
 X_{16}=\mathcal{L}(q^{-7};q^{18})
 =\mathcal{L}(q^{11};q^{18})+1
 =1-\mathcal{L}(q^7;q^{18})=1-X_2.
\end{align*}
Thus, the $k=8$ instance of \eqref{eq-X-recurrence} is
\begin{equation}
 2X_8+X_2=1+P_8.
 \label{eq-X8X2}
\end{equation}
Solving \eqref{X248} and \eqref{eq-X8X2}, we obtain
\begin{align}
&9X_2=1+4P_2+2P_4+P_8,
 \label{eq-X2-exp} \\
& 4-9X_8=2P_2+P_4-4P_8,
 \label{eq-X8-exp}\\
& 2-9X_4=P_2-4P_4-2P_8.
 \label{eq-X4-exp}
\end{align}

For $k=6$, we have
\begin{align*}
 X_{12}=\mathcal{L}(q^{-3};q^{18})
 =\mathcal{L}(q^{15};q^{18})+1
 =1-\mathcal{L}(q^3;q^{18})=1-X_6.
\end{align*}
Therefore, \eqref{eq-X-recurrence} gives
\begin{align}\label{eq-X6-exp}
  1-3X_6=-P_6.
\end{align}

(1) Using \eqref{eq-inverse} we have
\begin{align*}
 \mathcal{L}(q^{11};q^{18})=-\mathcal{L}(q^7;q^{18})=-X_2.
\end{align*}
Hence, from \eqref{Wk-L}  and \eqref{eq-X2-exp} we have
\begin{align*}
W_1(q)=J_{7,18}(1-9X_2)=-J_{7,18}(4P_2+2P_4+P_8).
\end{align*}
By the definition of $P_C$ in \eqref{eq-X-recurrence} we obtain
\begin{align}
 W_1(q)
=
 \frac{J_{18}^{4}J_{2,18}J_{8,18}^{2}}
 {J_{9}^2J_{1,18}^{2}}
 -2q\frac{J_{18}^{4}J_{7,18}J_{4,18}^{2}J_{8,18}}
 {J_{9}^2J_{1,18}J_{5,18}^{2}} -4q^5\frac{J_{18}^{4}J_{2,18}^{2}J_{4,18}}
 {J_{9}^2J_{5,18}J_{7,18}},
\end{align}
Using the Maple approach in \cite{Frye-Garvan}, we obtain \eqref{W1-product} from the above identity.

(2) Using \eqref{eq-inverse} we have
\begin{align*}
 \mathcal{L}(q^{13};q^{18})=-\mathcal{L}(q^5;q^{18})=-X_4.
\end{align*}
Hence, from \eqref{Wk-L}  and \eqref{eq-X4-exp} we have
\begin{align*}
 W_2(q)=J_{5,18}(2-9X_4)=J_{5,18}(P_2-4P_4-2P_8).
\end{align*}
By the definition of $P_C$ in \eqref{eq-X-recurrence} we obtain
\begin{align}
W_2(q)
=2\frac{J_{18}^{4}J_{2,18}J_{5,18}J_{8,18}^{2}}
 {J_{9}^2J_{1,18}^{2}J_{7,18}}
 -4q\frac{J_{18}^{4}J_{4,18}^{2}J_{8,18}}
 {J_{9}^2J_{1,18}J_{5,18}}
 +q^5\frac{J_{18}^{4}J_{2,18}^{2}J_{4,18}}
 {J_{9}^2J_{7,18}^{2}}.
\end{align}
Using the Maple approach in \cite{Frye-Garvan}, we obtain \eqref{W2-product} from the above identity.

(3) Using \eqref{eq-inverse} we have
\begin{align*}
 \mathcal{L}(q^{15};q^{18})=-\mathcal{L}(q^3;q^{18})=-X_6.
\end{align*}
Hence, from \eqref{Wk-L} and \eqref{eq-X6-exp} we have
\begin{align*}
W_3(q)=3J_{3,18}(1-3X_6)
 =-3J_{3,18}P_6=3\frac{J_{18}^4J_{6,18}^3}
 {J_{9}^2J_{3,18}^2}.
\end{align*}
This proves \eqref{W3-product}.

(4) Using \eqref{eq-inverse} we have
\begin{align*}
 \mathcal{L}(q^{17};q^{18})=-\mathcal{L}(q;q^{18})=-X_8.
\end{align*}
It follows from \eqref{Wk-L} and \eqref{eq-X8-exp} that
\begin{align*}
 W_4(q)=J_{1,18}(4-9X_8)=J_{1,18}(2P_2+P_4-4P_8).
\end{align*}
By the definition of $P_C$ in \eqref{eq-X-recurrence} we obtain
\begin{align}
 W_4(q)
=4\frac{J_{18}^{4}J_{2,18}J_{8,18}^{2}}
 {J_{9}^2J_{1,18}J_{7,18}}
 +q\frac{J_{18}^{4}J_{4,18}^{2}J_{8,18}}
 {J_{9}^2J_{5,18}^{2}} +2q^5\frac{J_{18}^{4}J_{1,18}J_{2,18}^{2}J_{4,18}}
 {J_{9}^2J_{5,18}J_{7,18}^{2}}.
\end{align}
Using the Maple approach in \cite{Frye-Garvan}, we obtain \eqref{W4-product} from the above identity.
\end{proof}

Now we review some transformation/summation formulas on the $q$-hypergeometric series ${}_r\phi_s$  defined by
\begin{align}
{}_r\phi_s \bigg(\genfrac{}{}{0pt}{}{a_1,  \cdots,  a_r}{b_1,   \dots,  b_s}; q,z \bigg) := ~\sum_{n=0}^\infty\frac{(a_1,\cdots,a_r;q)_n}{(q,b_1,\cdots,b_s;q)_n}((-1)^nq^{n(n-1)/2})^{1+s-r}z^n.
\end{align}
We need Heine's transformation formula \cite[Corollary 2.4]{Andrews-book}:
\begin{align}\label{Heine}
    \sum_{n=0}^\infty \frac{(a,b;q)_nz^n}{(c,q;q)_n}=\frac{(b,az;q)_{\infty}}{(c,z;q)_{\infty}}\sum_{n=0}^\infty \frac{(c/b,z;q)_nb^n}{(az,q;q)_n}.
\end{align}
Setting $z=c/(ab)$ in \eqref{Heine} and using \eqref{Euler}, we have
\begin{align}\label{Heine-cor}
    \sum_{n=0}^\infty \frac{(a;q)_n(b;q)_n}{(q;q)_n(c;q)_n}\big(\frac{c}{ab}\big)^n=\frac{({c}/{a};q)_{\infty}({c}/{b};q)_{\infty}}{(c;q)_{\infty}({c}/{ab};q)_{\infty}}.
\end{align}
Let $a\rightarrow\infty$. Replacing $q$ by $q^4$ and then setting $(b,c)$ as $(-q,q^2)$ and $(-q^3,q^{6})$ in \eqref{Heine-cor}, we obtain the following identities:
\begin{align}\label{dual-8-ex2-2}
    \sum_{n=0}^\infty \frac{(-q;q^4)_nq^{2n^2-n}}{(q^2;q^2)_{2n}}&=\frac{1}{(q;q^4)_{\infty}(q^{6};q^{8})_{\infty}},\\
    \sum_{n=0}^\infty \frac{(-q^3;q^4)_nq^{2n^2+n}}{(q^2;q^2)_{2n+1}}&=\frac{1}{(q^3;q^4)_{\infty}(q^2;q^{8})_{\infty}}. \label{dual-8-ex2-3}
\end{align}

In the proof of Nahm sum identities for the dual of Example 7, we will need the following transformation formula \cite[\uppercase\expandafter{\romannumeral 3}.9]{GR-book}:
\begin{align}
{}_3\phi_2\bigg(\genfrac{}{}{0pt}{} {a,b,c}{d,e};q,\frac{de}{abc}  \bigg)    =\frac{(e/a,de/bc;q)_\infty}{(e,de/abc;q)_\infty} {}_3\phi_2\bigg(\genfrac{}{}{0pt}{} {a,d/b,d/c}{d,de/bc};q,\frac{e}{a}  \bigg). \label{GR-III.9}
\end{align}

From \cite[Eq.~(II.8)]{GR-book} we find Heine's $q$-analog of Gauss' sum:
\begin{align}
    \sum_{n=0}^\infty \frac{(-1)^nq^{n(n-1)/2}(a;q)_n}{(q,c;q)_n}\big(\frac{c}{a}\big)^n=\frac{(c/a;q)_\infty}{(c;q)_\infty}.
\end{align}
Replacing $q$ by $q^2$ and setting $(a,c)$ as $(-q,q)$, $(-q,q^3)$ and $(-1,q)$, we obtain \cite[(2.2.3)--(2.2.4)]{MSP}:
\begin{align}
    \sum_{n=0}^\infty \frac{q^{n^2-n}(-q;q^2)_n}{(q;q)_{2n}}&=2\frac{J_4}{J_1}, \label{new} \\
      \sum_{n=0}^\infty \frac{q^{n^2+n}(-q;q^2)_n}{(q;q)_{2n+1}}&=\frac{J_4}{J_1}, \label{H-cor-id-1} \\
       \sum_{n=0}^\infty \frac{q^{n^2}(-1;q^2)_n}{(q;q)_{2n}}&=\frac{J_2^3}{J_1^2J_4}. \quad \text{(S.~47)} \label{S47} 
\end{align}
Here and throughout, we use (S.~$n$) to denote the $n$-th identity in Slater's list of Rogers--Ramanujan type identities~\cite{Slater1952}.

Recall Andrews' $q$-analogs of Gauss' ${}_2F_1(\frac{1}{2})$ sum and Bailey's ${}_2F_1(\frac{1}{2})$ sum (see \cite[(1.8) and (1.9)]{Andrews1973} or \cite[(\uppercase\expandafter{\romannumeral 2}.11) and (\uppercase\expandafter{\romannumeral 2}.10)]{GR-book}):
\begin{align}
    \sum_{n=0}^\infty \frac{(a,b;q)_nq^{n(n+1)/2}}{(q;q)_n(abq;q^2)_n}&=\frac{(aq,bq;q^2)_\infty}{(q,abq;q^2)_\infty}, \label{Andrews-Gauss}\\
    \sum_{n=0}^\infty \frac{(b,q/b;q)_nc^nq^{n(n-1)/2}}{(c;q)_n(q^2;q^2)_n}&=\frac{(bc,cq/b;q^2)_\infty}{(c;q)_\infty}. \label{Andrews-Bailey}
\end{align}
Setting $a=b=-1$ in \eqref{Andrews-Gauss}, we obtain
\begin{align}
     & \sum_{n=0}^\infty \frac{q^{(n^2+n)/2}(-1;q)_n^2(-q;q)_n}{(q;q)_{2n}}=\frac{J_2^6}{J_1^4J_4^2}.  \label{new-2} 
\end{align}
Similarly, replacing $q$ by $q^2$ and setting $(b,c)=(-q,q^3)$ in \eqref{Andrews-Bailey}, we obtain
\begin{align}
    & \sum_{n=0}^\infty \frac{q^{n^2+2n}(-q;q^2)_n^2(-q;q^2)_{n+1}}{(q^2;q^2)_{2n+1}}= \frac{J_2J_8^2}{J_1J_4^2}. \label{new-3}
\end{align}

We will also use the following single sum Rogers--Ramanujan type identities:
\begin{align}
&\sum_{n=0}^{\infty} \frac{(-1)^nq^{n(n+2)}(q;q^2)_n}{(-q;q^2)_n(q^4;q^4)_n}=\frac{J_5^2J_{3,10}}{J_{10}J_{4,10}^2}, \quad \text{(\cite[(2.17)]{BMS})}\label{BMS2.17} \\
&\sum_{n=0}^{\infty} \frac{(-1)^nq^{n(n+2)}(q;q^2)_n}{(-q;q^2)_{n+1}(q^4;q^4)_n}=\frac{J_5J_{1,10}^2J_{3,10}^2}{J_1J_{10}^4}, \quad \text{(\cite[(2.7)]{MSP})}\label{MSZ2.7} \\
&\sum_{n=0}^\infty \frac{q^{n^2}(-q^2;q^2)_n}{(q;q)_{2n+1}}=\frac{J_2^3}{J_1^2J_4}, \quad \text{(\cite[Ent.~1.7.13]{Lost2})}\label{Rama1713} \\
&\sum_{n=0}^{\infty} \frac{q^{n^2+2n}(-q;q^2)_n}{(q^4;q^4)_n}=\frac{J_6J_{12}}{J_{3,12}J_{4,12}}, \quad \text{(\cite[Ent.~4.2.11]{Lost2})}\label{Rama4211} \\
&\sum_{n=0}^{\infty} \frac{q^{n(n+1)/2}(-1;q)_n}{(q;q)_n(q;q^2)_n}=\frac{J_5J_{10}^2}{J_1J_{1,10}J_{3,10}}, \quad \text{(\cite[p. 330(4), line 3, corrected]{Rogers1917})}\label{Rogers3303} \\
 & \sum_{n=0}^\infty \frac{q^{n^2}(-q;q^2)_n^2}{(q^2;q^2)_{2n}}=\frac{J_{10}^2J_{20}^5}{J_{1,20}J_{2,20}J_{5,20}^2J_{8,20}^2J_{9,20}}, \quad \text{(S.\ 21, \cite[(2.19)]{Wang-rank3})} \label{dual-10-ex2-h0} \\
 &\sum_{n=0}^\infty \frac{q^{n(n+1)}(-q^2;q^2)_n}{(q;q)_{2n+1}}=\frac{J_3J_{12}}{J_1J_6}, \quad  \text{(S.\ 28)}\label{S28} \\
 &\sum_{n=0}^\infty \frac{q^{n^2}(-q;q^2)_n}{(q;q)_{2n}}=\frac{J_6^2}{J_1J_{12}}, \quad \text{({S.\ 29})}\label{S29} \\
 &\sum_{n=0}^{\infty} \frac{q^{n(n+3)/2}(-q;q)_n}{(q;q^2)_{n+1}(q;q)_n}=\frac{J_{10}^3}{J_1J_{5}J_{3,10}}, \quad \text{(S.\ 43)}\label{S43} \\
 &\sum_{n=0}^\infty \frac{q^{n(n+1)/2}(-q;q)_n}{(q;q)_n(q;q^2)_{n+1}}=\frac{J_2J_{3,10}}{J_1^2},   \quad \text{(S.\ 45)}\label{S45} \\
    &\sum_{n=0}^\infty \frac{q^{n^2+n}(-1;q^2)_n}{(q;q)_{2n}}=\frac{J_4J_6^5}{J_2^2J_3^2J_{12}^2}, \quad \text{(S.~48-)} \label{S48} \\
     &\sum_{n=0}^\infty \frac{q^{n^2+2n}(-q;q^2)_n}{(q;q)_{2n+1}}=\frac{J_2J_{12}^2}{J_1J_4J_6}, \quad\text{({S.\ 50})}\label{S50} \\
     &\sum_{n=0}^\infty \frac{q^{4n^2}(q;q^2)_{2n}}{(q^4;q^4)_{2n}}=\frac{J_{5,12}}{J_4}, \quad \text{(S.~53)} \label{S53} \\
     &\sum_{n=0}^\infty \frac{q^{4n(n+1)}(q;q^2)_{2n+1}}{(q^4;q^4)_{2n+1}}=\frac{J_{1,12}}{J_4}, \quad \text{(S.~55)} \label{S55}  \\
    &\mathcal{A}_0(q):=\sum_{n=0}^\infty \frac{q^{n^2}}{(q;q)_{2n}}=\frac{J_2J_{20}}{J_1J_{4,20}},   \quad \text{(S.~79)} \label{S79} \\
     &\mathcal{A}_1(q):=\sum_{n=0}^\infty \frac{q^{n(n+1)}}{(q;q)_{2n+1}}=\frac{J_{3,10}J_{4,20}}{J_1J_{20}}, \quad \text{(S.~94)} \label{S94} \\
     &\mathcal{B}_0(q):=\sum_{n=0}^\infty \frac{q^{(n+1)^2}}{(q;q)_{2n+1}}=q\frac{J_2J_{20}}{J_1J_{8,20}},  \quad \text{(S.~96)} \label{S96}  \\
     &\mathcal{B}_1(q):=\sum_{n=0}^\infty \frac{q^{n(n+1)}}{(q;q)_{2n}}=\frac{J_{1,10}J_{8,20}}{J_1J_{20}}. \quad \text{(S.~99)} \label{S99}
\end{align}
The functions $\mathcal{A}_i(q)$ and $\mathcal{B}_i(q)$ ($i=0,1$) play important roles in studying Nahm sums dual to Examples 9 and 12 (see Sections \ref{sec-double}, \ref{sec-exam9} and \ref{sec-exam12}).

In our proofs of the dual of Examples 9 and 11 as well as Theorem \ref{thm-CW-conj3.3}, we employ the constant-term method, which has been used extensively to establish Nahm sum identities; see, for example, \cite{MW24,Shi-Wang,Wang-rank2,Wang-rank3,CRW}. Roughly speaking, the method is to express Nahm sums as constant term, i.e., the coefficient of $z^0$ of certain $z$-parametrized series, and then transform and expand the series to some different forms so that the constant term can be evaluated explicitly. Sometimes we need to introduce more variables rather than the single variable $z$. In this paper one variable will be sufficient, and we define the constant-term extraction operator by
\begin{align*}
    \mathrm{CT}_z \big[ \sum_{n\in \mathbb{Z}}a_nz^n\big]:=a_0.
\end{align*}

\subsection{Bailey pairs}\label{sec-Bailey}
Some of our proofs utilize the Bailey pair method. If for all $n\geq 0$, 
\begin{align}\label{defn-BP}
     \beta_n(a;q)=\sum_{k=0}^n\frac{\alpha_k(a;q)}{(q;q)_{n-k}(aq;q)_{n+k}},
 \end{align}
 then we call $(\alpha_n(a;q),\beta_n(a;q))$ a Bailey pair relative to $a$. We will always assume that $\alpha_0(a;q)=\beta_0(a;q)=1$ unless otherwise stated (see \eqref{new23-bp2} and \eqref{dual-9-bp6}).

Suppose that $(\alpha_n(a;q),\beta_n(a;q))$ is a Bailey pair relative to $a$. Bailey's lemma (see e.g.,~\cite[Lemma 1.1]{Warnaar}) implies that $(\alpha_n'(a;q),\beta_n'(a;q))$ is also a Bailey pair relative to $a$ where
\begin{equation}\label{alpha-beta}
\begin{split}
\alpha_n'(a;q)&:=\frac{(\rho_1,\rho_2;q)_n(aq/\rho_1\rho_2)^n}{(aq/\rho_1,aq/\rho_2;q)_n}\alpha_n(a;q), \\ \beta_n'(a;q)&:=\sum_{r=0}^n\frac{(\rho_1,\rho_2;q)_r(aq/\rho_1\rho_2;q)_{n-r}(aq/\rho_1\rho_2)^r}{(aq/\rho_1,aq/\rho_2;q)_n(q;q)_{n-r}}\beta_r(a;q).
\end{split}
\end{equation}
If we take $\rho_1,\rho_2\rightarrow\infty$, then we obtain a Bailey pair relative to $a$ (see e.g.,~\cite[(S1)]{Bressoud2000}):
\begin{align}\label{bailey-s1-finite}
    \begin{cases}
        \alpha'_n(a;q)=a^nq^{n^2}\alpha_n(a;q), \\
        \beta'_n(a;q)=\sum_{r=0}^n\frac{a^rq^{r^2}}{(q;q)_{n-r}}\beta_r(a;q). 
        \end{cases}
\end{align}
Substituting \eqref{bailey-s1-finite} into \eqref{defn-BP} and taking $n\rightarrow \infty$, we obtain
\begin{align}\label{bailey-s1}
    \sum_{n=0}^\infty a^nq^{n^2}\beta_{n}(a;q)&=\frac{1}{(aq;q)_{\infty}}\sum_{n=0}^\infty a^nq^{n^2}\alpha_n(a;q).
\end{align}
Substituting \eqref{alpha-beta} into \eqref{defn-BP}, and taking $n,\rho_1, \rho_2\rightarrow\infty$ and $\rho_2\in \{-a^{\frac{1}{2}}q^{\frac{1}{2}},-a^{\frac{1}{2}}\}$, we have (see e.g.~\cite[(S2) and (S6)]{Bressoud2000})
\begin{align}
    \sum_{n=0}^\infty a^{\frac{n}{2}}q^{\frac{n^2}{2}}(-a^{\frac{1}{2}}q^{\frac{1}{2}};q)_n\beta_n(a;q)&=\frac{(-a^{\frac{1}{2}}q^{\frac{1}{2}};q)_{\infty}}{(aq;q)_{\infty}}\sum_{n=0}^\infty a^{\frac{n}{2}}q^{\frac{n^2}{2}}\alpha_n(a;q), \label{bailey-s2}\\
    \sum_{n=0}^\infty (-a^{\frac{1}{2}};q)_na^{\frac{n}{2}}q^{\frac{n^2+n}{2}}\beta_n(a;q)&=\frac{(-a^{\frac{1}{2}}q;q)_{\infty}}{(aq;q)_{\infty}}\sum_{n=0}^\infty \frac{(-a^{\frac{1}{2}};q)_n}{(-a^{\frac{1}{2}}q;q)_n}a^{\frac{n}{2}}q^{\frac{n^2+n}{2}}\alpha_n(a;q).\label{bailey-s6}
\end{align}
Substituting \eqref{bailey-s1-finite} into \eqref{bailey-s1}, we obtain the following transformation formula:
\begin{align}\label{dual-9-bailey's lemma}
    \sum_{n=0}^\infty a^nq^{n^2}\sum_{r=0}^n\frac{a^rq^{r^2}}{(q;q)_{n-r}}\beta_r(a;q)=\frac{1}{(aq;q)_{\infty}}\sum_{n=0}^\infty a^{2n}q^{2n^2}\alpha_n(a;q).
\end{align}

Occasionally, we need to shift the parameter of a Bailey pair. The limiting case $b\rightarrow \infty$  of~\cite[Lemma~2.3]{Lovejoy} provides the following construction of a Bailey pair relative to $aq$ from one relative to $a$ (see also~\cite[(2.14)]{WW-2025AIM}):
\begin{equation}\label{Bailey-atoaq}
\alpha_n'(aq;q)=\frac{1-aq^{2n+1}}{1-aq}q^{-n}\sum_{r=0}^n \alpha_r(a;q), \quad \beta_n'(aq;q)=q^{-n}\beta_n(a;q).
\end{equation}

As an example to illustrate the power of the Bailey pair method, here we give another proof of \eqref{new-2} and \eqref{new-3}. Setting $x=-1$ in \cite[(4.4)--(4.5)]{Andrews2016}, we obtain the Bailey pair
\begin{equation}\label{new23-bp1}
\begin{cases}
 \alpha_n(1;q)=q^{\frac{n(n-1)}{2}}+q^{\frac{n(n+1)}{2}}, \\
 \beta_n(1;q)=\frac{(-1;q)_n(-q;q)_n}{(q;q)_{2n}}.
\end{cases}
\end{equation}
Setting $x=-q^{-1/2}$ in \cite[(4.11)--(4.12)]{Andrews2016}, we obtain the Bailey pair (here $\alpha_0(1;q)=1+q^{\frac{1}{2}}$)
\begin{equation}
\label{new23-bp2}
\begin{cases}
\alpha_n(q;q)=q^{\frac{n^2}{2}}+q^{\frac{(n+1)^2}{2}}, \\
\beta_n(q;q)=\frac{(-q^{\frac{1}{2}};q)_n(-q^{\frac{1}{2}};q)_{n+1}}{(q^2;q)_{2n}}.
\end{cases}
\end{equation}
Substituting \eqref{new23-bp1} and \eqref{new23-bp2} into \eqref{bailey-s6},   we obtain \eqref{new-2} and \eqref{new-3}, respectively.

In the proof of Theorem \ref{thm-ex9-new} we need the following Bailey pair \cite[(2.16)]{WW-2025AIM}:
\begin{align}
    \begin{cases}\label{dual-9-bp1}
 \alpha_n(q;q)=\frac{1-q^{2n+1}}{1-q}(-1)^nq^{\frac{n^2-3n}{4}}, \\
        \beta_n(q;q)=\frac{(-1)^nq^{\frac{1}{2}n^2-n}}{(-q^{\frac{1}{2}};q)_n(q^2;q^2)_n}.
    \end{cases}
    \end{align}
We also list some Bailey pairs in Slater's Groups C and G \cite[p.~469]{Slater}:
\begin{align}
    &\begin{cases}\label{dual-9-bp5}
     \alpha_{2n}(1;q)=(-1)^n(q^{n^2-n}+q^{n^2+n}),\ \alpha_{2n+1}(1;q)=0,  \\
        \beta_n(1;q)=\frac{q^{\frac{n^2-n}{2}}}{(q;q)_n(q;q^2)_n}; 
    \end{cases} \quad \text{(\cite[C(5)]{Slater})} \\
     &\begin{cases}\label{dual-9-bp4}
        \alpha_{2n}(q;q)=(-1)^nq^{n^2-n},\ \alpha_{2n+1}(q;q)=(-1)^{n+1}q^{n^2+3n+2},  \\
        \beta_n(q;q)=\frac{q^{\frac{1}{2}(n^2-n)}}{(q;q)_{n}(q^3;q^2)_n};
    \end{cases} \quad \text{(\cite[C(6)]{Slater})}\\
    &\begin{cases}\label{dual-9-bp8}
        \alpha_{2n}(1;q)=-\alpha_{2n+1}(1;q)=(-1)^nq^{n^2+n}, \\
        \beta_n(q;q)=\frac{q^{\frac{n^2+n}{2}}}{(q;q)_{n}(q^3;q^2)_n};
    \end{cases} \text{\cite[C(7)]{Slater})} \\
&\begin{cases}\label{dual-9-bp3}
         \alpha_n(1;q)=(-1)^n(q^{\frac{n^2-3n}{4}}+q^{\frac{n^2+3n}{4}}),  \\
        \beta_n(1;q)=\frac{(-1)^nq^{\frac{1}{2}n^2-n}}{(-q^{\frac{1}{2}};q)_n(q^2;q^2)_{n}}; 
    \end{cases}   \\
 &\begin{cases}\label{dual-9-bp7}
        \alpha_{n}(1;q)=(-1)^nq^{\frac{n(n-1)}{4}}(1+q^n), \\
        \beta_n(1;q)=\frac{(-1)^nq^{\frac{1}{2}n^2}}{(-q^{\frac{1}{2}};q)_n(q^2;q^2)_{n}}. 
    \end{cases}   \text{(\cite[G(4)]{Slater})}
\end{align}
Here \eqref{dual-9-bp3} is the Bailey pair without label above (G4) in \cite{Slater}. 
Replacing $q^{1/2}$ by $-q^{1/2}$ in the Bailey pair in \cite[G(5)]{Slater}, we obtain the Bailey pair: for $n\geq 0$,
\begin{align}
        &\begin{cases}\label{dual-9-bp6}
        \alpha_{2n}(q;q)=\frac{(-1)^n(q^{n^2-\frac{1}{2}n}+q^{n^2+\frac{3}{2}n+\frac{1}{2}})}{1-q},\\
        \alpha_{2n+1}(q;q)=\frac{(-1)^{n+1}(q^{n^2+\frac{1}{2}n}+q^{n^2+\frac{5}{2}n+\frac{3}{2}})}{1-q}, \\
        \beta_n(q;q)=\frac{q^{\frac{1}{2}n^2}(-q^{\frac{1}{2}};q)_{n+1}}{(q;q)_{2n+1}}.
    \end{cases}
\end{align}
Substituting the Bailey pair \eqref{dual-9-bp4} into \eqref{Bailey-atoaq}, we obtain the Bailey pair:
\begin{align}
     &\begin{cases}\label{dual-9-bp2}
        \alpha_{2n}(q^2;q)=\frac{(-1)^nq^{n^2-n}+(-1)^{n+1}q^{n^2+3n+2}}{1-q^2}, \\
        \alpha_{2n+1}(q^2;q)=\frac{(-1)^n}{1-q^2}(q^{n^2-n-1}-q^{n^2+n+1}-q^{n^2+3n+3}+q^{n^2+5n+5}), \\
        \beta_n(q^2;q)=\frac{q^{\frac{1}{2}(n^2-3n)}}{(q;q)_n(q^3;q^2)_n}.
    \end{cases}
\end{align}

\subsection{Some new double sum identities}\label{sec-double}
We are going to establish some double sum identities which will be the key to evaluate the Nahm sums dual to Example 9.

We define
\begin{align}
 \Phi(z)&:=\sum_{r\geq0}\frac{q^{r^2}z^r}{(q^4;q^4)_r},
 \label{eq-Phi} \\
 \mathcal{X}_n&:=\sum_{r\geq0}\frac{q^{(r-2n)^2}}{(q^4;q^4)_r}=q^{4n^2}\Phi(q^{-4n}), \label{eq-Xn}\\
 \mathcal{Y}_n&:=\sum_{r\geq0}
 \frac{q^{(r-2n+1)^2}}{(q^4;q^4)_r}=q^{(2n-1)^2}\Phi(q^{2-4n}).
 \label{eq-Yn}
\end{align}

\begin{lemma}\label{lem-XY-rec}
For $n\geq1$ we have
\begin{align}
 \mathcal{X}_n&=q^{8n-4}\mathcal{X}_{n-1}+q\mathcal{Y}_n,
 \label{eq-rec-X-1st}\\
 \mathcal{Y}_n&=q^{8n-8}\mathcal{Y}_{n-1}+\mathcal{X}_{n-1}.
 \label{eq-rec-Y-1st}
\end{align}
Consequently, for $n\geq2$,
\begin{align}
 \mathcal{X}_n&=(1+q^{8n-8}+q^{8n-4})\mathcal{X}_{n-1}
       -q^{16n-20}\mathcal{X}_{n-2},
 \label{eq-X-rec}\\
 \mathcal{Y}_n&=(1+q^{8n-12}+q^{8n-8})\mathcal{Y}_{n-1}
       -q^{16n-28}\mathcal{Y}_{n-2}.
 \label{eq-Y-rec}
\end{align}
\end{lemma}

\begin{proof}
Let $c_r=q^{r^2}/(q^4;q^4)_r$.  For $r\geq1$,
\begin{align*}
 (1-q^{4r})c_r=q^{2r-1}c_{r-1}.
\end{align*}
Multiplication by $z^r$ and summation over $r\geq1$ gives
\begin{align}\label{eq-Phi-rec}
 \Phi(z)-\Phi(q^4z)
 =\sum_{r\geq1}q^{2r-1}c_{r-1}z^r=qz\sum_{s\geq0}q^{2s}c_sz^s
 =qz\Phi(q^2z).
\end{align}
Setting $z=q^{-4n}$ in \eqref{eq-Phi-rec} and multiplying by
$q^{4n^2}$, we obtain \eqref{eq-rec-X-1st}.  Setting $z=q^{2-4n}$ and multiplying by $q^{4n(n-1)}$, we obtain \eqref{eq-rec-Y-1st}.

From \eqref{eq-rec-X-1st} and \eqref{eq-rec-Y-1st} we deduce that
\begin{align*}
& \mathcal{X}_n-q^{8n-4}\mathcal{X}_{n-1}=\mathcal{Y}_n =q^{8n-8}\mathcal{Y}_{n-1}+\mathcal{X}_{n-1}\nonumber \\
&=q^{8n-8}(\mathcal{X}_{n-1}-q^{8n-12}\mathcal{X}_{n-2})+\mathcal{X}_{n-1}, \\
&\mathcal{Y}_{n-1}-q^{8n-16}\mathcal{Y}_{n-2}=\mathcal{X}_{n-2}=q^{12-8n}(\mathcal{X}_{n-1}-\mathcal{Y}_{n-1}) \nonumber \\
&=q^{12-8n}(\mathcal{Y}_n-q^{8n-8}\mathcal{Y}_{n-1})-q^{12-8n}\mathcal{Y}_{n-1}.
\end{align*}
This proves \eqref{eq-X-rec} and \eqref{eq-Y-rec}, respectively.
\end{proof}

The initial values are well-known from \eqref{S79}--\eqref{S99}:
\begin{align}
 \mathcal A&=\mathcal{X}_0=\sum_{m\geq0}\frac{q^{4m^2}}{(q^4;q^4)_{2m}}+
 q\sum_{m\geq0}\frac{q^{4m(m+1)}}{(q^4;q^4)_{2m+1}}=\mathcal{A}_0(q^4)+q\mathcal{A}_1(q^4),
 \label{eq-A-product}\\
 \mathcal B&=\mathcal{Y}_0=q\sum_{m\geq0}\frac{q^{4m(m+1)}}{(q^4;q^4)_{2m}}+ q^4\sum_{m\geq0}\frac{q^{4m^2+8m}}{(q^4;q^4)_{2m+1}}=\mathcal{B}_0(q^4)+q\mathcal{B}_1(q^4).
 \label{eq-B-product}
\end{align}

We are now going to express $\mathcal{X}_n$ (resp.~$\mathcal{Y}_n$) as a combination of two solutions to \eqref{eq-X-rec} (resp.~\eqref{eq-Y-rec}). This idea has been applied in \cite{GIS,GP} to establish generalization of some Rogers--Ramanujan type identities. To construct such solutions, we define for $a=0,1$ and $n\geq 0$ that
\begin{align}
    U_n^{(a)}=U_n^{(a)}(q):=\sum_{j\geq 0} q^{j(j+a)}\qbinom{n-j}{j}_{q}.
\end{align}
Then $U_0^{(0)}=U_1^{(0)}=1$, $U_0^{(1)}=U_1^{(1)}=1$ and for $n\geq 2$ we have (see e.g.,~\cite[(3.14) and (3.18)]{Sills-finite})
\begin{align}
  U_n^{(0)}=U_{n-1}^{(0)}+q^{n-1}U_{n-2}^{(0)}, \quad
U_n^{(1)}=U_{n-1}^{(1)}+q^nU_{n-2}^{(1)}. \label{Un-rec}
\end{align}
We then extend the definition of $U_n^{(a)}$ to negative $n$ by the recurrence relation \eqref{Un-rec}. For instance, we should have
\begin{align}
    U_{-1}^{(0)}=0, \quad U_{-2}^{(1)}=q, \quad U_{-1}^{(1)}=0, \quad U_{-2}^{(1)}=1.
\end{align}

\begin{lemma}\label{lem-XY-rep}
We have for $n\geq 0$ that
\begin{align}
 \mathcal{X}_n&=\mathcal AU_{2n}^{(0)}(q^4)+\mathcal BU_{2n-1}^{(1)}(q^4),
 \label{eq-X-split}\\
 \mathcal{Y}_n&=\mathcal A U_{2n-1}^{(0)}(q^4)+\mathcal B U_{2n-2}^{(1)}(q^4).
 \label{eq-Y-split}
\end{align}
\end{lemma}
\begin{proof}
Using \eqref{Un-rec} it is straightforward to prove that $U_{2n}^{(0)}(q^4)$ and $U_{2n-1}^{(1)}(q^4)$ satisfy the recurrence relation \eqref{eq-X-rec}, and $U_{2n-1}^{(0)}(q^4)$ and $U_{2n-2}^{(1)}(q^4)$ satisfy the recurrence relation \eqref{eq-Y-rec}. Hence the sequences on both sides of \eqref{eq-X-split} and \eqref{eq-Y-split} satisfy the same recurrence relations. 

Taking $n=1$ in \eqref{eq-rec-X-1st}--\eqref{eq-rec-Y-1st} gives
\begin{align*}
 \mathcal{Y}_1=\mathcal{Y}_0+\mathcal{X}_0,\qquad
 \mathcal{X}_1=(1+q^4)\mathcal{X}_0+\mathcal{Y}_0.
\end{align*}
It is easy to check that the right side of \eqref{eq-X-split} and \eqref{eq-Y-split} also share the same initial values at $n=0,1$. Hence they are the same with $\mathcal{X}_n$ and $\mathcal{Y}_n$, respectively.    
\end{proof}
Lemma \ref{lem-XY-rep} is one of the key steps in our evaluation of Nahm sums dual to Example 9. A related result can be found in \cite[Theorem 2.2]{GP}.

For $a\in\{0,1\}$ and $N\geq0$, recall the following  identity (see e.g.,~\cite[(3.14) and (3.18)]{Sills-finite}):
\begin{equation}
 \sum_{j\geq0}q^{j(j+a)}\qbinom{N-j-a}{j}_{q}
 =\sum_{j\in\mathbb Z}(-1)^jq^{j(5j+2a+1)/2}
 \qbinom{N}{\left\lfloor(N-5j-a)/2\right\rfloor}_{q}.
 \label{eq-finite-Schur}
\end{equation}
 
Splitting the sum in the right side according to $j=2m$ and $j=2m+1$, we obtain
\begin{align}
 U_{2n}^{(0)}(q)={}&\sum_{m\in\mathbb Z}\Big(
 q^{10m^2+m}\qbinom{2n}{n-5m}_{q}
 -q^{10m^2+11m+3}\qbinom{2n}{n-5m-3}_{q}
 \Big),
 \label{eq-U-even-0}\\
 U_{2n-1}^{(1)}(q)={}&\sum_{m\in\mathbb Z}\Big(
 q^{10m^2+3m}\qbinom{2n}{n-5m-1}_{q}
 -q^{10m^2+13m+4}\qbinom{2n}{n-5m-3}_{q}
 \Big),
 \label{eq-U-odd-1}\\
U_{2n-2}^{(1)}(q)={}&\sum_{m\in\mathbb Z} \Big(
 q^{10m^2+3m}\qbinom{2n-1}{n-5m-1}_{q}
 -q^{10m^2+13m+4}\qbinom{2n-1}{n-5m-4}_{q}
 \Big),
 \label{eq-U-even-1}\\
 U_{2n-1}^{(0)}(q)={}&\sum_{m\in\mathbb Z} \Big(
 q^{10m^2+m}\qbinom{2n-1}{n-5m-1}_{q}
 -q^{10m^2+11m+3}\qbinom{2n-1}{n-5m-3}_{q}
 \Big).
 \label{eq-U-odd-0}
\end{align}

To evaluate the Nahm sums dual to Example 9, we need to deal with the following series:
\begin{align}
 I_s^{(a)}(q)&:=\sum_{n\geq1}
 \frac{(-1)^nq^{2n^2+4a n}(q^2;q^4)_n}{(q^4;q^4)_{2n}}
 \qbinom{2n-1}{n-1-s}_{q^4}, \quad a=0,1,
 \label{eq-I01-defn}\\
 I_s^{(2)}(q)&:=\sum_{n\geq0}
 \frac{(-1)^{n+1}q^{2n^2+4n}(q^2;q^4)_{n+1}}{(q^4;q^4)_{2n+1}}
 \qbinom{2n}{n-s}_{q^4}.
 \label{eq-I2-defn}
\end{align}
We will connect them with the following series:
\begin{align}\label{eq-I-defn}
    I_s(q):=\sum_{n\geq0}
 \frac{(-1)^nq^{2n^2}(q^2;q^4)_n}{(q^4;q^4)_{2n}}
 \qbinom{2n}{n-s}_{q^4}.
\end{align}

\begin{lemma}\label{lem-I-relation}
We have
\begin{align}
 I_s(q)=(-1)^sq^{2s^2}\frac{J_2}{J_4^2}.
 \label{eq-I-product}
\end{align}
Let $\widetilde{I}_s=I_s-\delta_{s,0}$ where $\delta_{s,0}$ is 1 when $s=0$ and 0 otherwise.  Then
\begin{align}
 I_{s-1}^{(0)}-q^{8s+4}I_{s+1}^{(0)}
 &=\widetilde{I}_s-q^{8s+4}\widetilde{I}_{s+1},
 \label{eq-I0-relation}\\
 I_{s-1}^{(1)}-q^{8s+4}I_{s+1}^{(1)}
 &=q^{4s}(\widetilde{I}_s-\widetilde{I}_{s+1}),
 \label{eq-I1-relation}\\
 I_{s-1}^{(2)}-q^{8s}I_{s+1}^{(2)}
 &=q^{-2}(1-q^{8s})I_s.
 \label{eq-I2-relation}
\end{align}
\end{lemma}
\begin{proof}
Since $I_s(q)=I_{-s}(q)$, we may assume that $s\geq 0$. 
We write $n=s+r$ to deduce that
\begin{align*}
I_s(q)=\frac{(-1)^sq^{2s^2}(q^2;q^4)_s}{(q^4;q^4)_{2s}}
 \sum_{r\geq0}
 \frac{(q^{4s+2};q^4)_r}{(q^4;q^4)_r(q^{8s+4};q^4)_r}
 (-1)^rq^{2r^2+4sr}.
\end{align*}
We use the limiting $q$-Gauss summation
\begin{equation}
 \sum_{r\geq0}\frac{(a;q)_r}{(q;q)_r(c;q)_r}
 (-1)^rq^{\binom r2}\left(\frac ca\right)^r
 =\frac{(c/a;q)_\infty}{(c;q)_\infty}.
 \label{eq-lim-q-Gauss}
\end{equation}
Applying \eqref{eq-lim-q-Gauss} with $q$ replaced by $q^4$ and $a=q^{4s+2}$, $c=q^{8s+4}$, we obtain \eqref{eq-I-product}.

Recall the following identities (see e.g., \cite[(1.5) and (1.6)]{Sills-finite}):
\begin{align*}
 \qbinom{2n}{n-s}_{q^4}
 &=\qbinom{2n-1}{n-s}_{q^4}
   +q^{4n+4s}\qbinom{2n-1}{n-s-1}_{q^4},\\
 \qbinom{2n}{n-s}_{q^4}
 &=q^{4n-4s}\qbinom{2n-1}{n-s}_{q^4}
   +\qbinom{2n-1}{n-s-1}_{q^4}.
\end{align*}
From them we deduce that
\begin{align*}
 I_{s-1}^{(0)}+q^{4s}I_s^{(1)}=\widetilde{I}_s,
 \qquad
 I_s^{(0)}+q^{-4s}I_{s-1}^{(1)}=\widetilde{I}_s.
\end{align*}
Eliminating $I_s^{(1)}$ and $I_s^{(0)}$, we obtain  \eqref{eq-I0-relation} and \eqref{eq-I1-relation}, respectively.

Note that
\begin{align*}
 \qbinom{2n}{n-s+1}_{q^4}-q^{8s}\qbinom{2n}{n-s-1}_{q^4}
 =\frac{(1-q^{8s})(q^4;q^4)_{2n+1}}
 {(q^4;q^4)_{n-s+1}(q^4;q^4)_{n+s+1}}.
\end{align*}
Using it with $r=n+1$ we have
\begin{align*}
I_{s-1}^{(2)}-q^{8s}I_{s+1}^{(2)}=
 q^{-2}(1-q^{8s})\sum_{r\geq0}
 \frac{(-1)^rq^{2r^2}(q^2;q^4)_r}
 {(q^4;q^4)_{r-s}(q^4;q^4)_{r+s}}
 =q^{-2}(1-q^{8s})I_s^{(2)}.  
\end{align*}
This proves \eqref{eq-I2-relation}.
\end{proof}

Now we are ready to establish the following double sum identities.
\begin{lemma}\label{lem-PR-product}
We have
\begin{align}
 \mathcal P_1(q)&:=\sum_{n\geq0}
 \frac{(-1)^nq^{2n^2}(q^2;q^4)_n}{(q^4;q^4)_{2n}}U_{2n}^{(0)}(q^4)=\frac{J_2}{J_4^2}
 (J_{86,180}-q^{16}J_{14,180}),
 \label{eq-P1-product}\\
\mathcal R_1(q)&:=\sum_{n\geq0}
 \frac{(-1)^nq^{2n^2}(q^2;q^4)_n}{(q^4;q^4)_{2n}}U_{2n-1}^{(1)}(q^4)=-q^2\frac{J_2}{J_4^2}
 (J_{58,180}+q^{10}J_{22,180}),
 \label{eq-R1-product}\\
 {\mathcal P}_{2}(q)&:=\sum_{n\geq0}
 \frac{(-1)^nq^{2n^2}(q^2;q^4)_n}{(q^4;q^4)_{2n}}
U_{2n-2}^{(1)}(q^4)=\frac{J_2}{J_4^2}
 (J_{78,180}+q^6J_{42,180}),
 \label{eq-P2-product}\\
 {\mathcal R}_{2}(q)&:=\sum_{n\geq0}
 \frac{(-1)^nq^{2n^2}(q^2;q^4)_n}{(q^4;q^4)_{2n}}U_{2n-1}^{(0)}(q^4)=-\frac{J_2}{J_4^2}
 (q^2J_{66,180}+q^{20}J_{6,180}),
 \label{eq-R2-product}\\
 \mathcal P_3(q)&:=\sum_{n\geq0}
 \frac{(-1)^{n+1}q^{2n^2+4n}(q^2;q^4)_{n+1}}{(q^4;q^4)_{2n+1}}U_{2n}^{(0)}(q^4)=\frac{J_2}{J_4^2}
 (-J_{74,180}+q^8J_{34,180}),
 \label{eq-P3-product}\\
 \mathcal R_3(q)&:=\sum_{n\geq0}
 \frac{(-1)^{n+1}q^{2n^2+4n}(q^2;q^4)_{n+1}}{(q^4;q^4)_{2n+1}}U_{2n-1}^{(1)}(q^4)=\frac{J_2}{J_4^2}
 (q^6J_{38,180}+q^{20}J_{2,180}),
 \label{eq-R3-product} \\
 {\mathcal P}_{4}(q)&:=\sum_{n\geq0}
 \frac{(-1)^nq^{2n^2+4n}(q^2;q^4)_n}{(q^4;q^4)_{2n}}
U_{2n-2}^{(1)}(q^4)=\frac{J_2}{J_4^2}
 (J_{82,180}+q^2J_{62,180}),
 \label{eq-P4-product}\\
  {\mathcal R}_{4}(q)&:=\sum_{n\geq0}
 \frac{(-1)^nq^{2n^2+4n}(q^2;q^4)_n}{(q^4;q^4)_{2n}}U_{2n-1}^{(0)}(q^4)=-\frac{J_2}{J_4^2}
 (q^6J_{46,180}+q^{12}J_{26,180}).
 \label{eq-R4-product}
\end{align}
\end{lemma}

\begin{proof}
(1) Substituting \eqref{eq-U-even-0} into $\mathcal P_1$, and using \eqref{eq-I-product}, we deduce that
\begin{align*}
 &\mathcal P_1(q)=\sum_{m\in \mathbb{Z}} \big(q^{40m^2+4m}I_{5m}(q)-q^{40m^2+44m+12}I_{5m+3}(q)\big) \nonumber \\
 &=\frac{J_2}{J_4^2}
 \sum_m(-1)^m
 \left(q^{90m^2+4m}+q^{90m^2+104m+30}\right).
\end{align*}
This proves \eqref{eq-P1-product}.  Similarly, using
\eqref{eq-U-odd-1} and \eqref{eq-I-product} we have
\begin{align*}
 &\mathcal R_1(q)=\sum_{m\in \mathbb{Z}}\big(q^{40m^2+12m}I_{5m+1}(q)-q^{40m^2+52m+16}I_{5m+3}(q)  \big) \nonumber \\
 &=\frac{J_2}{J_4^2}
 \sum_m(-1)^m
 \left(-q^{90m^2+32m+2}+q^{90m^2+112m+34}\right).
\end{align*}
This proves \eqref{eq-R1-product}.

(2) Using \eqref{eq-U-even-1}, \eqref{eq-I0-relation} and \eqref{eq-I-product}, we have
\begin{align*}
 &\mathcal{P}_2(q)=1+\sum_{m=-\infty}^\infty\big(q^{40m^2+12m}I_{5m}^{(0)}
 -q^{40m^2+52m+16}I_{5m+3}^{(0)}\big) \nonumber \\
 &=1-\sum_m q^{40m^2-28m+4}
 \big(I_{5m-2}^{(0)}-q^{8(5m-1)+4}I_{5m}^{(0)}\big) \nonumber \\
&= 1-\sum_m q^{40m^2-28m+4}
 (\widetilde{I}_{5m-1}-q^{40m-4}\widetilde{I}_{5m})  \nonumber \\
&=\frac{J_2}{J_4^2}\sum_m(-1)^m
 \left(q^{90m^2+12m}+q^{90m^2-48m+6}\right).
\end{align*}
Here we replaced $m$ by $m-1$ in the second sum. 
This proves \eqref{eq-P2-product}. Next, using \eqref{eq-U-odd-0}, \eqref{eq-I0-relation}  and \eqref{eq-I-product}, we have
\begin{align*}
   & \mathcal{R}_2(q)= \sum_{m=-\infty}^\infty \big(q^{40m^2+4m}I_{5m}^{(0)}
 -q^{40m^2+44m+12}I_{5m+2}^{(0)}\big)\nonumber \\
 & =\sum_m q^{40m^2+4m}
 \big(\widetilde{I}_{5m+1}-q^{8(5m+1)+4}\widetilde{I}_{5m+2}\big) \nonumber \\
 &=\frac{J_2}{J_4^2}\sum_m (-1)^{m+1}
 \left(q^{90m^2+24m+2}+q^{90m^2+84m+20}\right).
\end{align*}
This proves \eqref{eq-R2-product}.

(3) Using \eqref{eq-U-even-0}, \eqref{eq-I2-relation}  and \eqref{eq-I-product}, we have
\begin{align*}
    &\mathcal{P}_3(q)=\sum_{m=-\infty}^\infty \big(q^{40m^2+4m}I_{5m}^{(2)}
 -q^{40m^2+44m+12}I_{5m+3}^{(2)}\big) \nonumber \\
 &=-\sum_m q^{40m^2-36m+8}(I_{5m-2}^{(2)}-q^{8(5m-1)}I_{5m}^{(2)})\nonumber \\
&=\frac{J_2}{J_4^2} \sum_m(-1)^m
 \left(-q^{90m^2-16m}+q^{90m^2-56m+8}\right).
\end{align*}
Here for the second equality we replaced $m$ by $m-1$ in the second sum. This proves \eqref{eq-P3-product}. Next, using \eqref{eq-U-odd-1}, \eqref{eq-I2-relation}  and \eqref{eq-I-product}, we have
\begin{align*}
&\mathcal{R}_3(q)=\sum_{m=-\infty}^\infty\big(q^{40m^2+12m}I_{5m+1}^{(2)}
 -q^{40m^2+52m+16}I_{5m+3}^{(2)}\big) \nonumber \\
&= \sum_m q^{40m^2+12m}(I_{5m+1}^{(2)}-q^{8(5m+2)}I_{5m+3}^{(2)})\nonumber \\
&=\frac{J_2}{J_4^2} \sum_m(-1)^m
 \left(q^{90m^2+52m+6}-q^{90m^2+92m+22}\right).
\end{align*}
This proves \eqref{eq-R3-product}.

(4) Using \eqref{eq-U-even-1}, \eqref{eq-I1-relation}  and \eqref{eq-I-product}, we have
\begin{align*}
&\mathcal{P}_4(q)=1+\sum_{m=-\infty}^\infty \big(q^{40m^2+12m}I_{5m}^{(1)}
 -q^{40m^2+52m+16}I_{5m+3}^{(1)}\big) \nonumber \\
 &= 1-\sum_m q^{40m^2-8m}(\widetilde{I}_{5m-1}-\widetilde{I}_{5m}) \nonumber \\
&=\frac{J_2}{J_4^2} \sum_m(-1)^m
 \left(q^{90m^2-8m}+q^{90m^2-28m+2}\right).
\end{align*}
Here for the second equality we replaced $m$ by $m-1$ in the second sum. This proves \eqref{eq-P4-product}.
Next, using \eqref{eq-U-odd-0}, \eqref{eq-I1-relation} and \eqref{eq-I-product}, we have
\begin{align*}
&\mathcal{R}_4(q)=\sum_{m=-\infty}^\infty \big(q^{40m^2+4m}I_{5m}^{(1)}
 -q^{40m^2+44m+12}I_{5m+2}^{(1)}\big) \nonumber \\
 &=\frac{J_2}{J_4^2} \sum_m(-1)^m
 \left(-q^{90m^2+44m+6}-q^{90m^2+64m+12}\right).
\end{align*}
This proves \eqref{eq-R4-product}.
\end{proof}

\section{Nahm sums dual to Zagier's rank three examples}\label{sec-dual}
As mentioned in Section \ref{sec-intro}, we only need to consider the dual of Zagier's Examples 1--3 and 7--12.

\subsection{Dual of Examples 1--3}
We shall discuss the dual of Example 1--3 together since they share similar feature and are closely related to each other.
\subsubsection{Dual of Example 1}
We list Zagier's Example 1 and its dual in Table \ref{tab-exam1} where $\alpha,\nu \in \mathbb{Q}$ and $h\in \mathbb{Z}$. Here we have $a=-\frac{\alpha}{(h^2-2)\alpha+1}$, $v_1=\nu-h/2$ and $v_2=\nu+h/2$.

As pointed out by Wang \cite{Wang-rank3}, to make $A$ positive definite we must have $h\in \{0,-1,1\}$.  
{\small
\begin{table}[htbp]
\centering
\caption{Modular triples for Example 1 and their duals}
    \label{tab-exam1}
\begin{tabular}{|c|ccc|c|ccc|}
\hline 
\padedvphantom{I}{3.5ex}{3.5ex}
    $A$   & \multicolumn{3}{c|}{$\left(\begin{smallmatrix}  \alpha h^2+1 & \alpha h & -\alpha h \\ \alpha h & \alpha & 1-\alpha \\ -\alpha h & 1-\alpha & \alpha \end{smallmatrix}\right)$}  & $\mathcal{D}(A)$
    & \multicolumn{3}{c|}{$\left(\begin{smallmatrix} ah^2+1 & -ah & ah  \\ -ah & a & 1-a \\ ah & 1-a & a  \end{smallmatrix}\right)$} \\
\hline 
\padedvphantom{I}{3.5ex}{3.5ex}
   $B$
     & $\left(\begin{smallmatrix}\alpha \nu h \\ \alpha \nu \\ -\alpha \nu \end{smallmatrix}\right)$  & $\left(\begin{smallmatrix} \alpha \nu h +\frac{1}{2} \\ \alpha \nu  \\ -\alpha \nu \end{smallmatrix}\right)$ & $\left(\begin{smallmatrix}
\alpha \nu h-1/2 \\ \alpha \nu \\ -\alpha \nu \end{smallmatrix}\right)$
     &$\mathcal{D}(B)$
    & $\left(\begin{smallmatrix} -a\nu h\\ a\nu \\ -a\nu  \end{smallmatrix}\right)$ & $\left(\begin{smallmatrix} -av_1h+\frac{1}{2} \\ av_1 \\ -av_1 \end{smallmatrix}\right)$ & $\left(\begin{smallmatrix} -av_2h-\frac{1}{2} \\ av_2 \\ -av_2 \end{smallmatrix}\right)$  \\
     \padedvphantom{I}{1ex}{1ex}
   $C$
    & {\tiny $\frac{1}{2}\alpha \nu^2-\frac{1}{16}$} & {\tiny $\frac{1}{2}\alpha \nu^2$} & {\tiny $\frac{1}{2}\alpha\nu^2$}  & $\mathcal{D}(C)$
    & {\tiny $ \frac{1}{2}a \nu^2-\frac{1}{16}$}  & {\tiny $\frac{1}{2}av_1^2$} & {\tiny $ \frac{1}{2}av_2^2$}   \\
    \hline
\end{tabular}
\end{table}
}

As can be observed from Table \ref{tab-exam1}, this example is dual to itself. The modularity has already been confirmed by Zagier \cite{Zagier}. Moreover, if we allow $h\in \frac{1}{2}\mathbb{Z}$, then Cao and Wang \cite[Table 2 and Theorem 1.1]{Cao-Wang2026} proved that for $h=\pm \frac{1}{2}$, the matrix $A$ and the first and the third choices of $(B,C)$ in Table \ref{tab-exam1} still form modular triples.

\subsubsection{Dual of Example 2}
We list Zagier's Example 1 and its dual in Table \ref{tab-exam2} where $\alpha,\nu\in \mathbb{Q}$ and $h\in \mathbb{Z}$. Here $a=-\frac{2\alpha}{(h^2-4)\alpha+2}$, $v=\nu-\frac{1}{2}h$, and we corrected a typo in \cite[Table 3]{Zagier} and \cite[Table 2]{Wang-rank3} which wrote $17/120$ as $17/20$. As pointed out by Wang \cite{Wang-rank3}, to make $A$ positive definite we need $h\in \{0,-1,1\}$.
{\small
\begin{table}[htbp]
\centering
\caption{Modular triples for Example 2 and their duals}
    \label{tab-exam2}
\begin{tabular}{|c|cc|c|cc|}
\hline 
\padedvphantom{I}{3.5ex}{3.5ex}
    $A$   & \multicolumn{2}{c|}{$\left(\begin{smallmatrix}  \alpha h^2+2 & \alpha h & -\alpha h \\ \alpha h & \alpha & 1-\alpha \\
-\alpha h & 1-\alpha & \alpha \end{smallmatrix}\right)$}  & $\mathcal{D}(A)$
    & \multicolumn{2}{c|}{$\left(\begin{smallmatrix} \frac{1}{4}ah^2+\frac{1}{2} & -\frac{1}{2}ah & \frac{1}{2}ah  \\ -\frac{1}{2}ah & a & 1-a \\ \frac{1}{2}ah & 1-a & a  \end{smallmatrix}\right)$} \\
\hline 
\padedvphantom{I}{3.5ex}{3.5ex}
   $B$
     & $\left(\begin{smallmatrix} \alpha \nu h \\ \alpha \nu \\ -\alpha \nu  \end{smallmatrix}\right)$ & $\left(\begin{smallmatrix} \alpha \nu h +1 \\ \alpha \nu  \\ -\alpha \nu \end{smallmatrix}\right)$
     &$\mathcal{D}(B)$
    & $\left(\begin{smallmatrix} -a\nu h/2 \\ a\nu \\ -a\nu  \end{smallmatrix}\right)$ & $\left(\begin{smallmatrix} -\frac{1}{2}ahv+\frac{1}{2} \\ av \\ -av \end{smallmatrix}\right)$  \\
     \padedvphantom{I}{1ex}{1ex}
   $C$
    &  {\tiny $\frac{1}{2}\alpha\nu^2-\frac{7}{120}$} & {\tiny $\frac{1}{2}\alpha \nu^2+\frac{17}{120}$} & $\mathcal{D}(C)$
    & {\tiny $\frac{1}{2}a\nu^2-\frac{1}{15}$}  & {\tiny $\frac{1}{2}av_1^2-\frac{1}{60}$}   \\
    \hline
\end{tabular}
\end{table}
}

The dual triples $\mathcal{D}(A,B,C)$ correspond to Zagier's Example 3 (see Table \ref{tab-exam3}). Zagier proves the modularity for the case $h=0$, and the cases $h=\pm 1$ were confirmed by Cao and Wang \cite[Theorem 1.2]{Cao-Wang2026}.

\subsubsection{Dual of Example 3}
We list Zagier's Example 3 and its dual in Table \ref{tab-exam3}. Here $a=-\frac{\alpha}{2\alpha(h^2-1)+1}$ and $v=\nu-h$ and $h\in \{0,\frac{1}{2},-\frac{1}{2}\}$. Zagier \cite{Zagier} requires that $\alpha,\nu \in \mathbb{Q}$ and $h\in \mathbb{Z}$. As pointed out by Wang \cite{Wang-rank3}, in this setting we must have $h=0$ to make $A$ positive definite. However, the condition $h\in \mathbb{Z}$ can be relaxed. In fact, if we allow $h\in \frac{1}{2}\mathbb{Z}$, then Cao and Wang \cite[Table 2 and Theorem 1.1]{Cao-Wang2026} proved that when $h=\pm \frac{1}{2}$, $(A,B,C)$ in the left side of Table \ref{tab-exam3} are still modular triples. 
{\small
\begin{table}[htbp]
\centering
\caption{Modular triples for Example 3 and their duals}
    \label{tab-exam3}
\begin{tabular}{|c|cc|c|cc|}
\hline
\padedvphantom{I}{3.5ex}{3.5ex}
    $A$   & \multicolumn{2}{c|}{$\left(\begin{smallmatrix}  \alpha h^2+1/2 & \alpha h & -\alpha h \\ \alpha h & \alpha & 1-\alpha \\ -\alpha h & 1-\alpha & \alpha \end{smallmatrix}\right)$}  & $\mathcal{D}(A)$
    & \multicolumn{2}{c|}{$\left(\begin{smallmatrix} 4ah^2+2 & -2ah & 2ah  \\ -2ah & a & 1-a \\ 2ah & 1-a & a  \end{smallmatrix}\right)$} \\
\hline 
\padedvphantom{I}{3.5ex}{3.5ex}
   $B$
     & $\left(\begin{smallmatrix} \alpha \nu h \\  \alpha \nu \\ -\alpha \nu  \end{smallmatrix}\right)$ & $\left(\begin{smallmatrix}  \alpha \nu h+1/2 \\  \alpha \nu  \\ -\alpha \nu \end{smallmatrix}\right)$
     &$\mathcal{D}(B)$
    & $\left(\begin{smallmatrix} -2a\nu h \\ a \nu \\ -a\nu  \end{smallmatrix}\right)$ & $\left(\begin{smallmatrix} -2ahv+1 \\ av \\ -av \end{smallmatrix}\right)$ \\
     \padedvphantom{I}{1ex}{1ex}
   $C$
    & {\tiny $\frac{1}{2}\alpha {\nu^2}- \frac{1}{15}$} & {\tiny $ \frac{1}{2}\alpha \nu^2-\frac{1}{60}$} & $\mathcal{D}(C)$ &
    {\tiny $\frac{1}{2}a\nu^2-\frac{7}{120}$} & {\tiny $\frac{1}{2}av^2+\frac{17}{120}$} \\
    \hline
\end{tabular}
\end{table}
}

The dual triples $\mathcal{D}(A,B,C)$ coincide with Zagier's Example 2 and so their modularity are known.

\subsection{Dual of Example 7}

We list Zagier's Example 7 and its dual in Table \ref{tab-exam7}. Here $\nu $ is an arbitrary rational number.
{\small
\begin{table}[htbp]
\centering
\caption{Modular triples for Example 7 and their duals}\label{tab-exam7}
\vspace{2mm}
\begin{tabular}{c|ccccccc}
  \hline
    \padedvphantom{I}{4ex}{4ex}
  $A$ &  \multicolumn{7}{c}{$\begin{pmatrix} 2 &1 &1 \\ 1 & 2 & 0 \\ 1 & 0 & 2 \end{pmatrix}$}  \\
  \hline
\padedvphantom{I}{4ex}{4ex}
$B$ & $\begin{pmatrix} 0 \\ \nu \\ -\nu \end{pmatrix}$ & $\begin{pmatrix} 1/2 \\ 0 \\ 1 \end{pmatrix}$ &
$\begin{pmatrix} 1/2 \\ 1 \\ 0  \end{pmatrix}$ & $\begin{pmatrix} 1 \\ 0 \\ 1 \end{pmatrix}$ & $\begin{pmatrix} 1\\ 1 \\ 0  \end{pmatrix}$ &
$\begin{pmatrix} -1/2 \\ 0 \\0   \end{pmatrix}$  & $\begin{pmatrix} -1 \\ -1/2 \\ -1/2 \end{pmatrix}$ \\
  ~~ $C$ & $\nu^2/4-1/24$ & $7/48$ & $7/48$ & $5/24$ & $5/24$ & $-1/48$ & $1/48$ \\
 \hline
    \padedvphantom{I}{4ex}{4ex}
  $\mathcal{D}(A)$ &  \multicolumn{7}{c}{$\begin{pmatrix} 1 & -1/2 & -1/2 \\
-1/2 & 3/4 & 1/4 \\
-1/2 & 1/4 & 3/4 \end{pmatrix}$}  \\
  \hline
\padedvphantom{I}{4ex}{4ex}
$\mathcal{D}(B)$ &  $\begin{pmatrix} 0 \\ \nu/2 \\ -\nu/2 \end{pmatrix}$ & $\begin{pmatrix} 0 \\ 0 \\ 1/2 \end{pmatrix}$ & $\begin{pmatrix} 0 \\ 1/2 \\ 0 \end{pmatrix}$  & $\begin{pmatrix} 1/2 \\ -1/4 \\ 1/4 \end{pmatrix}$ &
$\begin{pmatrix} 1/2 \\ 1/4 \\ -1/4 \end{pmatrix}$ & $\begin{pmatrix} -1/2 \\ 1/4 \\ 1/4 \end{pmatrix}$  & $\begin{pmatrix} -1/2 \\ 0 \\ 0 \end{pmatrix}$ \\
  ~~ $\mathcal{D}(C)$ & $\nu^2/4-1/12$ & $-1/48$ & $-1/48$ & $1/24$ & $1/24$ & $1/48$ & $5/48$ \\
\hline
\end{tabular}
\end{table}
}

The modularity of Example 7 has been confirmed by Wang \cite[Theorem 4.8]{Wang-rank3}. As for its dual, we establish the following identities to prove its modularity.
\begin{theorem}\label{thm-dual-7}
We have
\begin{align}
\sum_{i,j,k\geq 0} \frac{q^{4i^2+3j^2+3k^2-4ij-4ik+2jk+2\nu j-2\nu k}}{(q^8;q^8)_i(q^8;q^8)_j(q^8;q^8)_k} &=\frac{J_8^3\overline{J}_{4(\nu+2),16}}{J_4^2J_{16}^2}+2q^{2\nu+3}\frac{J_{16}^2\overline{J}_{-4\nu,16}}{J_8^3},  \label{dual-7-1} \\
\sum_{i,j,k\geq 0} \frac{q^{4i^2+3j^2+3k^2-4ij-4ik+2jk+4k}}{(q^8;q^8)_i(q^8;q^8)_j(q^8;q^8)_k} &=\frac{J_6^5}{J_3^2J_4J_{12}^2}, \label{dual-7-2} \\
\sum_{i,j,k\geq 0} \frac{q^{4i^2+3j^2+3k^2-4ij-4ik+2jk+4i-2j+2k}}{(q^8;q^8)_i(q^8;q^8)_j(q^8;q^8)_k} &=\frac{J_{16}^7}{J_8^5J_{32}^2}+q\frac{J_8^3J_{32}^2}{J_4^2J_{16}^3},  \label{dual-7-3} \\
\sum_{i,j,k\geq 0} \frac{q^{4i^2+3j^2+3k^2-4ij-4ik+2jk-4i+2j+2k}}{(q^8;q^8)_i(q^8;q^8)_j(q^8;q^8)_k} &=2\frac{J_2^2J_3J_{12}}{J_1J_4^2J_{6}},   \label{dual-7-4}\\
\sum_{i,j,k\geq 0} \frac{q^{4i^2+3j^2+3k^2-4ij-4ik+2jk-4i}}{(q^8;q^8)_i(q^8;q^8)_j(q^8;q^8)_k} &=2q^{-1}\frac{J_2^5}{J_1^2J_4^3}.  \label{dual-7-5}
\end{align}
\end{theorem}
\begin{proof}
We define
\begin{align}
    &F(u,v,w;q):=\sum_{i,j,k\geq 0}\frac{q^{4i^2+3j^2+3k^2-4ij-4ik+2jk}u^iv^jw^k}{(q^8;q^8)_i(q^8;q^8)_j(q^8;q^8)_k} \nonumber \\
    &=\sum_{j,k\geq 0} \frac{q^{3j^2+3k^2+2jk}v^jw^k}{(q^8;q^8)_j(q^8;q^8)_k}(-q^{4-4j-4k}u;q^8)_\infty \nonumber \\
    &=\sum_{n=0}^\infty (-q^{4-4n}u;q^8)_\infty\sum_{j+k=n} \frac{q^{3n^2-4jk}v^jw^k}{(q^8;q^8)_j(q^8;q^8)_k} \nonumber \\
    &=S_0(q)+S_1(q), \label{proof-7-start}
\end{align}
where
\begin{align}
   S_0(q)&=\sum_{n=0}^\infty (-q^{4-8n}u;q^8)_\infty\sum_{j+k=2n} \frac{q^{12n^2-4jk}v^jw^k}{(q^8;q^8)_j(q^8;q^8)_k}, \\
    S_1(q)&=\sum_{n=0}^\infty (-q^{-8n}u;q^8)_\infty\sum_{j+k=2n+1} \frac{q^{3(2n+1)^2-4jk}v^jw^k}{(q^8;q^8)_j(q^8;q^8)_k}.
\end{align}

(1) Let $(u,v,w)=(1,q^\nu,-q^{-\nu})$. We have
\begin{align}
    S_0(q)&=(-q^4;q^8)_\infty \sum_{n=0}^\infty q^{4n^2}(-q^4;q^8)_n\sum_{i=-n}^n \frac{q^{4i^2-4\nu i}}{(q^8;q^8)_{n-i}(q^8;q^8)_{n+i}}, \label{proof-7-1-S0-start}\\
    S_1(q)&=q^{2\nu+3}(-1;q^8)_\infty \sum_{n=0}^\infty q^{4n^2+4n}(-q^8;q^8)_n\sum_{i=-n-1}^n \frac{q^{q^{4i^2+4(\nu+1)i}}}{(q^8;q^8)_{n-i}(q^8;q^8)_{n+i+1}}. \label{proof-7-1-S1-start}
\end{align}
Recall the following identities from \cite[(4.3) and (4.10)]{Andrews2016}:
\begin{align}
    \sum_{i=-n}^n \frac{(-1)^iq^{i(i-1)/2}x^{-i}}{(q;q)_{n-i}(q;q)_{n+i}}&=\frac{(x^{-1};q)_n(xq;q)_n}{(q;q)_{2n}}, \label{Andrews1} \\
    \sum_{i=-n-1}^n \frac{(-1)^iq^{i(i-1)/2}x^{-i}}{(q;q)_{n-i}(q;q)_{n+i+1}}&=\frac{(x^{-1};q)_n(xq;q)_{n+1}}{(q;q)_{2n+1}}. \label{Andrews2}
\end{align}
Setting $x=-q^{(\nu-1)/2}$ (resp.~ $x=-q^{-\nu/2-1}$) in \eqref{Andrews1} (resp.~\eqref{Andrews2}) and substituting the resulting identities into \eqref{proof-7-1-S0-start} and \eqref{proof-7-1-S1-start}, we deduce that
\begin{align}
S_0(q^{\frac{1}{4}})&=(-q;q^2)_\infty \sum_{n=0}^\infty \frac{q^{n^2}(-q,-q^{1+\nu},-q^{1-\nu};q^2)_n}{(q^2;q^2)_{2n}},   \label{proof-7-1-S0}\\
S_1(q^{\frac{1}{4}})&=2q^{\frac{3}{4}+\frac{1}{2}\nu}(-q^2;q^2)_\infty \sum_{n=0}^\infty \frac{q^{n^2+n}(-q^2;q^2)_n(-q^{2+\nu};q^2)_n(-q^{-\nu};q^2)_{n+1}}{(q^2;q^2)_{2n+1}}.  \label{proof-7-1-S1}
\end{align}
We have
\begin{align}
&\sum_{n=0}^\infty \frac{q^{n^2}(-q,-q^{1+\nu},-q^{1-\nu};q^2)_n}{(q^2;q^2)_{2n}}=\lim\limits_{t\rightarrow 0} \sum_{n=0}^\infty \frac{(-q/t,-q^{1+\nu},-q^{1-\nu};q^2)_nt^n}{(q,-q^2,q^2;q^2)_n} \nonumber \\
&=\lim\limits_{t\rightarrow 0} {}_3\phi_2\bigg(\genfrac{}{}{0pt}{} {-q/t,-q^{1+\nu},-q^{1-\nu}}{q,-q^2};q^2,t  \bigg) \nonumber \\
&=\lim\limits_{t\rightarrow 0} \frac{(-q^{-\nu},-tq^{1+\nu};q^2)_\infty}{(q,t;q^2)_\infty} {}_3\phi_2\bigg(\genfrac{}{}{0pt}{} {-qt,-q^{1+\nu},q^{1+\nu}}{-q^2,-tq^{1+\nu}};q^2,-q^{-\nu}  \bigg) \nonumber \\
&=\frac{(-q^{-\nu};q^2)_\infty}{(q;q^2)_\infty} {}_2\phi_1\bigg(\genfrac{}{}{0pt}{} {-q^{1+\nu},q^{1+\nu}}{-q^2};q^2,-q^{-\nu}  \bigg) \nonumber \\
&=\frac{(-q^{-\nu};q^2)_\infty}{(q;q^2)_\infty} \sum_{n=0}^\infty \frac{(q^{2\nu+2};q^4)_n(-q^{-\nu})^n}{(q^4;q^4)_n} \nonumber \\
&=\frac{(-q^{-\nu};q^2)_\infty(-q^{\nu+2};q^4)_\infty}{(q;q^2)_\infty(-q^{-\nu};q^4)_\infty}. \label{T0-result}
\end{align}
Here for the third equality we used \eqref{GR-III.9} with $q$ replaced by $q^2$ and $(a,b,c,d,e)=(-q^{1+\nu},-q^{1-\nu},-q/t,-q^2,q)$. 
Substituting \eqref{T0-result} into \eqref{proof-7-1-S0}, we obtain
\begin{align}
    S_0(q^{\frac{1}{4}})=\frac{J_2^3\overline{J}_{2+\nu,4}}{J_1^2J_4^2}. \label{proof-7-1-S0-result}
\end{align}

Similarly, we have
\begin{align}\label{add-exam7-proof-1}
&\sum_{n=0}^\infty \frac{q^{n^2+n}(-q^2;q^2)_n(-q^{2+\nu};q^2)_n(-q^{-\nu};q^2)_{n+1}}{(q^2;q^2)_{2n+1}} \nonumber \\
&=\frac{1+q^{-\nu}}{1-q^2} \sum_{n=0}^\infty \frac{q^{n^2+n}(-q^{2+\nu},-q^{2-\nu};q^2)_n}{(q^2,q^3,-q^3;q^2)_n} \nonumber \\
&=\frac{1+q^{-\nu}}{1-q^2} \lim\limits_{t\rightarrow 0} {}_3\phi_2\bigg(\genfrac{}{}{0pt}{} {-q^{2+\nu},-q^{2-\nu},-q^2/t}{q^3,-q^3};q^2,t\bigg)  \nonumber \\
&=\frac{1+q^{-\nu}}{1-q^2} \lim\limits_{t\rightarrow 0} \frac{(q^{1-\nu},-q^{2+\nu}t;q^2)_\infty}{(-q^3,t;q^2)_\infty} {}_3\phi_2\bigg(\genfrac{}{}{0pt}{} {-q^{2+\nu},-q^{1+\nu},-qt}{q^3,-q^{2+\nu}t};q^2,q^{1-\nu}\bigg)  \nonumber \\
&=\frac{1+q^{-\nu}}{1-q^2}\frac{(q^{1-\nu};q^2)_\infty}{(-q^3;q^2)_\infty}{}_2\phi_1\bigg(\genfrac{}{}{0pt}{} {-q^{2+\nu},-q^{1+\nu}}{q^3};q^2,q^{1-\nu}\bigg) \nonumber \\
&=\frac{1+q^{-\nu}}{1-q^2}\frac{(q^{1-\nu};q^2)_\infty}{(-q^3;q^2)_\infty}\sum_{n=0}^\infty \frac{(-q^{1+\nu};q)_{2n}q^{(1-\nu)n}}{(q^2;q)_{2n}}.
\end{align}
Note that
\begin{align}
  & \sum_{n=0}^\infty \frac{(-q^{1+\nu};q)_{2n}q^{(1-\nu)n}}{(q^2;q)_{2n}}=\frac{1-q}{1+q^\nu}q^{\frac{\nu-1}{2}} \sum_{n=0}^\infty \frac{(-q^\nu;q)_{2n+1}q^{(1-\nu)(2n+1)/2}}{(q;q)_{2n+1}} \nonumber \\
   &=\frac{1}{2}\frac{1-q}{1+q^\nu}q^{\frac{\nu-1}{2}}\Big( \sum_{n=0}^\infty \frac{(-q^\nu;q)_nq^{(1-\nu)n/2}}{(q;q)_n}{} \Big)-\sum_{n=0}^\infty \frac{(-q^\nu;q)_n(-q^{(1-\nu)/2})^n}{(q;q)_n}\Big)
   \nonumber \\
   &=\frac{1}{2}\frac{1-q}{1+q^\nu}q^{\frac{\nu-1}{2}}\Big( \frac{(-q^{(1+\nu)/2};q)_\infty}{(q^{(1-\nu)/2};q)_\infty}-\frac{(q^{(1+\nu)/2};q)_\infty}{(-q^{(1-\nu)/2};q)_\infty} \Big) \nonumber \\
   &=\frac{1}{2}\frac{1-q}{1+q^\nu}q^{\frac{\nu-1}{2}}\times \frac{1}{(q^{1-\nu};q^2)_\infty(q;q)_\infty}\Delta(q), \label{proof-7-1-S1-mid}
\end{align}
where
\begin{align}
     & \Delta(q)=(-q^{(1+\nu)/2},-q^{(1-\nu)/2},q;q)_\infty-(q^{(1-\nu)/2},q^{(1+\nu)/2},q;q)_\infty  
\nonumber \\
&=\sum_{n=-\infty}^\infty (1-(-1)^n) q^{\frac{\nu}{2}n+\frac{1}{2}n^2} =2q^{(\nu+1)/2}\sum_{n=-\infty}^\infty q^{2n^2+(\nu+2)n} \quad \text{(by \eqref{Jacobi})}\nonumber \\
&=2q^{(\nu+1)/2}(-q^{\nu+4},-q^{-\nu},q^4;q^4)_\infty. \label{proof-7-1-S1-final}
\end{align}
Substituting \eqref{proof-7-1-S1-final} into \eqref{proof-7-1-S1-mid}, and then substituting the result into \eqref{add-exam7-proof-1}, from \eqref{proof-7-1-S1} we deduce that \begin{align}\label{proof-7-1-S1-result}
    S_1(q^{\frac{1}{4}})=2q^{\frac{3}{4}+\frac{1}{2}\nu}\frac{(-q^2;q^2)_\infty(-q^{4+\nu},-q^{-\nu},q^4;q^4)_\infty}{(-q;q^2)_\infty(q;q)_\infty}
    =2q^{\frac{3}{4}+\frac{1}{2}\nu}\frac{J_4^2\overline{J}_{-\nu,4}}{J_2^3}.
\end{align}
Substituting \eqref{proof-7-1-S0-result} and \eqref{proof-7-1-S1-result} into \eqref{proof-7-start}, we obtain \eqref{dual-7-1}.

(2) Let $(u,v,w)=(1,1,q^4)$. We have
\begin{align} 
    S_0(q)&=(-q^4;q^8)_\infty \sum_{n=0}^\infty q^{4n^2+4n}(-q^4;q^8)_n\sum_{i=-n}^n \frac{q^{4i^2+4i}}{(q^8;q^8)_{n+i}(q^8;q^8)_{n-i}}, \label{proof-7-2-S0}\\
    S_1(q)&=2q^3(-q^8;q^8)_\infty \sum_{n=0}^\infty q^{4n^2+8n} (-q^8;q^8)_n\sum_{i=-n-1}^n \frac{q^{4i^2}}{(q^8;q^8)_{n-i}(q^8;q^8)_{n+i+1}}. \label{proof-7-2-S1}
\end{align}
Setting $x=-q^{-1}$ (resp.~$x=-q^{-1/2}$) in \eqref{Andrews1} and then substituting the resulting identity into \eqref{proof-7-2-S0} (resp.~\eqref{proof-7-2-S1}), we deduce that
\begin{align}
    S_0(q^{\frac{1}{4}})&=(-q;q^2)_\infty \sum_{n=0}^\infty \frac{q^{n^2+n}(-1;q^2)_n}{(q;q)_{2n}}=\frac{J_6^5}{J_1J_3^2J_{12}^2}, \quad \text{(by \eqref{S48})} \label{proof-7-2-S0-result} \\
    S_1(q^{\frac{1}{4}})&=2q^{\frac{3}{4}}(-q^2;q^2)_\infty \sum_{n=0}^\infty \frac{q^{n^2+2n}(-q;q^2)_n}{(q;q)_{2n+1}}=2q^{\frac{3}{4}}\frac{J_4J_{2,12}}{J_1J_2}. \quad \text{(by \eqref{S50})} \label{proof-7-2-S1-result} 
\end{align}
Substituting \eqref{proof-7-2-S0-result} and \eqref{proof-7-2-S1-result} into \eqref{proof-7-start}, and then using the method in \cite{Frye-Garvan} to prove theta function identities, we obtain \eqref{dual-7-2}.

(3) Let $(u,v,w)=(q^4,q^{-2},q^2)$. We have
\begin{align*}
    S_0(q)&=(-q^8;q^8)_\infty \sum_{n=0}^\infty q^{4n^2+4n}(-1;q^8)_n\sum_{i=-n}^n \frac{q^{4i^2+4i}}{(q^8;q^8)_{n-i}(q^8;q^8)_{n+i}}, \\
    S_1(q)&=q^5(-q^4;q^8)_\infty \sum_{n=0}^\infty q^{4n^2+8n}(-q^4;q^8)_n \sum_{i=-n-1}^n \frac{q^{4i^2+8i}}{(q^8;q^8)_{n-i}(q^8;q^8)_{n+i+1}}.
\end{align*}
Using \eqref{Andrews1} with $x=-q^{-1}$ and \eqref{Andrews2} with $x=-q^{-3/2}$, we have
\begin{align}
S_0(q^{\frac{1}{4}})&=(-q^2;q^2)_\infty \sum_{n=0}^\infty q^{n^2+n}\frac{(-1;q^2)_n^2(-q^2;q^2)_n}{(q^2;q^2)_{2n}}=\frac{J_4^7}{J_2^5J_8^2},  \quad \text{(by \eqref{new-2})}  \label{proof-7-3-S0-result} \\
S_1(q^{\frac{1}{4}})&=q^{\frac{1}{4}}(-q;q^2)_\infty \sum_{n=0}^\infty \frac{q^{n^2+2n}(-q;q^2)_n^2(-q;q^2)_{n+1}}{(q^2;q^2)_{2n+1}} \nonumber \\
&=q^{\frac{1}{4}}\frac{J_2^3J_8^2}{J_1^2J_4^3}. \quad \text{(by \eqref{new-3})} \label{proof-7-3-S1-result}
\end{align}
Substituting \eqref{proof-7-3-S0-result} and \eqref{proof-7-3-S1-result} into \eqref{proof-7-start}, we obtain \eqref{dual-7-3}.

(4) Let $(u,v,w)=(q^{-4},q^2,q^2)$. We have
\begin{align*}
    S_0(q)&=(-1;q^8)_\infty \sum_{n=0}^\infty q^{4n^2}(-q^8;q^8)_n \sum_{i=-n}^n \frac{q^{4i^2}}{(q^8;q^8)_{n-i}(q^8;q^8)_{n+i}}, \\
    S_1(q)&=q(-q^4;q^8)_\infty \sum_{n=0}^\infty q^{4n^2+4n}(-q^4;q^8)_{n+1}\sum_{i=-n-1}^n \frac{q^{4i^2+4i}}{(q^8;q^8)_{n-i}(q^8;q^8)_{n+i+1}}.
\end{align*}
Using \eqref{Andrews1} with $x=-q^{-1/2}$ and \eqref{Andrews2} with $x=-q^{-1}$, we have
\begin{align}
    S_0(q^{\frac{1}{4}})&=2(-q^2;q^2)_\infty \sum_{n=0}^\infty \frac{q^{n^2}(-q;q^2)_n}{(q;q)_{2n}}=2\frac{J_4J_6^2}{J_1J_2J_{12}},  \quad \text{(by \eqref{S29})}\label{proof-7-4-S0-result} \\
    S_1(q^{\frac{1}{4}})&=2q^{\frac{1}{4}}(-q;q^2)_\infty \sum_{n=0}^\infty \frac{q^{n^2+n}(-q^2;q^2)_n}{(q;q)_{2n+1}} = 2q^{\frac{1}{4}}\frac{J_2^2J_3J_{12}}{J_1^2J_4J_6}. \quad \text{(by \eqref{S28})} \label{proof-7-4-S1-result}
\end{align}
Substituting \eqref{proof-7-4-S0-result} and \eqref{proof-7-4-S1-result} into \eqref{proof-7-start}, we obtain \eqref{dual-7-4}.

(5) Let $(u,v,w)=(q^{-4},1,1)$. We have
\begin{align*}
S_0(q)&=(-1;q^8)_\infty \sum_{n=0}^\infty q^{4n^2-4n}(-q^8;q^8)_n\sum_{i=-n}^n \frac{q^{4i^2}}{(q^8;q^8)_{n-i}(q^8;q^8)_{n+i}},  \\
S_1(q)&=q^{-1}(-q^4;q^8)_\infty \sum_{n=0}^\infty q^{4n^2}(-q^4;q^8)_{n+1}\sum_{i=-n-1}^n \frac{q^{4i^2+4i}}{(q^8;q^8)_{n-i}(q^8;q^8)_{n+i+1}}. 
\end{align*}
Using \eqref{Andrews1} with $x=-q^{-1/2}$ and \eqref{Andrews2} with $x=-q^{-1}$, we have
\begin{align}
    S_0(q^{\frac{1}{4}})&=2(-q^2;q^2)_\infty \sum_{n=0}^\infty \frac{q^{n^2-n}(-q;q^2)_n}{(q;q)_{2n}}=4\frac{J_4^2}{J_1J_2},  \quad \text{(by \eqref{new})}  \label{proof-7-S0-result}\\
    S_1(q^{\frac{1}{4}})&=2q^{-\frac{1}{4}}(-q;q^2)_\infty\sum_{n=0}^\infty \frac{q^{n^2}(-q^2;q^2)_n}{(q;q)_{2n+1}}=2q^{-\frac{1}{4}}\frac{J_2^5}{J_1^3J_4^2}. \quad \text{(by \eqref{Rama1713})} \label{proof-7-S1-result}
\end{align}
Substituting \eqref{proof-7-S0-result} and \eqref{proof-7-S1-result} into \eqref{proof-7-start} and using the method in \cite{Frye-Garvan} to verify theta function identities, we obtain \eqref{dual-7-5}.
\end{proof}

\begin{rem}
The matrix part for this example corresponds to the case $m=1/4$ in the matrix part of the following modular triples discovered by Cao and Wang \cite{Cao-Wang-rank3}:
\begin{align}\label{eq-vector-B-dual-exam1}
   & \widetilde{A}=\begin{pmatrix} 1 & -\frac{1}{2} & -\frac{1}{2} \\ -\frac{1}{2} & \frac{1}{2}+m & \frac{1}{2}-m \\   -\frac{1}{2} & \frac{1}{2}-m & \frac{1}{2}+m \end{pmatrix}, ~~\widetilde{B}_1=\begin{pmatrix} 0  \\ \nu \\ -\nu  \end{pmatrix},\widetilde{B}_2=\begin{pmatrix} -1/2  \\ 0 \\ 0 \end{pmatrix},  \widetilde{B}_3=\begin{pmatrix}  1/2 \\ -1/4 \\ 1/4  \end{pmatrix}, \nonumber \\
    &\widetilde{B}_4=\begin{pmatrix} 1/2  \\ 1/4 \\ -1/4  \end{pmatrix},~~ \widetilde{C}_1=\frac{2\nu^2}{(4m+1)}-\frac{1}{12}, ~~ \widetilde{C}_2=\frac{m+1}{6(4m+1)}, ~~ \widetilde{C}_3=\frac{1}{24},  ~~ \widetilde{C}_4=\frac{1}{24}.
\end{align}
Cao and Wang \cite[Theorem 3.1]{Cao-Wang-rank3} proved the following identities to justify their modularity:
\begin{align}
&\sum_{i,j,k\geq 0} \frac{q^{2i^2+2m(j-k)^2+(j+k)^2-2i(j+k)+\nu(j-k)}}{(q^4;q^4)_i(q^4;q^4)_j(q^4;q^4)_k} \nonumber \\
&=\frac{J_4^3\overline{J}_{2(4m+\nu+1),4(4m+1)}}{J_2^2J_8^2} +2q^{2m+\nu+1} \frac{J_8^2\overline{J}_{-2\nu,4(4m+1)}}{J_4^3}, \label{CW-id-1}\\
&\sum_{i,j,k\geq 0} \frac{q^{2i^2+2m(j-k)^2+(j+k)^2-2i(j+k)-2i}}{(q^4;q^4)_i(q^4;q^4)_j(q^4;q^4)_k} =4\frac{J_{8}^{2}\overline{J}_{8m,16m+4}}{J_{4}^{3}}+2q^{2m-1}\frac{J_{4}^{3}\overline{J}_{2,16m+4}}{J_{2}^{2}J_{8}^{2}}, \label{CW-id-2}\\
&\sum_{i,j,k\geq 0} \frac{q^{2i^2+2m(j-k)^2+(j+k)^2-2i(j+k)+2i+j-k}}{(q^4;q^4)_i(q^4;q^4)_j(q^4;q^4)_k}=\frac{J_8^2\overline{J}_{8m+2,16m+4}}{J_4^3}+q^{2m}\frac{J_4^3J_{32m+8}^2}{J_2^2J_8^2J_{16m+4}}. \label{CW-id-3}
\end{align}
The identities \eqref{dual-7-1}, \eqref{dual-7-3} and \eqref{dual-7-5} follow from \eqref{CW-id-1}, \eqref{CW-id-3} and \eqref{CW-id-2} with $m=1/4$, respectively. 
\end{rem}

\subsection{Dual of Example 8}
We list Zagier's Example 8 and its dual in Table \ref{tab-exam8}. Here the modular triple corresponding to $B=(1,0,1/2)^\mathrm{T}$ and $C=1/12$ was missing from Zagier's list \cite[Table 3]{Zagier} and was found by Wang \cite[Table 8]{Wang-rank3}.

{\small
\begin{table}[htbp]
\centering
\caption{Modular triples for Example 8 and their duals}\label{tab-exam8}
\begin{tabular}{c|ccccccc}
  \hline
    \padedvphantom{I}{4ex}{4ex}
  $A$ &  \multicolumn{7}{c}{$\begin{pmatrix}
4 & 2 & 1 \\ 2 & 2 & 0 \\ 1 & 0 & 1
\end{pmatrix}$}  \\
  \hline
\padedvphantom{I}{4ex}{4ex}
$B$ & $\begin{pmatrix} -1 \\ -1 \\ 1/2 \end{pmatrix}$ & $\begin{pmatrix} 0 \\ -1/2 \\ 1/2 \end{pmatrix}$ & $\begin{pmatrix} 0 \\ 0 \\ 1/2 \end{pmatrix}$ & $\begin{pmatrix} 2 \\ 1 \\ 1/2 \end{pmatrix}$ & $\begin{pmatrix} 0 \\ 0 \\ 0 \end{pmatrix}$ & $\begin{pmatrix} 2 \\ 1 \\ 1 \end{pmatrix}$ & $\begin{pmatrix} 1 \\ 0 \\ 1/2  \end{pmatrix}$ \\
  ~~ $C$ & $1/12$ & $1/48$ & $0$ & $1/3$ & $-1/24$ & $11/24$ & $1/12$ \\
  \hline
    \padedvphantom{I}{4ex}{4ex}
  $\mathcal{D}(A)$ &  \multicolumn{7}{c}{$
\begin{pmatrix}
1 & -1 & -1 \\
-1 & 3/2 & 1 \\
-1 & 1 & 2
\end{pmatrix}$}  \\
  \hline
\padedvphantom{I}{4ex}{4ex}
$\mathcal{D}(B)$ & $\begin{pmatrix} -1/2 \\ 0 \\ 1 \end{pmatrix}$ &
$\begin{pmatrix} 0 \\ -1/4 \\ 1/2 \end{pmatrix}$ & $\begin{pmatrix} -1/2 \\ 1/2 \\ 1 \end{pmatrix}$ &
$\begin{pmatrix} 1/2 \\ 0 \\ 0 \end{pmatrix}$ & $\begin{pmatrix} 0 \\ 0 \\ 0 \end{pmatrix}$ & $\begin{pmatrix} 0 \\ 1/2 \\ 1 \end{pmatrix}$ & $\begin{pmatrix} 1/2 \\ -1/2 \\ 0 \end{pmatrix}$ \\
  ~~ $\mathcal{D}(C)$ & $7/24$ & $1/24$ & $1/8$ & $1/24$ & $-1/12$ & $1/6$ & $1/24$ \\
  \hline
\end{tabular}
\end{table}
}

The modularity of Example 8 has been confirmed by Wang \cite[Theorem 4.10]{Wang-rank3}.
As for its dual, we establish the following identities to prove its modularity. 
\begin{theorem}\label{thm-dual-8}
We have
\begin{align}
\sum_{i,j,k\geq 0} \frac{q^{2i^2+3j^2+4k^2-4ij-4ik+4jk-2i+4k}}{(q^4;q^4)_i(q^4;q^4)_j(q^4;q^4)_k}&=2q^{-1}\frac{J_2^5J_8^2}{J_1^2J_4^5}, \label{id-dual-8-1} \\
\sum_{i,j,k\geq 0} \frac{q^{2i^2+3j^2+4k^2-4ij-4ik+4jk-j+2k}}{(q^4;q^4)_i(q^4;q^4)_j(q^4;q^4)_k}&=2\frac{J_4^2}{J_2^2}, \label{id-dual-8-2} \\
\sum_{i,j,k\geq 0} \frac{q^{2i^2+3j^2+4k^2-4ij-4ik+4jk-2i+2j+4k}}{(q^4;q^4)_i(q^4;q^4)_j(q^4;q^4)_k}&
=2\frac{J_2^2J_3J_{24}^4}{J_1J_4^2J_6J_{12}J_{4,24}^2}, \label{id-dual-8-3} \\
\sum_{i,j,k\geq 0} \frac{q^{2i^2+3j^2+4k^2-4ij-4ik+4jk+2i}}{(q^4;q^4)_i(q^4;q^4)_j(q^4;q^4)_k}&=\frac{J_6^2J_{6,24}^3J_{24}}{J_3^2J_4J_{12}J_{4,24}^2},  \label{id-dual-8-4}\\
\sum_{i,j,k\geq 0} \frac{q^{2i^2+3j^2+4k^2-4ij-4ik+4jk}}{(q^4;q^4)_i(q^4;q^4)_j(q^4;q^4)_k}&=\frac{J_{12}^3J_{4,24}^2J_{5,24}J_{7,24}}{J_1J_{24}^6}, \label{id-dual-8-5} \\
\sum_{i,j,k\geq 0} \frac{q^{2i^2+3j^2+4k^2-4ij-4ik+4jk+2j+4k}}{(q^4;q^4)_i(q^4;q^4)_j(q^4;q^4)_k}&=\frac{J_{12}^3J_{1,24}J_{11,24}J_{4,24}^2}{J_1J_{24}^6},  \label{id-dual-8-6}\\
\sum_{i,j,k\geq 0} \frac{q^{2i^2+3j^2+4k^2-4ij-4ik+4jk+2i-2j}}{(q^4;q^4)_i(q^4;q^4)_j(q^4;q^4)_k}&=\frac{J_2^5J_8^2}{J_1^2J_4^5}. \label{id-dual-8-7}
\end{align}
\end{theorem}
\begin{proof}
We define \begin{align}\label{dual-8-def}
    &F(u,v,w;q):=\sum_{i,j,k\geq 0}\frac{q^{2i^2+3j^2+4k^2-4ij-4ik+4jk}u^iv^jw^k}{(q^4;q^4)_{i}(q^4;q^4)_j(q^4;q^4)_k}\nonumber\\
    &=\sum_{j,k\geq 0}\frac{q^{3j^2+4k^2+4jk}v^jw^k}{(q^4;q^4)_j(q^4;q^4)_k}(-q^{2-4(j+k)}u;q^4)_{\infty}\nonumber\\
    &=(-q^2u;q^4)_{\infty}\sum_{j,k\geq 0}\frac{q^{j^2+2k^2}u^{j+k}v^jw^k}{(q^4;q^4)_{j}(q^4;q^4)_k}\cdot(-{q^2}/{u};q^4)_{j+k}.
\end{align}

Setting $(b,c,z)=(-q^{\frac{1}{2}+j}u^{-1},0,-q^{\frac{1}{2}}uwa^{-1})$ and let $a\rightarrow\infty$ in \eqref{Heine}, we have
\begin{align}\label{dual-8-heine}
    \sum_{k\geq 0}\frac{(-{q^{2+4j}}/{u};q^4)_kq^{2k^2}u^kw^k}{(q^4;q^4)_k}=(-q^2uw,-{q^{2+4j}}/{u};q^4)_{\infty}\sum_{k\geq 0}\frac{(-1)^kq^{4jk+2k}u^{-k}}{(q^4,-q^2uw;q^4)_k}.
\end{align}
Substituting \eqref{dual-8-heine} into \eqref{dual-8-def}, we have
\begin{align}
  & F(u,v,w;q) \nonumber \\
  &=(-q^2u;q^4)_\infty\sum_{j\geq 0}\frac{q^{j^2}u^jv^j(-q^2u^{-1};q^4)_j}{(q^4;q^4)_j}\sum_{k\geq 0}\frac{(-q^{2+4j}u^{-1};q^4)_kq^{2k^2}u^kw^k}{(q^4;q^4)_k}\nonumber\\
    &=(-q^2u,-q^2u^{-1},-q^2uw;q^4)_{\infty}\sum_{j,k\geq 0}\frac{(-1)^kq^{j^2+4jk+2k}u^{j-k}v^j}{(q^4;q^4)_j(q^4;q^4)_k(-q^2uw;q^4)_k}.
\end{align}

Now we discuss the seven cases of $(u,v,w)$ corresponding to \eqref{id-dual-8-1}--\eqref{id-dual-8-7} separately. Note that the relation $uw=q^2$ holds except for the Nahm sums in \eqref{id-dual-8-1}--\eqref{id-dual-8-4} and \eqref{id-dual-8-7}. Summing over $k$ by \eqref{Euler}, we have
\begin{align}\label{add-dual-8-proof}
F(u,v,q^2/u;q)=(-q^2u,-q^2u^{-1},-q^4;q^4)_\infty\sum_{j\geq 0} \frac{q^{j^2}u^jv^j}{(q^4;q^4)_j(-q^{4j+2}/u;q^8)_\infty}.
\end{align}

(1) Let $(u,v,w)=(q^{-2},1,q^4)$. From \eqref{add-dual-8-proof} we have
\begin{align}
    &F(q^{-2},1,q^4;q) =2(-q^4;q^4)_{\infty}^3\sum_{j\geq 0}\frac{q^{j^2-2j}}{(q^4;q^4)_j(-q^{4j+4};q^8)_{\infty}}\nonumber\\
    &=2(-q^4;q^4)_{\infty}^3\Big(\frac{1}{(-q^4;q^8)_{\infty}}\sum_{j\geq 0}\frac{(-q^{4};q^8)_jq^{4j^2-4j}}{(q^4;q^4)_{2j}}+\frac{1}{(-q^8;q^8)_{\infty}}\sum_{j\geq 0}\frac{(-q^8;q^8)_jq^{4j^2-1}}{(q^4;q^4)_{2j+1}}\Big). \nonumber
\end{align}
Substituting \eqref{new} and \eqref{Rama1713} into it, we obtain \eqref{id-dual-8-1}.

(2) Let $(u,v,w)=(1,q^{-1},q^2)$. From \eqref{add-dual-8-proof} we have
\begin{align}\label{dual-8-ex2-1}
    &F(1,q^{-1},q^2;q)=(-q^2;q^4)_{\infty}^2(-q^4;q^4)_{\infty}\sum_{j\geq 0}\frac{q^{j^2-j}}{(q^4;q^4)_j(-q^{4j+2};q^8)_{\infty}}\nonumber\\
    &=(-q^2;q^4)_{\infty}^2(-q^4;q^4)_{\infty}\Big(\frac{1}{(-q^2;q^8)_{\infty}}\sum_{j\geq 0}\frac{(-q^2;q^8)_jq^{4j^2-2j}}{(q^4;q^4)_{2j}}\nonumber\\
    &\quad +\frac{1}{(-q^6;q^8)_{\infty}}\sum_{j\geq 0}\frac{(-q^6;q^8)_jq^{4j^2+2j}}{(q^4;q^4)_{2j+1}} \Big).
\end{align}
Substituting \eqref{dual-8-ex2-2} and \eqref{dual-8-ex2-3} into \eqref{dual-8-ex2-1}, we obtain \eqref{id-dual-8-2}.

(3) Let $(u,v,w)=(q^{-2},q^2,q^4)$.  From \eqref{add-dual-8-proof} we have
\begin{align}
    &F(q^{-2},q^2,q^4;q)=2(-q^4;q^4)_{\infty}^3\sum_{j\geq 0}\frac{q^{j^2}}{(q^4;q^4)_j(-q^{4j+4};q^8)_{\infty}}\nonumber\\
    &=2(-q^4;q^4)_{\infty}^3\Big( \frac{1}{(-q^4;q^8)_{\infty}}\sum_{j\geq 0}\frac{(-q^4;q^8)_jq^{4j^2}}{(q^4;q^4)_{2j}}+\frac{1}{(-q^8;q^8)_{\infty}}\sum_{j\geq 0}\frac{(-q^8;q^8)_jq^{4j^2+4j+1}}{(q^4;q^4)_{2j+1}} \Big). \nonumber
\end{align}
Substituting \eqref{S29} and \eqref{S28} into it, we obtain \eqref{id-dual-8-3}.

(4) Let $(u,v,w)=(q^2,1,1)$.  From \eqref{add-dual-8-proof} we have
\begin{align}
    &F(q^2,1,1;q)=2(-q^4;q^4)_{\infty}^3\sum_{j\geq 0}\frac{q^{j^2+2j}}{(q^4;q^4)_j(-q^{4j};q^8)_{\infty}}\nonumber\\
    &=2(-q^4;q^4)_{\infty}^3\Big( \frac{1}{(-1;q^8)_{\infty}}\sum_{j\geq 0}\frac{(-1;q^8)_jq^{4j^2+4j}}{(q^4;q^4)_{2j}} \nonumber \\
    &\quad \quad +\frac{1}{(-q^4;q^8)_{\infty}}\sum_{j\geq 0}\frac{(-q^4;q^8)_jq^{4j^2+8j+3}}{(q^4;q^4)_{2j+1}} \Big). \nonumber
\end{align}
Substituting \eqref{S48} and \eqref{S50} into it, we obtain \eqref{id-dual-8-4}.

 (5) Let $(u,v,w)=(1,1,1)$. We have
 \begin{align}
     &F(1,1,1;q)=(-q^2;q^4)_{\infty}^3\sum_{j,k\geq 0}\frac{(-1)^kq^{j^2+4jk+2k}}{(q^4;q^4)_j(-q^2;-q^2)_{2k}} \nonumber \\
     &=(-q^2;q^4)_\infty^3 \big(K_{0}(-q^2)+qK_{1}(-q^2)  \big) \label{G111-split}
 \end{align}
where
\begin{align}
        K_{0}(q):=\sum_{j,k\geq 0}\frac{q^{2j^2+4jk+k}}{(q^2;q^2)_{2j}(q;q)_{2k}}, \quad 
        K_{1}(q):=\sum_{j,k\geq 0}\frac{q^{2j^2+4jk+2j+3k}}{(q^2;q^2)_{2j+1}(q;q)_{2k}}.
\end{align}
 We further define some related series:
    \begin{align}
        L_{0}(q):=\sum_{j,k\geq 0}\frac{q^{4j^2+4jk+k}}{(q^4;q^4)_{2j}(q^2;q^2)_{k}}, \quad
        L_{1}(q):=\sum_{j,k\geq 0}\frac{q^{4j^2+4jk+4j+3k}}{(q^4;q^4)_{2j+1}(q^2;q^2)_{k}}. 
    \end{align}
By \eqref{Euler} and \eqref{S53} we have
    \begin{align}\label{dual-8-ex5-L10}
        L_{0}(q)=\frac{1}{(q;q^2)_{\infty}}\sum_{j\geq 0}\frac{q^{4j^2}(q;q^2)_{2j}}{(q^4;q^4)_{2j}}=\frac{J_{5,12}}{(q;q^2)_{\infty}J_{4}}.
    \end{align}
    Replacing $q$ by $-q$ in \eqref{dual-8-ex5-L10},  we have
    \begin{align}\label{dual-8-ex5-L'10}
        L_{0}(-q)=\frac{\overline{J}_{5,12}}{(-q;q^2)_{\infty}J_{4}}.
    \end{align}
By \eqref{dual-8-ex5-L10} and \eqref{dual-8-ex5-L'10} and the method in \cite{Frye-Garvan}, we have
    \begin{align}\label{K10-result}
        K_{0}(q)=\frac{1}{2}(L_{0}(q^{\frac{1}{2}})+L_0(-q^{\frac{1}{2}}))=\frac{J_{24}^{6}}{(q;q^2)_{\infty}J_{2,24}^2J_{6,24}^2J_{8,24}J_{10,24}}.
    \end{align}

    Similarly, by \eqref{Euler} and \eqref{S55} we have
    \begin{align}\label{dual-8-ex5-L11}
        L_{1}(q)=\frac{1}{(q;q^2)_{\infty}}\sum_{j\geq 0}\frac{q^{4j^2+4j}(q;q^2)_{2j+1}}{(q^4;q^4)_{2j+1}}=\frac{J_{1,12}}{(q;q^2)_{\infty}J_{4}}.
    \end{align}
Replacing $q$ by $-q$ in \eqref{dual-8-ex5-L11}, we have
    \begin{align}\label{dual-8-ex5-L'11}
        L_{1}(-q)=\frac{\overline{J}_{1,12}}{(-q;q^2)_{\infty}J_4}.
    \end{align}
It follows that
    \begin{align}\label{K11-result}
        K_{1}(q)=\frac{1}{2}(L_{1}(q^{\frac{1}{2}})+L_{1}(-q^{\frac{1}{2}}))=\frac{J_{24}^{12}}{J_{2,24}J_{3,24}J_{4,24}^2J_{5,24}^2J_{6,24}^2J_{7,24}^2J_{8,24}J_{9,24}}.
    \end{align}
    Substituting \eqref{K10-result} and \eqref{K11-result} into \eqref{G111-split}, 
   we obtain \eqref{id-dual-8-5}.

(6) Let $(u,v,w)=(1,q^2,q^4)$. We have
\begin{align}
    F(1,q^2,q^4;q)&=\frac{(-q^2;q^4)_{\infty}^3}{1+q^2}\sum_{j,k\geq 0}\frac{(-1)^kq^{j^2+4jk+2j+2k}}{(q^4;q^4)_j(q^4;-q^2)_{2k}} \nonumber \\
    &=\frac{(-q^2;q^4)_{\infty}^3}{1+q^2}\big(\widetilde{K}_{0}(-q^2)+q^3\widetilde{K}_{1}(-q^2)  \big) \label{G124-split}
\end{align}
where
\begin{align}
    \widetilde{K}_{0}(q):=\sum_{j,k\geq 0}\frac{q^{2j^2+4jk+2j+k}}{(q^2;q^2)_{2j}(q^2;q)_{2k}}, \quad 
    \widetilde{K}_{1}(q):=\sum_{j,k\geq 0}\frac{q^{2j^2+4jk+4j+3k}}{(q^2;q^2)_{2j+1}(q^2;q)_{2k}}.
\end{align}
We further define some related series:
\begin{align}
    \widetilde{L}_{0}(q):=\sum_{j,k\geq 0}\frac{q^{4j^2+4jk+4j+k}}{(q^4;q^4)_{2j}(q^4;q^2)_k}, \quad
    \widetilde{L}_{1}(q):=\sum_{j,k\geq 0}\frac{q^{4j^2+4jk+8j+3k}}{(q^4;q^4)_{2j+1}(q^4;q^2)_k}.
\end{align}
Then we have
\begin{align}\label{K2-L2}
    \widetilde{K}_{0}(q)=\frac{1}{2}(\widetilde{L}_{0}(q^{\frac{1}{2}})+\widetilde{L}_{0}(-q^{\frac{1}{2}})), \quad 
    \widetilde{K}_{1}(q)=\frac{1}{2}(\widetilde{L}_{1}(q^{\frac{1}{2}})+\widetilde{L}_{1}(-q^{\frac{1}{2}})).
\end{align}

We have
\begin{align}
    &\widetilde{L}_{0}(q)=(1-q^2)\sum_{j,k\geq 0}\frac{q^{4j^2+4jk+4j+k}}{(q^4;q^4)_{2j}(q^2;q^2)_{k+1}}\nonumber\\
    &=(1-q^2)\Big(-\sum_{j\geq 0}\frac{q^{4j^2-1}}{(q^4;q^4)_{2j}}+\sum_{j,k\geq 0}\frac{q^{4j^2+4jk+k-1}}{(q^4;q^4)_{2j}(q^2;q^2)_k}\Big)\nonumber\\
    &=(1-q^2)\Big(-\sum_{j\geq 0}\frac{q^{4j^2-1}}{(q^4;q^4)_{2j}}+q^{-1}L_{0}(q)\Big).
\end{align}
Substituting \eqref{S79} and \eqref{dual-8-ex5-L10} into it, we obtain
\begin{align}\label{dual-8-ex6-L20}
    \widetilde{L}_{0}(q)=(1-q^2)q^{-1}\Big(\frac{J_{5,12}}{(q;q^2)_{\infty}J_{4}}-\frac{J_8J_{80}}{J_4J_{16,80}}\Big).
\end{align}
Replacing $q$ by $-q$ in \eqref{dual-8-ex6-L20}, we obtain
\begin{align}\label{dual-8-ex6-L'20}
    \widetilde{L}_{0}(-q)=(1-q^2)q^{-1}\Big(-\frac{\overline{J}_{5,12}}{(-q;q^2)_{\infty}J_{4}}+\frac{J_8J_{80}}{J_4J_{16,80}}\Big).
\end{align}

Similarly we have
\begin{align}
    &\widetilde{L}_{1}(q)=(1-q^2)\sum_{j,k\geq 0}\frac{q^{4j^2+4jk+8j+3k}}{(q^4;q^4)_{2j+1}(q^2;q^2)_{k+1}}\nonumber\\
    &=(1-q^2)\Big(-\sum_{j\geq 0}\frac{q^{4j^2+4j-3}}{(q^4;q^4)_{2j+1}} +\sum_{j,k\geq 0}\frac{q^{4j^2+4jk+4j+3k-3}}{(q^4;q^4)_{2j+1}(q^2;q^2)_k} \Big)\nonumber\\
    &=(1-q^2)\Big(-\sum_{j\geq 0}\frac{q^{4j^2+4j-3}}{(q^4;q^4)_{2j+1}} +q^{-3}L_{1}(q) \Big).
\end{align}
Substituting \eqref{S94} and \eqref{dual-8-ex5-L11} into it, we obtain
\begin{align}\label{dual-8-ex6-L21}
   \widetilde{L}_{1}(q)=(1-q^2)q^{-3}\Big(\frac{J_{1,12}}{(q;q^2)_{\infty}J_{4}}-\frac{J_{12,40}J_{16,80}}{J_4J_{80}}\Big).
\end{align}
Replacing $q$ by $-q$ in \eqref{dual-8-ex6-L21}, we obtain
\begin{align}\label{dual-8-ex6-L'21}
    \widetilde{L}_{1}(q)=(1-q^2)q^{-3}\Big(-\frac{\overline{J}_{1,12}}{(-q;q^2)_{\infty}J_{4}}+\frac{J_{12,40}J_{16,80}}{J_4J_{80}}\Big).
\end{align}
Substituting \eqref{dual-8-ex6-L20}, \eqref{dual-8-ex6-L'20}, \eqref{dual-8-ex6-L21} and \eqref{dual-8-ex6-L'21} into \eqref{K2-L2}, and then substituting the result into \eqref{G124-split},  we obtain \eqref{id-dual-8-6}.

(7) Let $(u,v,w)=(q^2,q^{-2},1)$. We have
\begin{align}
    &F(q^2,q^{-2},1;q)=2(-q^4;q^4)_{\infty}^3\sum_{j,k\geq 0}\frac{(-1)^kq^{j^2+4jk}}{(q^4;q^4)_j(q^8;q^8)_k}  \\
    &=2(-q^4;q^4)_{\infty}^3\sum_{j\geq 0}\frac{q^{j^2}}{(q^4;q^4)_j(-q^{4j};q^8)_{\infty}}\nonumber\\
    &=2(-q^4;q^4)_{\infty}^3\Big( \frac{1}{(-1;q^8)_{\infty}}\sum_{j\geq 0}\frac{(-1;q^8)_jq^{4j^2}}{(q^4;q^4)_{2j}} +\frac{1}{(-q^4;q^8)_{\infty}}\sum_{j\geq 0}\frac{(-q^4;q^8)_jq^{4j^2+4j+1}}{(q^4;q^4)_{2j+1}} \Big).  \nonumber
\end{align}
Substituting \eqref{S47} and \eqref{H-cor-id-1} into it, we obtain \eqref{id-dual-8-7}.
\end{proof}

\subsection{Dual of Example 9}\label{sec-exam9}
We present Zagier's Example~9 and its dual in Table~\ref{tab-exam9}. The final entry, which is absent from Zagier's original list~\cite[Table~3]{Zagier}, was subsequently identified by Wang~\cite[Table~8]{Wang-rank3}. 

{\small 
\begin{table}[htbp]
\centering
\caption{Modular triples for Example 9 and their duals}\label{tab-exam9}
\begin{tabular}{c|cccc}
  \hline
    \padedvphantom{I}{4ex}{4ex}
  $A$ &  \multicolumn{4}{c}{$\begin{pmatrix} 6 &4 & 2 \\ 4 & 4 & 2 \\ 2 &2 & 2 \end{pmatrix}$}  \\
  \hline
\padedvphantom{I}{4ex}{4ex}
$B$ & $\begin{pmatrix} 0 \\ 0 \\ 0 \end{pmatrix}$ & $\begin{pmatrix} 1  \\ 0  \\ 0 \end{pmatrix}$ & $\begin{pmatrix} 2 \\ 1 \\ 0\end{pmatrix}$  & $\begin{pmatrix} 3 \\ 2 \\ 1 \end{pmatrix}$ \\
  ~~ $C$ & $-1/36$ & $1/12$ & $11/36$ & $23/36$ \\
  \hline
    \padedvphantom{I}{4ex}{4ex}
  $\mathcal{D}(A)$ &  \multicolumn{4}{c}{$\begin{pmatrix} 1/2 & -1/2 & 0 \\ -1/2 & 1 & -1/2 \\ 0 & -1/2 & 1 \end{pmatrix}$}  \\
  \hline
\padedvphantom{I}{4ex}{4ex}
$\mathcal{D}(B)$ & $\begin{pmatrix} 0 \\ 0 \\ 0 \end{pmatrix}$ & $\begin{pmatrix} 1/2  \\ -1/2  \\ 0 \end{pmatrix}$ & $\begin{pmatrix} 1/2 \\ 0 \\ -1/2 \end{pmatrix}$  & $\begin{pmatrix} 1/2 \\ 0 \\ 0 \end{pmatrix}$ \\
  ~~ $\mathcal{D}(C)$ & $-7/72$ & $1/24$ & $5/72$ & $-1/72$ \\
  \hline
\end{tabular}
\end{table}
}

The modularity of Zagier's Example 9 has been confirmed by Wang \cite[Theorem 4.11]{Wang-rank3}.  
Now we study the modularity of its dual. 

Recall the functions $M^{(h)}(q)$ and $M^{(h)}_{x,y,z}(q)$ defined in \eqref{M-defn} and \eqref{def-Mxyz}. Our first observation is that for each $h=1,2,3,4$, there are essentially only four different $M_{x,y,z}^{(h)}(q)$.
\begin{lemma}\label{lem-M}
We have
    \begin{align}\label{dual-9-firstlemma}
        M_{1,0,0}^{(1)}(q)&=M_{1,1,0}^{(1)}(q)=M_{1,1,1}^{(1)}(q),\\
        M_{0,0,1}^{(1)}(q)&=M_{0,1,0}^{(1)}(q)=M_{0,1,1}^{(1)}(q),\label{dual-9-firstlemma1}\\
        M_{1,0,1}^{(2)}(q)&=M_{1,1,0}^{(2)}(q)=M_{1,1,1}^{(2)}(q),\\
        M_{0,0,0}^{(2)}(q)&=M_{0,1,0}^{(2)}(q)=M_{0,1,1}^{(2)}(q),\\
        M_{1,0,0}^{(3)}(q)&=M_{1,0,1}^{(3)}(q)=M_{1,1,0}^{(3)}(q),\\
        M_{0,0,0}^{(3)}(q)&=M_{0,0,1}^{(3)}(q)=M_{0,1,1}^{(3)}(q),\\
        M_{1,0,0}^{(4)}(q)&=M_{1,1,0}^{(4)}(q)=M_{1,1,1}^{(4)}(q),\\
        M_{0,0,1}^{(4)}(q)&=M_{0,1,0}^{(4)}(q)=M_{0,1,1}^{(4)}(q). \label{dual-9-firstlemma2}
    \end{align}
\end{lemma}

\begin{proof}
By \eqref{Euler} we have
    \begin{align*}
        &M_{1,0,0}^{(1)}(q)-M_{1,1,0}^{(1)}(q)=\sum_{i,j,k\geq 0}\frac{(-1)^jq^{(2i+1-j)^2+(j-2k)^2+4k^2}}{(q^4;q^4)_{2i+1}(q^4;q^4)_j(q^4;q^4)_{2k}}\nonumber\\
        &=\sum_{i,k\geq 0}\frac{q^{4i^2+8k^2+4i+1}}{(q^4;q^4)_{2i+1}(q^4;q^4)_{2k}}\cdot (q^{-4i-4k};q^4)_{\infty}=0, \\
        &M_{1,1,0}^{(1)}(q)-M_{1,1,1}^{(1)}(q)=\sum_{i,j,k\geq 0}\frac{(-1)^kq^{(2i-2j)^2+(2j+1-k)^2+k^2}}{(q^4;q^4)_{2i+1}(q^4;q^4)_{2j+1}(q^4;q^4)_k}\nonumber\\
        &=\sum_{i,j\geq0}\frac{q^{(2i-2j)^2+(2j+1)^2}}{(q^4;q^4)_{2i+1}(q^4;q^4)_{2j+1}}\cdot(q^{-4j};q^4)_{\infty}=0.
    \end{align*}
Here for the last equality we used the fact that $(q^{-4n};q^4)_{\infty}=0$ for any $n\geq 0$. This proves \eqref{dual-9-firstlemma}. The identities \eqref{dual-9-firstlemma1}--\eqref{dual-9-firstlemma2} can be proved in the same way.
\end{proof}
Now we pick a representative for each of the four different $M_{x,y,z}^{(h)}(q)$:
\begin{equation}
\begin{split}
    &A^{(1)}(q):=M_{0,0,0}^{(1)}(q), \quad B^{(1)}(q):=M_{1,0,0}^{(1)}(q), \\
    &C^{(1)}(q):=M_{0,0,1}^{(1)}(q), \quad D^{(1)}(q):=M_{1,0,1}^{(1)}(q), \\
    & A^{(2)}(q):=M_{1,0,0}^{(2)}(q), \quad  B^{(2)}(q):=M_{0,0,0}^{(2)}(q), \\
    & C^{(2)}(q):=M_{1,1,1}^{(2)}(q), \quad D^{(2)}(q):=M_{0,0,1}^{(2)}(q), \\
    &A^{(3)}(q):=M_{0,1,0}^{(3)}(q), \quad B^{(3)}(q):=M_{1,0,0}^{(3)}(q),  \\
    & C^{(3)}(q):=M_{0,0,0}^{(3)}(q),  \quad D^{(3)}(q):=M_{1,1,1}^{(3)}(q), \\
    &A^{(4)}(q):=M_{1,0,1}^{(4)}(q),  \quad B^{(4)}(q):=M_{0,0,1}^{(4)}(q), \\
    &C^{(4)}(q):=M_{1,0,0}^{(4)}(q), \quad D^{(4)}(q):=M_{0,0,0}^{(4)}(q).
\end{split}
\end{equation}
Considering the power of $q$ in the series expansions of these series, we see that these series can be classified as shown in Table \ref{tab-ABCD}.
\begin{table}[htbp]
    \centering
    \caption{Classifications of series by the residue of the exponent modulo $4$} \label{tab-ABCD}
    \renewcommand{\arraystretch}{1.15}  
    \begin{tabular}{*{4}{>{\(}c<{\)}}}
        \toprule
        \mathbb{Z}[[q^4]] & q\mathbb{Z}[[q^4]] & q^{2}\mathbb{Z}[[q^4]] & q^{3}\mathbb{Z}[[q^4]] \\
        \midrule
        R_4^{(1)}(q) & R_1^{(1)}(q) & R_2^{(1)}(q) & R_3^{(1)}(q) \\
        R_1^{(2)}(q) & R_2^{(2)}(q)  & R_3^{(2)}(q)  & R_4^{(2)}(q) \\ 
       R_2^{(3)}(q) & R_3^{(3)}(q) & R_4^{(3)}(q) & R_1^{(3)}(q) \\
       R_3^{(4)}(q)   &  R_4^{(4)}(q) &  R_1^{(4)}(q)    & R_2^{(4)}(q) \\
        A^{(1)}(q) & B^{(1)}(q) & C^{(1)}(q) & D^{(1)}(q) \\
        B^{(2)}(q) & C^{(2)}(q) & D^{(2)}(q) & A^{(2)}(q) \\
        C^{(3)}(q) & D^{(3)}(q) & A^{(3)}(q) & B^{(3)}(q) \\
        D^{(4)}(q) & A^{(4)}(q) & B^{(4)}(q) & C^{(4)}(q) \\
        \bottomrule
    \end{tabular}
\end{table}

By definition and Lemma \ref{lem-M} we have
\begin{align}\label{M-ABCD}
    M^{(h)}(q)=A^{(h)}(q)+3B^{(h)}(q)+3C^{(h)}(q)+D^{(h)}(q).
\end{align}
Therefore, to evaluate the Nahm sums $M^{(h)}(q)$, it suffices to evaluate $A^{(h)}(q)$, $B^{(h)}(q)$, $C^{(h)}(q)$ and $D^{(h)}(q)$. We will present two different approaches. The first approach is more straightforward and relies on the results in Section \ref{sec-double}. The second approach gives equivalent uniform representations for the dual Nahm sums, and leads us to the discovery of the new modular rank three Nahm sums stated in Theorem \ref{thm-ex9-new}.

\subsubsection{First approach}\label{sec-1st}
\begin{theorem}\label{thm-S-prod}
Let 
\begin{align}
    S_h(q):=A^{(h)}(q)+B^{(h)}(q)-C^{(h)}(q)-D^{(h)}(q).
\end{align}
We have
\begin{align}
 S_1(q)&=\frac{J_2^2}{J_4^3}\Big(
 \mathcal A(J_{86,180}-q^{16}J_{14,180})
 -q^2\mathcal B(J_{58,180}+q^{10}J_{22,180})
 \Big),
 \label{eq-S1-product}\\
 S_2(q)&=q^{-1}\frac{J_2^2}{J_4^3}\Big(
 \mathcal B(J_{78,180}+q^6J_{42,180})
 -\mathcal A(q^2J_{66,180}+q^{20}J_{6,180})
 \Big),
 \label{eq-S2-product}\\
 S_3(q)&=\frac{J_2^2}{J_4^3}\Big(
 \mathcal A(-J_{74,180}+q^8J_{34,180})
 +\mathcal B(q^6J_{38,180}+q^{20}J_{2,180})
 \Big),
 \label{eq-S3-product}\\
 S_4(q)&=q^{-1}\frac{J_2^2}{J_4^3}\Big(
 \mathcal A(q^6J_{46,180}+q^{12}J_{26,180})
 -\mathcal B(J_{82,180}+q^2J_{62,180})
 \Big).
 \label{eq-S4-product}
\end{align}
Here $\mathcal A$ and $\mathcal B$ are the infinite products given in
\eqref{eq-A-product}--\eqref{eq-B-product}.
\end{theorem}

\begin{proof}
Note that for $n\geq0$,
\begin{align}
 \sum_{k\geq0}\frac{(-1)^kq^{(2n-k)^2+k^2}}{(q^4;q^4)_k}
 &=(-1)^nq^{2n^2}(q^2;q^4)_n(q^2;q^4)_\infty,
 \label{eq-simplify-even}\\
 \sum_{k\geq0}\frac{(-1)^kq^{(2n+1-k)^2+k^2-2k}}{(q^4;q^4)_k}
 &=(-1)^{n+1}q^{2n^2-1}(q^2;q^4)_{n+1}(q^2;q^4)_\infty.
 \label{eq-simplify-odd}
\end{align}
Recall the sequences $\mathcal{X}_n$ and $\mathcal{Y}_n$ defined in \eqref{eq-Xn} and \eqref{eq-Yn}. We have
\begin{align}
    S_1(q)&=M_{0,0,0}^{(1)}+M_{1,0,0}^{(1)}-M_{0,0,1}^{(1)}-M_{1,0,1}^{(1)}\nonumber \\
    &=\sum_{\substack{i,j,k\geq 0 \\ j \equiv 0 \!\!\! \pmod{2}  }} \frac{(-1)^kq^{(i-j)^2+(j-k)^2+k^2}}{(q^4;q^4)_i(q^4;q^4)_j(q^4;q^4)_k} =\sum_{i,n,k\geq 0} \frac{(-1)^kq^{(i-2n)^2+(2n-k)^2+k^2}}{(q^4;q^4)_i(q^4;q^4)_{2n}(q^4;q^4)_k} \nonumber \\
    &=(q^2;q^4)_\infty \sum_{n\geq 0} \frac{(-1)^nq^{2n^2} (q^2;q^4)_n}{(q^4;q^4)_{2n}}  \sum_{i\geq 0} \frac{q^{(i-2n)^2}}{(q^4;q^4)_i} \nonumber \\
    &=(q^2;q^4)_\infty\sum_{n\geq0}\frac{(-1)^nq^{2n^2}(q^2;q^4)_n}{(q^4;q^4)_{2n}}\mathcal{X}_n.
 \label{eq-S1-start}
\end{align}
Similarly, we can prove that
\begin{align}
    S_2(q)&=M_{1,0,0}^{(2)}+M_{0,0,0}^{(2)}-M_{1,0,1}^{(2)}-M_{0,0,1}^{(2)}\nonumber \\
    &=q^{-1}(q^2;q^4)_\infty\sum_{n\geq0}\frac{(-1)^nq^{2n^2}(q^2;q^4)_n}{(q^4;q^4)_{2n}}\mathcal{Y}_n,
 \label{eq-S2-start}\\
 S_3(q)&=M_{0,1,0}^{(3)}+M_{1,1,0}^{(3)}-M_{0,1,1}^{(3)}-M_{1,1,1}^{(3)}\nonumber \\
 &=(q^2;q^4)_\infty\sum_{n\geq0} \frac{(-1)^{n+1}q^{2n^2+4n}(q^2;q^4)_{n+1}}{(q^4;q^4)_{2n+1}} \mathcal{X}_n,
 \label{eq-S3-start}\\
 S_4(q)&=M_{1,0,1}^{(4)}+M_{0,0,1}^{(4)}-M_{1,0,0}^{(4)}-M_{0,0,0}^{(4)}  \nonumber \\
 &=-q^{-1}(q^2;q^4)_\infty\sum_{n\geq0}\frac{(-1)^nq^{2n^2+4n}(q^2;q^4)_n}{(q^4;q^4)_{2n}}\mathcal{Y}_n.
 \label{eq-S4-start}
\end{align}

Substituting \eqref{eq-X-split}--\eqref{eq-Y-split} into \eqref{eq-S1-start}--\eqref{eq-S4-start}, we obtain
\begin{align}
 S_1(q)&=(q^2;q^4)_\infty
 (\mathcal A\mathcal P_1+\mathcal B\mathcal R_1),
 \label{eq-S1-exp}\\
 S_2(q)&=q^{-1}(q^2;q^4)_\infty
 (\mathcal B \mathcal P_2
  +\mathcal A \mathcal R_2),
 \label{eq-S2-exp}\\
 S_3(q)&=(q^2;q^4)_\infty
 (\mathcal A\mathcal P_3+\mathcal B\mathcal R_3),
 \label{eq-S3-exp}\\
 S_4(q)&=-q^{-1}(q^2;q^4)_\infty
 ( \mathcal{B} {\mathcal P}_4
  +\mathcal A {\mathcal R}_4).
 \label{eq-S4-exp}
\end{align}
Finally, substituting the results in Lemma~\ref{lem-PR-product} into the above formulas, we obtain \eqref{eq-S1-product}--\eqref{eq-S4-product}.
\end{proof}

For $h=1,2,3,4$, we have
\begin{align}
 \mathcal{F}_0^{(h)}(q):=A^{(h)}-C^{(h)}
 &=\frac{S_h(q)-(-1)^{h}S_h(-q)}2,
 \label{eq-F0-dissection}\\
\mathcal{F}_1^{(h)}(q):=B^{(h)}-D^{(h)}
 &=\frac{S_h(q)+(-1)^{h}S_h(-q)}2.
 \label{eq-F1-dissection}
\end{align}
Substituting \eqref{eq-A-product} and \eqref{eq-B-product} into
\eqref{eq-S1-product}--\eqref{eq-S4-product}, and then extracting the even and odd power terms, we obtain the following formulas.
\begin{corollary}\label{cor-ABCD-difference}
We have for $k=0,1$ that
\begin{align}
& \mathcal{F}_k^{(1)}(q)=q^{k}\frac{J_2^2}{J_4^3}\Big(
 \mathcal A_k(q^4)(J_{86,180}-q^{16}J_{14,180})
 -q^2\mathcal B_k(q^4)(J_{58,180}+q^{10}J_{22,180})
 \Big),
 \label{eq-F1}\\
& \mathcal{F}_k^{(2)}(q)=q^{k-1}\frac{J_2^2}{J_4^3}\Big(
 \mathcal B_k(q^4)(J_{78,180}+q^6J_{42,180})
 -\mathcal A_k(q^4)(q^2J_{66,180}+q^{20}J_{6,180})
 \Big),
 \label{eq-F2}\\
& \mathcal{F}_k^{(3)}(q)=q^k\frac{J_2^2}{J_4^3}\Big(
 \mathcal A_k(q^4)(-J_{74,180}+q^8J_{34,180})
 +\mathcal B_k(q^4)(q^6J_{38,180}+q^{20}J_{2,180})
 \Big),
 \label{eq-F3}\\
&\mathcal{F}_k^{(4)}(q)=q^{k-1}\frac{J_2^2}{J_4^3}\Big(
 \mathcal A_k(q^4)(q^6J_{46,180}+q^{12}J_{26,180})
 -\mathcal B_k(q^4)(J_{82,180}+q^2J_{62,180})
 \Big).
 \label{eq-F4}
\end{align}
\end{corollary}

\begin{proof}[Proof of Theorem \ref{thm-dual-exam9-1st}]
Replacing $q$ by $\ii q$ in \eqref{eq-F0-dissection} and \eqref{eq-F1-dissection}, we obtain
\begin{align}
    A^{(h)}+C^{(h)}=\ii^{h-1}\mathcal{F}_0^{(h)}(\ii q), \quad
    B^{(h)}+D^{(h)}=\ii^{h-2}\mathcal{F}_1^{(h)}(\ii q).
\end{align}
Therefore, we have
\begin{align}
 A^{(h)}&=\frac12\big(\mathcal{F}_0^{(h)}(q)+\ii^{h-1}\mathcal{F}_0^{(h)}(\ii q)\big), \quad
 C^{(h)}=\frac12\big(\ii^{h-1}\mathcal{F}_0^{(h)}(\ii q)-\mathcal{F}_0^{(h)}(q)\big),
 \label{eq-AC-iq}\\
 B^{(h)}&=\frac12\big(\mathcal{F}_1^{(h)}(q)+\ii^{h-2}\mathcal{F}_1^{(h)}(\ii q)\big), \quad
 D^{(h)}=\frac12\big(\ii^{h-2}\mathcal{F}_1^{(h)}(\ii q)-\mathcal{F}_1^{(h)}(q)\big).  \label{eq-BD-iq}
\end{align}
Substituting \eqref{eq-AC-iq} and \eqref{eq-BD-iq} into \eqref{M-ABCD}, we obtain \eqref{eq-thm-Mh}.
\end{proof}
When $q$ is replaced by $\ii q$, $J_2$ becomes
\begin{equation}
(-q^2;-q^2)_\infty=\frac{J_4^3}{J_2J_8}.
 \label{eq-J2iq}
\end{equation}
When $a\equiv 2$ (mod 4) and $m\equiv 0$ (mod 4), when $q$ is replaced by $\ii q$, $J_{a,m}$ becomes $\overline J_{a,m}$. Clearly, $\mathcal{A}_k(q^4)$ and $\mathcal{B}_k(q^4)$ ($k=0,1$) will not change upon replacing $q$ by $\ii q$. 
Substituting \eqref{eq-F1}--\eqref{eq-F4} into \eqref{eq-AC-iq} and \eqref{eq-BD-iq}, we obtain explicit formulas for $A^{(h)}$, $B^{(h)}$, $C^{(h)}$ and $D^{(h)}$. For instance, we have
\begin{align}
& A^{(1)} =\frac{1}{2} \frac{J_2^2}{J_4^3}\Big(
 \mathcal A_0(q^4)(J_{86,180}-q^{16}J_{14,180})
 -q^2\mathcal B_0(q^4)(J_{58,180}+q^{10}J_{22,180})\Big) \nonumber \\
 &+\frac{1}{2}\frac{J_4^3}{J_2^2J_8^2}\Big(
 \mathcal A_0(q^4)(\overline{J}_{86,180}-q^{16}\overline{J}_{14,180})
 +q^2\mathcal B_0(q^4)(\overline{J}_{58,180}-q^{10}\overline{J}_{22,180})
 \Big).
\end{align}

\subsubsection{Second approach}\label{sec-2nd}
We now introduce an alternative approach to studying the Nahm sums dual to Example 9. This approach has three advantages. First, it reveals additional interesting relations among the partial Nahm sums $A^{(h)}$, $B^{(h)}$, $C^{(h)}$, and $D^{(h)}$. Second, it yields another uniform representation for $M^{(h)}(q)$. Third, it leads to the discovery of the new rank three Nahm sums appearing in Theorems \ref{thm-ex9-new} and \ref{thm-dual-rank3}.

From now on we set
\begin{align}\label{ab-defn}
    a=a(q):=\sum_{n=-\infty}^\infty q^{(2n)^2}=\frac{J_8^5}{J_4^2J_{16}^2}, \quad  b=b(q):=\sum_{n=-\infty}^\infty q^{(2n+1)^2}=2q\frac{J_{16}^2}{J_8}.
\end{align}
From \eqref{Jacobi} and \cite[Entry 25 vii]{Berndt-notebook} we have
\begin{align}\label{eq-ab-sum}
    a+b=\sum_{n=-\infty}^\infty q^{n^2}=\frac{J_2^5}{J_1^2J_4^2}, \quad   a^4-b^4=\frac{J_4^8}{J_8^4}. 
\end{align}

Recall the theta series $W_k(q)$ ($k=1,2,3,4$) in Lemma \ref{lem-W-product}. In order to express the dual Nahm sums as infinite products, we define some series $R_{i}^{(h)}(q)$ (abbreviated as $R_i^{(h)}$) for $i,h\in\{1,2,3,4\}$ as follows:
\begin{align}
    R_{1}^{(1)}&:=\frac{1}{2}q \Big(\frac{J_4^2J_{4,72}J_{28,72}J_{32,72}J_{36,72}}{J_{2}J_8J_{14,72}J_{18,72}J_{22,72}}+\frac{J_{2}J_{18}J_{36}J_{14,36}}{J_{8,36}J_{12,36}J_{16,36}}\Big), \label{R11}\\
    R_{2}^{(1)}&:=-q^6\frac{J_{8}J_{8,72}}{J_4}, \label{R21} \\
    R_{3}^{(1)}&:=\frac{1}{2}q\Big(\frac{J_4^2J_{4,72}J_{28,72}J_{32,72}J_{36,72}}{J_{2}J_8J_{14,72}J_{18,72}J_{22,72}}-\frac{J_{2}J_{18}J_{36}J_{14,36}}{J_{8,36}J_{12,36}J_{16,36}}\Big),   \label{R31}\\
R_{4}^{(1)}&:=-\frac{J_4}{J_8^3}\cdot\sum_{n\in\mathbb{Z}}(-1)^n(9n+1)q^{36n^2+8n}=-\frac{J_4}{J_8^3}W_1(q^4);  \label{R41}\\
    R_{1}^{(2)}&:=\frac{1}{2}\Big(-\frac{J_4^2J_{24}J_{12,24}}{J_2J_8J_{6,24}}-\frac{J_4J_{2,12}J_{6,12}}{J_{12}J_{4,12}}\Big),  \label{R12} \\
     R_{2}^{(2)}&:=q\frac{J_{8}J_{24}}{J_4},  \label{R22}\\
     R_{3}^{(2)}&:=\frac{1}{2}\Big(-\frac{J_4^2J_{24}J_{12,24}}{J_2J_8J_{6,24}}+\frac{J_4J_{2,12}J_{6,12}}{J_{12}J_{4,12}}\Big),  \label{R32}\\
     R_{4}^{(2)}&:=\frac{J_4}{J_8^3}\cdot\sum_{n\in\mathbb{Z}}(-1)^n(9n+3)q^{36n^2+24n+3}=q^3\frac{J_4}{J_8^3}W_3(q^4);  \label{R42} \\
    R_{1}^{(3)}&:=\frac{1}{2}q\Big(\frac{J_4^2J_{4,72}J_{16,72}J_{20,72}J_{36,72}}{J_2J_8J_{2,72}J_{18,72}J_{34,72}}-\frac{J_2J_{18}J_{36}J_{2,36}}{J_{4,36}J_{8,36}J_{12,36}}\Big),  \label{R13} \\
    R_{2}^{(3)}&:=-\frac{J_8J_{32,72}}{J_4}, \label{R23} \\
    R_{3}^{(3)}&:=\frac{1}{2}q\Big(\frac{J_4^2J_{4,72}J_{16,72}J_{20,72}J_{36,72}}{J_2J_8J_{2,72}J_{18,72}J_{34,72}}+\frac{J_2J_{18}J_{36}J_{2,36}}{J_{4,36}J_{8,36}J_{12,36}}\Big),  \label{R33} \\
    R_{4}^{(3)}&:=-\frac{J_4}{J_8^3}\cdot\sum_{n\in\mathbb{Z}}(-1)^n(9n+4)q^{36n^2+32n+6}=-q^6\frac{J_4}{J_8^3}W_4(q^4); \label{R43}\\
    R_{1}^{(4)}&:=\frac{1}{2}\Big(-\frac{J_4^2J_{8,72}J_{20,72}J_{28,72}J_{36,72}}{J_2J_8J_{10,72}J_{18,72}J_{26,72}}+\frac{J_2J_{18}J_{36}J_{10,36}}{J_{4,36}J_{12,36}J_{16,36}}\Big); \label{R14} \\
    R_{2}^{(4)}&:=q^3\frac{J_{8}J_{16,72}}{J_4}; \label{R23R24} \\
    R_{3}^{(4)}&:=\frac{1}{2}\Big(-\frac{J_4^2J_{8,72}J_{20,72}J_{28,72}J_{36,72}}{J_2J_8J_{10,72}J_{18,72}J_{26,72}}-\frac{J_2J_{18}J_{36}J_{10,36}}{J_{4,36}J_{12,36}J_{16,36}}\Big),  \label{R34} \\
    R_{4}^{(4)}&:=\frac{J_4}{J_8^3}\cdot\sum_{n\in\mathbb{Z}}(-1)^n(9n+2)q^{36n^2+16n+1}=q\frac{J_4}{J_8^3}W_2(q^4).        \label{R44}
    \end{align}
Note that in the definitions of $R_1^{(h)}(q)$ and $R_3^{(h)}(q)$ ($h=1,2,3,4$), the second infinite product is obtained from the first by replacing $q^2$ with $-q^2$. 
The above series can be classified into four families of the form  $q^r\mathbb{Z}[[q^4]]$ ($r=0,1,2,3$), as shown in Table \ref{tab-ABCD}.

We find four relations satisfied by them. 
\begin{theorem}\label{thm-ABCD}
The following equations hold for $h=1,2,3,4$:
    \begin{align}
        b\cdot A^{(h)}(q)-a\cdot B^{(h)}(q)=R_{1}^{(h)},  \label{dual-9-eq1}\\
        b\cdot B^{(h)}(q)-a\cdot C^{(h)}(q)=R_{2}^{(h)}, \label{dual-9-eq2} \\
          b\cdot C^{(h)}(q)-a\cdot D^{(h)}(q)=R_{3}^{(h)},  \label{dual-9-eq3} \\
        b\cdot D^{(h)}(q)-a\cdot A^{(h)}(q)=R_{4}^{(h)}. \label{dual-9-eq4}
    \end{align}
\end{theorem}
Since we already have product representations for $A^{(h)}$, $B^{(h)}$, $C^{(h)}$ and $D^{(h)}$ using the first approach, it is routine to prove the above theorem using the Maple approach in \cite{Frye-Garvan}. Our aim is to provide a proof without using any results in Section \ref{sec-1st}. We are able to prove \eqref{dual-9-eq1}--\eqref{dual-9-eq3} in a uniform manner. However, we face difficulty in proving \eqref{dual-9-eq4} and leave its proof to the interested reader. 

We divide our proof of \eqref{dual-9-eq1}--\eqref{dual-9-eq3} into a sequence of lemmas.
\begin{lemma}\label{dual-ex9-lemma1}
Let 
\begin{align}\label{T-defn}
T(u,v,w;q):=\sum_{i,j,k\geq 0}\frac{q^{4i^2+3j^2+2k^2+4ij-2jk}u^iv^jw^k}{(q^4;q^4)_i(q^4;q^4)_j(q^4;q^4)_k}.
\end{align}
For $h=1,2,3,4$ we define
\begin{align}\label{S-ABCD-defn}
    S^{(h)}(q):=b(A^{(h)}(q)+C^{(h)}(q)-2B^{(h)}(q))+a(2C^{(h)}(q)-B^{(h)}(q)-D^{(h)}(q)).
\end{align} 
We have
 \begin{align}
        &S^{(1)}(q)=qJ_4T(q^4,q^2,1;q),  \label{eq-lem-ABCD-1}\\
    &S^{(2)}(q)=-J_4T(1,q^{-2},1;q),  \label{eq-lem-ABCD-2}\\
    &S^{(3)}(q)=J_4T(1,1,q^{-2};q), \label{eq-lem-ABCD-3} \\
    &S^{(4)}(q)=-J_4T(1,1,1;q).  \label{eq-lem-ABCD-4}
    \end{align}
\end{lemma}

\begin{proof}
(1) We have
\begin{align*}
   &\sum_{\substack{i,j,k\geq 0\\i\equiv j\ (\mathrm{mod}\ 2 )}}\frac{(-1)^iq^{(i-j)^2+(j-k)^2+k^2}}{(q^4;q^4)_{i}(q^4;q^4)_{j}(q^4;q^4)_{k}}=M_{0,0,0}^{(1)}+M_{0,0,1}^{(1)}-M_{1,1,0}^{(1)}-M_{1,1,1}^{(1)}\nonumber \\
   &=A^{(1)}(q)+C^{(1)}(q)-2B^{(1)}(q),\\
    &\sum_{\substack{i,j,k\geq 0\\i\not\equiv j\ (\mathrm{mod}\ 2 )}}\frac{(-1)^iq^{(i-j)^2+(j-k)^2+k^2}}{(q^4;q^4)_{i}(q^4;q^4)_{j}(q^4;q^4)_{k}}=M_{0,1,0}^{(1)}+M_{0,1,1}^{(1)}-M_{1,0,0}^{(1)}-M_{1,0,1}^{(1)} \nonumber \\
&=2C^{(1)}(q)-B^{(1)}(q)-D^{(1)}(q).
\end{align*}
Substituting them into \eqref{S-ABCD-defn}, we have
\begin{align}
    S^{(1)}(q)=\sum_{i,j,k\geq 0}\frac{(-1)^iq^{(i-j-1)^2+(j-k)^2+k^2+2i-2j-1}}{(q^4;q^4)_i(q^4;q^4)_j(q^4;q^4)_k}\sum_{s\in\mathbb{Z}}q^{(2s-(i-j-1))^2}. \label{exam9-S1-proof}
\end{align}
Note that
\begin{align}
    (2s-(i-j-1))^2+(i-j-1)^2=2s^2+2(i-j-1-s)^2.
\end{align}
From \eqref{exam9-S1-proof} we have
\begin{align}
    &S^{(1)}(q)=\sum_{\substack{i,j,k\geq 0\\s,t\in\mathbb{Z}\\s+t=i-j-1}}\frac{(-1)^iq^{2s^2+2t^2+(j-k)^2+k^2+2i-2j-1}}{(q^4;q^4)_i(q^4;q^4)_j(q^4;q^4)_k}\nonumber\\
    &=\mathrm{CT}_z\Big[\sum_{\substack{i,j,k\geq 0\\s,t\in\mathbb{Z}}} \frac{(-1)^iq^{2s^2+2t^2+(j-k)^2+k^2+2i-2j-1}z^{i-j-1-s-t}}{(q^4;q^4)_i(q^4;q^4)_j(q^4;q^4)_k} \Big]\nonumber\\
    &=\mathrm{CT}_z\Big[ \sum_{\substack{j,k\geq 0\\s\in\mathbb{Z}}}\frac{q^{2s^2+(j-k)^2+k^2-2j-1}z^{-j-1-s}}{(q^4;q^4)_j(q^4;q^4)_k}\frac{(-zq^2,-z^{-1}q^2,q^4;q^4)_{\infty}}{(-zq^{2};q^4)_{\infty}} \Big]\nonumber\\
    &=(q^4;q^4)_{\infty}\mathrm{CT}_z\Big[ \sum_{\substack{j,k\geq 0\\s\in\mathbb{Z}}}\frac{q^{2s^2+(j-k)^2+k^2-2j-1}z^{-j-1-s}}{(q^4;q^4)_j(q^4;q^4)_k}\sum_{i\geq 0}\frac{q^{2i^2}z^{-i}}{(q^4;q^4)_i} \Big]\nonumber\\
    &=(q^4;q^4)_{\infty}\sum_{i,j,k\geq 0}\frac{q^{2i^2+2(i+j+1)^2+(j-k)^2+k^2-2j-1}}{(q^4;q^4)_i(q^4;q^4)_j(q^4;q^4)_k}\nonumber\\
    &=(q^4;q^4)_{\infty}\sum_{i,j,k\geq 0}\frac{q^{4i^2+4ij+3j^2-2jk+2k^2+4i+2j+1}}{(q^4;q^4)_i(q^4;q^4)_j(q^4;q^4)_k}=qJ_4T(q^4,q^2,1;q).
\end{align}

(2) We have
\begin{align*}
    &-\sum_{\substack{i,j,k\geq 0\\i\not\equiv j\ (\mathrm{mod}\ 2)}}\frac{(-1)^iq^{(i-j)^2+(j-k)^2+k^2+2i-2j}}{(q^4;q^4)_i(q^4;q^4)_j(q^4;q^4)_k}=M_{1,0,0}^{(2)}+M_{1,0,1}^{(2)}-M_{0,1,0}^{(2}-M_{0,1,1}^{(2)} \nonumber \\
    &=A^{(2)}(q)+C^{(2)}(q)-2B^{(2)}(q), \\
    &-\sum_{\substack{i,j,k\geq 0\\i\equiv j\ (\mathrm{mod}\ 2)}}\frac{(-1)^iq^{(i-j)^2+(j-k)^2+k^2+2i-2j}}{(q^4;q^4)_i(q^4;q^4)_j(q^4;q^4)_k}=M_{1,1,0}^{(2)}+M_{1,1,1}^{(2)}-M_{0,0,0}^{(2)}-M_{0,0,1}^{(2)} \nonumber \\
    &=2C^{(2)}(q)-B^{(2)}(q)-D^{(2)}(q).
\end{align*}
Substituting them into \eqref{S-ABCD-defn}, we have
\begin{align}
S^{(2)}(q)=-\sum_{i,j,k\geq 0}\frac{(-1)^iq^{(i-j)^2+(j-k)^2+k^2+2i-2j}}{(q^4;q^4)_i(q^4;q^4)_j(q^4;q^4)_k}\sum_{s\in\mathbb{Z}}q^{(2s-(i-j))^2}.  \label{exam9-S2-proof}
\end{align}

Note that
\begin{align}\label{eq-sij}
(2s-(i-j))^2+(i-j)^2=2s^2+2(i-j-s)^2.   
\end{align}
From \eqref{exam9-S2-proof} we have
\begin{align*}
    &S^{(2)}(q)=-\sum_{\substack{i,j,k\geq 0\\s,t\in\mathbb{Z}\\s+t=i-j}}\frac{(-1)^iq^{2s^2+2t^2+(j-k)^2+k^2+2i-2j}}{(q^4;q^4)_i(q^4;q^4)_j(q^4;q^4)_k}\nonumber\\
    &=-\mathrm{CT}_z\Big[\sum_{\substack{i,j,k\geq 0\\s,t\in\mathbb{Z}}} \frac{(-1)^iq^{2s^2+2t^2+(j-k)^2+k^2+2i-2j}z^{i-j-s-t}}{(q^4;q^4)_i(q^4;q^4)_j(q^4;q^4)_k} \Big]\nonumber\\
    &=-\mathrm{CT}_z\Big[ \sum_{\substack{j,k\geq 0\\s\in\mathbb{Z}}}\frac{q^{2s^2+(j-k)^2+k^2-2j}z^{-j-s}}{(q^4;q^4)_j(q^4;q^4)_k}\frac{(-zq^2,-z^{-1}q^2,q^4;q^4)_{\infty}}{(-zq^{2};q^4)_{\infty}} \Big]\nonumber\\
    &=-(q^4;q^4)_\infty \mathrm{CT}_z\Big[ \sum_{\substack{j,k\geq 0\\s\in\mathbb{Z}}}\frac{q^{2s^2+(j-k)^2+k^2-2j}z^{-j-s}}{(q^4;q^4)_j(q^4;q^4)_k}\sum_{i\geq 0}\frac{q^{2i^2}z^{-i}}{(q^4;q^4)_i} \Big]\nonumber\\
    &=-(q^4;q^4)_\infty\sum_{i,j,k\geq 0}\frac{q^{2i^2+2(i+j)^2+(j-k)^2+k^2-2j}}{(q^4;q^4)_i(q^4;q^4)_j(q^4;q^4)_k}\nonumber\\
    &=-(q^4;q^4)_\infty\sum_{i,j,k\geq 0}\frac{q^{4i^2+4ij+3j^2-2jk+2k^2-2j}}{(q^4;q^4)_i(q^4;q^4)_j(q^4;q^4)_k}=-J_4T(1,q^{-2},1;q). 
\end{align*}

(3) We have
\begin{align*}
   &\sum_{\substack{i,j,k\geq 0\\i\not\equiv j\ (\mathrm{mod}\ 2)}}\frac{(-1)^iq^{(i-j)^2+(j-k)^2+k^2+2i-2k}}{(q^4;q^4)_i(q^4;q^4)_j(q^4;q^4)_k}=M_{0,1,0}^{(3)}+M_{0,1,1}^{(3)}-M_{1,0,0}^{(3)}-M_{1,0,1}^{(3)}  \nonumber \\
   &=A^{(3)}(q)+C^{(3)}(q)-2B^{(3)}(q), \\
&\sum_{\substack{i,j,k\geq 0\\i\equiv j\ (\mathrm{mod}\ 2)}}\frac{(-1)^iq^{(i-j)^2+(j-k)^2+k^2+2i-2k}}{(q^4;q^4)_i(q^4;q^4)_j(q^4;q^4)_k}=M_{0,0,0}^{(3)}+M_{0,0,1}^{(3)}-M_{1,1,0}^{(3)}-M_{1,1,1}^{(3)} \nonumber \\
&=2C^{(3)}(q)-B^{(3)}(q)-D^{(3)}(q).
\end{align*}
Substituting them into \eqref{S-ABCD-defn}, we have
\begin{align}
    S^{(3)}(q)=\sum_{i,j,k\geq 0}\frac{(-1)^iq^{(i-j)^2+(j-k)^2+k^2+2i-2k}}{(q^4;q^4)_i(q^4;q^4)_j(q^4;q^4)_k}\sum_{s\in\mathbb{Z}}q^{(2s-(i-j))^2}.
\end{align}
Using \eqref{eq-sij} we have
\begin{align}
    &S^{(3)}(q)=\sum_{\substack{i,j,k\geq 0\\s,t\in\mathbb{Z}\\s+t=i-j}}\frac{(-1)^iq^{2s^2+2t^2+(j-k)^2+k^2+2i-2k}}{(q^4;q^4)_i(q^4;q^4)_j(q^4;q^4)_k}\nonumber\\
    &=\mathrm{CT}_z\big[\sum_{\substack{i,j,k\geq 0\\s,t\in\mathbb{Z}}} \frac{(-1)^iq^{2s^2+2t^2+(j-k)^2+k^2+2i-2k}z^{i-j-s-t}}{(q^4;q^4)_i(q^4;q^4)_j(q^4;q^4)_k} \big]\nonumber\\
    &=\mathrm{CT}_z\Big[ \sum_{\substack{j,k\geq 0\\s\in\mathbb{Z}}}\frac{q^{2s^2+(j-k)^2+k^2-2k}z^{-j-s}}{(q^4;q^4)_j(q^4;q^4)_k}\frac{(-zq^2,-z^{-1}q^2,q^4;q^4)_{\infty}}{(-zq^{2};q^4)_{\infty}} \Big]\nonumber\\
    &=(q^4;q^4)_\infty \mathrm{CT}_z\Big[ \sum_{\substack{j,k\geq 0\\s\in\mathbb{Z}}}\frac{q^{2s^2+(j-k)^2+k^2-2k}z^{-j-s}}{(q^4;q^4)_j(q^4;q^4)_k}\sum_{i\geq 0}\frac{q^{2i^2}z^{-i}}{(q^4;q^4)_i} \Big]\nonumber\\
    &=(q^4;q^4)_\infty\sum_{i,j,k\geq 0}\frac{q^{2i^2+2(i+j)^2+(j-k)^2+k^2-2k}}{(q^4;q^4)_i(q^4;q^4)_j(q^4;q^4)_k}\nonumber\\
    &=(q^4;q^4)_\infty\sum_{i,j,k\geq 0}\frac{q^{4i^2+4ij+3j^2-2jk+2k^2-2k}}{(q^4;q^4)_i(q^4;q^4)_j(q^4;q^4)_k}=J_4T(1,1,q^{-2};q).
\end{align}

(4) We have
\begin{align*}
    &-\sum_{\substack{i,j,k\geq 0\\i\not\equiv j\ (\mathrm{mod}\ 2 )}}\frac{(-1)^iq^{(i-j)^2+(j-k)^2+k^2+2i}}{(q^4;q^4)_{i}(q^4;q^4)_{j}(q^4;q^4)_{k}}=M_{1,0,0}^{(4)}+M_{1,0,1}^{(4)}-M_{0,1,0}^{(4)}-M_{0,1,1}^{(4)} \nonumber  \\
    &=A^{(4)}(q)+C^{(4)}(q)-2B^{(4)}(q), \\
    &-\sum_{\substack{i,j,k\geq 0\\i \equiv j\ (\mathrm{mod}\ 2 )}}\frac{(-1)^iq^{(i-j)^2+(j-k)^2+k^2+2i}}{(q^4;q^4)_{i}(q^4;q^4)_{j}(q^4;q^4)_{k}}=M_{1,1,0}^{(4)}+M_{1,1,1}^{(4)}-M_{0,0,0}^{(4)}-M_{0,0,1}^{(4)} \nonumber \\
    &=2C^{(4)}(q)-B^{(4)}(q)-D^{(4)}(q).
\end{align*}
Substituting them into \eqref{S-ABCD-defn}, we have
\begin{align}
    S^{(4)}(q)=-\sum_{i,j,k\geq 0}\frac{(-1)^iq^{(i-j)^2+(j-k)^2+k^2+2i}}{(q^4;q^4)_i(q^4;q^4)_j(q^4;q^4)_k}\sum_{s\in\mathbb{Z}}q^{(2s-(i-j))^2}. \label{proof-ABCD-4}
\end{align}
Using \eqref{eq-sij} we have
\begin{align*}
    &S^{(4)}(q)=-\sum_{\substack{i,j,k\geq 0\\s,t\in\mathbb{Z}\\s+t=i-j}}\frac{(-1)^iq^{2s^2+2t^2+(j-k)^2+k^2+2i}}{(q^4;q^4)_i(q^4;q^4)_j(q^4;q^4)_k}\nonumber\\
    &=-\mathrm{CT}_z\Big[\sum_{\substack{i,j,k\geq 0\\s,t\in\mathbb{Z}}} \frac{(-1)^iq^{2s^2+2t^2+(j-k)^2+k^2+2i}z^{i-j-s-t}}{(q^4;q^4)_i(q^4;q^4)_j(q^4;q^4)_k} \Big]\nonumber\\
    &=-\mathrm{CT}_z\Big[ \sum_{\substack{j,k\geq 0\\s\in\mathbb{Z}}}\frac{q^{2s^2+(j-k)^2+k^2}z^{-j-s}}{(q^4;q^4)_j(q^4;q^4)_k}\frac{(-zq^2,-z^{-1}q^2,q^4;q^4)_{\infty}}{(-zq^{2};q^4)_{\infty}} \Big]\nonumber\\
    &=-(q^4;q^4)_\infty \mathrm{CT}_z\Big[ \sum_{\substack{j,k\geq 0\\s\in\mathbb{Z}}}\frac{q^{2s^2+(j-k)^2+k^2}z^{-j-s}}{(q^4;q^4)_j(q^4;q^4)_k}\sum_{i\geq 0}\frac{q^{2i^2}z^{-i}}{(q^4;q^4)_i} \Big]\nonumber\\
    &=-(q^4;q^4)_\infty \sum_{i,j,k\geq 0}\frac{q^{2i^2+2(i+j)^2+(j-k)^2+k^2}}{(q^4;q^4)_i(q^4;q^4)_j(q^4;q^4)_k}\nonumber\\
    &=-(q^4;q^4)_\infty \sum_{i,j,k\geq 0}\frac{q^{4i^2+4ij+3j^2-2jk+2k^2}}{(q^4;q^4)_i(q^4;q^4)_j(q^4;q^4)_k}=-J_4T(1,1,1;q). \qedhere
\end{align*}
\end{proof}

We are now ready to prove the first three identitise in Theorem~\ref{thm-ABCD}. The argument relies on Theorem~\ref{thm-ex9-new}, whose proof is deferred to Section~\ref{sec-rank3}.
\begin{proof}[Proof of \eqref{dual-9-eq1}--\eqref{dual-9-eq3}]
By Lemma \ref{dual-ex9-lemma1} and Theorem \ref{thm-ex9-new}, we have
\begin{align}
    &b(A^{(1)}(q)+C^{(1)}(q)-2B^{(1)}(q))+a(2C^{(1)}(q)-B^{(1)}(q)-D^{(1)}(q)) \nonumber\\
    &=\frac{qJ_4^2J_{4,72}J_{28,72}J_{32,72}J_{36,72}}{J_2J_8J_{14,72}J_{18,72}J_{22,72}}+\frac{2q^6J_8J_{8,72}}{J_4}  \nonumber \\
    &=R_{1}^{(1)}(q)+R_3^{(1)}(q)-2R_2^{(1)}(q), \label{dual-9-ex1-eq0}\\
    &b(A^{(2)}(q)+C^{(2)}(q)-2B^{(2)}(q))+a (2C^{(2)}(q)-B^{(2)}(q)-D^{(2)}(q))\nonumber\\
    &=-\frac{J_4^2J_{24}J_{12,24}}{J_2J_8J_{6,24}}-\frac{2qJ_8J_{24}}{J_4}  \nonumber \\
    &=R_1^{(2)}(q)+R_3^{(2)}(q)-2R_2^{(2)}(q),\label{dual-9-ex2-eq0}\\
    &b(A^{(3)}(q)+C^{(3)}(q)-2B^{(3)}(q))+a(2C^{(3)}(q)-B^{(3)}(q)-D^{(3)}(q)) \nonumber\\
    &=\frac{qJ_4^2J_{4,72}J_{16,72}J_{20,72}J_{36,72}}{J_2J_8J_{2,72}J_{18,72}J_{34,72}}+\frac{2J_8J_{32,72}}{J_4} \nonumber \\
     &=R_1^{(3)}(q)+R_3^{(3)}(q)-2R_2^{(3)}(q),\label{dual-9-ex3-eq0}\\
    &b(A^{(4)}(q)+C^{(4)}(q)-2B^{(4)}(q))+a(2C^{(4)}(q)-B^{(4)}(q)-D^{(4)}(q))\nonumber\\
    &=-\frac{J_4^2J_{8,72}J_{20,72}J_{28,72}J_{36,72}}{J_2J_8J_{10,72}J_{18,72}J_{26,72}}-\frac{2q^3J_8J_{16,72}}{J_4} \nonumber \\ 
     &=R_1^{(4)}(q)+R_3^{(4)}(q)-2R_2^{(4)}(q). \label{dual-9-ex4-eq0}
\end{align}
Note that the series $a(q)\in \mathbb{Z}[[q^4]]$ and $b(q)\in q\mathbb{Z}[[q^4]]$ and we have the classification of $R_i^{(h)}(q)$, $A^{(h)}(q)$, $B^{(h)}(q)$, $C^{(h)}(q)$ and $D^{(h)}(q)$ in Table \ref{tab-ABCD}.
Extracting the terms of the form $q^{4n+r}$ ($r=0,1,2,3$) in both sides of  \eqref{dual-9-ex1-eq0}--\eqref{dual-9-ex4-eq0}, we obtain \eqref{dual-9-eq1}--\eqref{dual-9-eq3} for $h=1,2,3,4$, respectively.
\end{proof}

We have attempted to prove \eqref{dual-9-eq4} by adapting the above arguments, but have so far been unsuccessful. To illustrate the difficulty, we briefly consider the case $h=1$. A natural way of obtaining the combination $b\cdot D^{(1)}(q)-a\cdot A^{(1)}(q)$ is 
    \begin{align}\label{Zq-1}
    &\mathcal{Z}(q):=\sum_{i,j,k\geq 0}\frac{(-1)^{j+k}q^{(i-j)^2+(j-k)^2+k^2}}{(q^4;q^4)_{i}(q^4;q^4)_{j}(q^4;q^4)_{k}}\sum_{s\in\mathbb{Z}}q^{(2s-k)^2}\nonumber\\
    &=a\big(A^{(1)}(q)-C^{(1)}(q)\big)+b\big(B^{(1)}(q)-D^{(1)}(q)\big)\nonumber\\
    &=aA^{(1)}(q)-bD^{(1)}(q)+R_2^{(1)}.
  \end{align}
Therefore, an infinite-product evaluation of $\mathcal{Z}(q)$ would immediately imply \eqref{dual-9-eq4}. Note that Corollary \ref{cor-ABCD-difference} give evaluations for $A^{(h)}(q)-C^{(h)}(q)$ and $B^{(h)}(q)-D^{(h)}(q)$. Substituting them into \eqref{Zq-1} will provide product representations for $\mathcal{Z}(q)$. However,  we are unable to evaluate $\mathcal{Z}(q)$ directly without relying on the first approach.

Based on Theorem \ref{thm-ABCD}, we are able to express the Nahm sums dual to Zagier's Example 9 in the following way.
\begin{theorem}\label{thm-dual-9}
Let $M^{(h)}(q)$ be defined in \eqref{M-defn}. We have for $h=1,2,3,4$ that
 \begin{align}\label{eq-thm-exam9}
     &M^{(h)}(q)=-\frac{J_2^5J_8^4}{J_1^2J_4^{10}}\Big((3a^2+b^2)R_{1}^{(h)}+(3a^2-2ab+3b^2)R_{2}^{(h)}\nonumber\\
     &\quad +(a^2+3b^2)R_{3}^{(h)}+(a+b)^2R_{4}^{(h)}\Big)
 \end{align}
where $R_{i}^{(h)}$ ($1\leq i,h\leq 4$) are defined in \eqref{R11}--\eqref{R44}.
\end{theorem}
\begin{proof}
By \eqref{dual-9-eq1}, \eqref{dual-9-eq2}, \eqref{dual-9-eq3} and \eqref{dual-9-eq4}, we have
\begin{align*}
    A^{(h)}(q)&=\frac{1}{b^4-a^4}\big(b^3R_{1}^{(h)}+ab^2R_{2}^{(h)}+a^2bR_{3}^{(h)}+a^3R_{4}^{(h)}\big),\\
    B^{(h)}(q)&=\frac{1}{b^4-a^4}\big(a^3R_{1}^{(h)}+b^3R_{2}^{(h)}+ab^2R_{3}^{(h)}+a^2bR_{4}^{(h)}\big),\\
    C^{(h)}(q)&=\frac{1}{b^4-a^4}\big(a^2bR_{1}^{(h)}+a^3R_{2}^{(h)}+b^3R_{3}^{(h)}+ab^2R_{4}^{(h)}\big),\\
    D^{(h)}(q)&=\frac{1}{b^4-a^4}\big(ab^2R_{1}^{(h)}+a^2bR_{2}^{(h)}+a^3R_{3}^{(h)}+b^3R_{4}^{(h)}\big).
\end{align*}
Substituting them into \eqref{M-ABCD}, we deduce that
\begin{align}\label{proof-eq-thm-exam9}
     &M^{(h)}(q)=\frac{1}{b^4-a^4}\Big((3a^3+3a^2b+ab^2+b^3)R_{1}^{(h)}+(3a^3+a^2b+ab^2+3b^3)R_{2}^{(h)}\nonumber\\
     &\quad +(a^3+a^2b+3ab^2+3b^3)R_{3}^{(h)}+(a^3+3a^2b+3ab^2+b^3)R_{4}^{(h)}\Big).
 \end{align}
The common factor $a+b$ can be  canceled  from the numerator and denominator. 
Substituting \eqref{eq-ab-sum} into \eqref{proof-eq-thm-exam9}, we obtain \eqref{eq-thm-exam9}.
\end{proof}
Using the Maple approach in \cite{Frye-Garvan}, it is easy to show that Theorems \ref{thm-dual-exam9-1st} and  \ref{thm-dual-9} are equivalent. While both of them are sufficient to justify the modularity of the dual of Example 9, the representations in \eqref{eq-thm-Mh} and \eqref{eq-thm-exam9} are a bit complicated compared with other examples.  It is meaningful if one can find simpler formulas. Using Maple we find the following elegant formula for $M^{(2)}(q)$:
\begin{align}\label{M2-simple}
M^{(2)}(q)=3\frac{J_2J_3J_6^2}{J_1J_4^3}.
\end{align}
This can be proved from \eqref{eq-thm-Mh} (or equivalently \eqref{eq-thm-exam9}) using the approach in \cite{Frye-Garvan}. It remains interesting to find the minimal number of infinite products needed to express $M^{(1)}(q)$, $M^{(3)}(q)$ and $M^{(4)}(q)$.

\subsection{Dual of Example 10}
We list the modular triples for this example and their duals in Table \ref{tab-exam10}. Due to the symmetry of $n_2$ and $n_3$ in the quadratic form $n^{\mathrm{T}}An$, there are essentially only five different Nahm sums.
{\small
\begin{table}[htbp]
\centering
\caption{Modular triples for Example 10 and their duals}\label{tab-exam10}
{\tiny
\begin{tabular}{c|cccccccc}
  \hline
    \padedvphantom{I}{4ex}{4ex}
  $A$ &  \multicolumn{8}{c}{$\begin{pmatrix} 4 & 2 & 2 \\ 2 & 2 & 1 \\ 2 & 1 & 2 \end{pmatrix}$}  \\
  \hline
\padedvphantom{I}{4ex}{4ex}
$B$ & $\begin{pmatrix} 0 \\ -1/2 \\ 0 \end{pmatrix}$ & $\begin{pmatrix} 0 \\ 0 \\ -1/2 \end{pmatrix}$
& $\begin{pmatrix}  0 \\ 0 \\ 0 \end{pmatrix}$ & $\begin{pmatrix} 1 \\ 0 \\ 1/2 \end{pmatrix}$ & $\begin{pmatrix} 1 \\ 0 \\ 1 \end{pmatrix}$  & $\begin{pmatrix} 1 \\ 1/2 \\ 0 \end{pmatrix}$ & $\begin{pmatrix} 1 \\ 1 \\ 0 \end{pmatrix}$ & $\begin{pmatrix} 2 \\ 1 \\ 1 \end{pmatrix}$ \\
  ~~ $C$ & $-1/120$ & $-1/120$ & $-1/30$ & $11/120$ & $1/6$ & $11/120$ & $1/6$ & $11/30$ \\
  \hline
\padedvphantom{I}{4ex}{4ex}
  $\mathcal{D}(A)$ &  \multicolumn{8}{c}{$\begin{pmatrix} 3/4 & -1/2 & -1/2  \\ -1/2 & 1 & 0 \\ -1/2 & 0 & 1 \end{pmatrix}$}  \\
  \hline
\padedvphantom{I}{4ex}{4ex}
$\mathcal{D}(B)$ & $\begin{pmatrix}1/4 \\ -1/2 \\ 0 \end{pmatrix}$ & $\begin{pmatrix} 1/4 \\ 0 \\ -1/2 \end{pmatrix}$ & $\begin{pmatrix} 0 \\ 0 \\ 0 \end{pmatrix}$ & $\begin{pmatrix} 1/2 \\ -1/2 \\ 0 \end{pmatrix}$ & $\begin{pmatrix} 1/4 \\ -1/2 \\ 1/2 \end{pmatrix}$ & $\begin{pmatrix} 1/2 \\ 0 \\ -1/2 \end{pmatrix}$
 & $\begin{pmatrix} 1/4 \\ 1/2 \\ -1/2 \end{pmatrix}$ & $\begin{pmatrix} 1/2 \\ 0 \\ 0 \end{pmatrix}$ \\
  ~~ $\mathcal{D}(C)$ & $1/120$ & $1/120$ & $-11/120$ & $1/30$ & $1/12$ & $1/30$ & $1/12$ & $1/120$ \\
  \hline
\end{tabular}
}
\end{table}
}

The modularity of Example 10 has been proved by Wang \cite[Theorem 4.12]{Wang-rank3}. 
As for its dual, we establish the following identities to prove its modularity.
\begin{theorem}\label{thm-dual-10}
We have
\begin{align}
&\sum_{i,j,k\geq 0} \frac{q^{3i^2+4j^2+4k^2-4ij-4ik+2i-4j}}{(q^8;q^8)_i(q^8;q^8)_j(q^8;q^8)_k}
=2\frac{J_{2,8}J_{10}^2J_{2,20}}{J_2J_{20}J_{1,10}J_{4,20}}, \label{id-dual-10-1} \\
&\sum_{i,j,k\geq 0} \frac{q^{3i^2+4j^2+4k^2-4ij-4ik}}{(q^8;q^8)_i(q^8;q^8)_j(q^8;q^8)_k}  \nonumber \\ &=\frac{J_{80}^{11}J_{8,80}J_{24,80}^2J_{40,80}^2}{J_{4,80}^3J_{12,80}^2J_{20,80}^4J_{28,80}^2J_{32,80}^2J_{36,80}^3} +4q^3\frac{J_{80}^{11}}{J_{8,80}^4J_{16,80}J_{24,80}^3J_{32,80}J_{40,80}^2},  \label{id-dual-10-2}\\
&\sum_{i,j,k\geq 0} \frac{q^{3i^2+4j^2+4k^2-4ij-4ik+4i-4j}}{(q^8;q^8)_i(q^8;q^8)_j(q^8;q^8)_k}
=2\frac{J_{40}^{11/2}J_{6,40}J_{10,40}J_{14,40}}{J_{3,40}J_{4,40}J_{7,40}J_{8,40}J_{12,40}^2J_{13,40}J_{17,40}J_{20,40}^{1/2}},  \label{id-dual-10-3} \\
&\sum_{i,j,k\geq 0} \frac{q^{3i^2+4j^2+4k^2-4ij-4ik+2i-4j+4k}}{(q^8;q^8)_i(q^8;q^8)_j(q^8;q^8)_k} \nonumber \\
&=2\frac{J_{80}^{11}}{J_{8,80}^4J_{16,80}J_{24,80}^4J_{32,80}J_{40,80}}  +q\frac{J_{80}^{11}J_{8,80}^2J_{24,80}^2J_{40,80}}{J_{4,80}^3J_{12,80}^3J_{16,80}J_{20,80}^2J_{28,80}^3J_{32,80}J_{36,80}^3},  \label{id-dual-10-4} \\
&\sum_{i,j,k\geq 0} \frac{q^{3i^2+4j^2+4k^2-4ij-4ik+4i}}{(q^8;q^8)_i(q^8;q^8)_j(q^8;q^8)_k} \nonumber \\
&=\frac{J_{8,80}^2J_{80}^{11}J_{24,80}J_{40,80}^2}{J_{4,80}^2J_{12,80}^3J_{16,80}^2J_{20,80}^4J_{28,80}^3J_{36,80}^2}  +4q^7\frac{J_{80}^{11}}{J_{8,80}^3J_{16,80}J_{24,80}^4J_{32,80}J_{40,80}^2}. \label{id-dual-10-5}
\end{align}
\end{theorem}
\begin{proof}
We define
\begin{align}\label{dual-10-h}
    &F(u,v,w;q):=\sum_{i,j,k\geq 0}\frac{q^{3i^2+4j^2+4k^2-4ij-4ik}u^iv^jw^k}{(q^8;q^8)_i(q^8;q^8)_j(q^8;q^8)_k}\nonumber\\
    &=\sum_{i\geq 0}\frac{q^{3i^2}u^i(-q^{4-4i}v;q^8)_{\infty}(-q^{4-4i}w;q^8)_{\infty}}{(q^8;q^8)_i}\nonumber\\
    &=(-q^4v,-q^4w;q^8)_{\infty}S_0(u,v,w;q)+(-v,-w;q^8)_{\infty}S_1(u,v,w;q),
\end{align}
where
\begin{align}
    S_0(u,v,w;q)&:=\sum_{i\geq 0}\frac{q^{4i^2}u^{2i}v^iw^i(-q^4v^{-1},-q^4w^{-1};q^8)_i}{(q^8;q^8)_{2i}},\\
    S_1(u,v,w;q)&:=\sum_{i\geq 0}\frac{q^{4i^2+4i+3}u^{2i+1}v^iw^i(-q^8v^{-1},-q^8w^{-1};q^8)_i}{(q^8;q^8)_{2i+1}}.
\end{align}

(1) Let $(u,v,w)=(q^2,q^{-4},1)$. By \eqref{S79} we have
\begin{align}\label{dual-10-ex1-h0}
    &S_0(q^2,q^{-4},1;q)=\sum_{i\geq 0}\frac{q^{4i^2}(-q^4,-q^8;q^8)_{i}}{(q^8;q^8)_{2i}}=\sum_{i\geq 0}\frac{q^{4i^2}}{(q^4;q^4)_{2i}}=\frac{J_8J_{80}}{J_4J_{16,80}}.
\end{align}
By \eqref{S94} we have
\begin{align}\label{dual-10-ex1-h1}
    &S_1(q^2,q^{-4},1;q)=\sum_{i\geq 0}\frac{q^{4i^2+4i+5}(-q^8,-q^{12};q^8)_i}{(q^8;q^8)_{2i+1}}=\frac{q^5}{1+q^4}\sum_{i\geq 0}\frac{q^{4i^2+4i}}{(q^4;q^4)_{2i+1}}\nonumber\\
    &=\frac{q^5}{1+q^4}\frac{J_{12,40}J_{16,80}}{J_4J_{80}}.
\end{align}
Substituting \eqref{dual-10-ex1-h0} and \eqref{dual-10-ex1-h1} into \eqref{dual-10-h}, we obtain \eqref{id-dual-10-1}.

(2) Let $(u,v,w)=(1,1,1)$. By \eqref{dual-10-ex2-h0} and \eqref{S45} we have
\begin{align}
    S_0(1,1,1;q)&=\sum_{i\geq 0}\frac{q^{4i^2}(-q^4;q^8)_i^2}{(q^8;q^8)_{2i}}=\frac{J_8J_{40}^2J_{12,80}J_{16,80}J_{28,80}}{J_4J_{16}J_{80}J_{8,80}J_{20,80}J_{32,80}}, \label{proof-exam10-2-H0} \\
    S_1(1,1,1;q)&=\sum_{i\geq 0}\frac{q^{4i^2+4i+3}(-q^8;q^8)_i^2}{(q^8;q^8)_{2i+1}}=q^3\frac{J_{16}J_{24,80}}{J_8^2}. \label{proof-exam10-2-H1}
\end{align}
Substituting \eqref{proof-exam10-2-H0} and \eqref{proof-exam10-2-H1} into \eqref{dual-10-h}, we obtain \eqref{id-dual-10-2}.

(3) Let $(u,v,w)=(q^4,q^{-4},1)$. By \eqref{S99} and \eqref{S96}  we have
\begin{align}
    S_0(q^4,q^{-4},1;q)&=\sum_{i\geq 0}\frac{q^{4i^2+4i}(-q^4,-q^8;q^8)_i}{(q^8;q^8)_{2i}}=\sum_{i\geq 0}\frac{q^{4i^2+4i}}{(q^4;q^4)_{2i}} \nonumber \\
    &=\frac{J_{4,80}J_{32,80}J_{36,80}J_{40}}{J_4J_{80}^3}, \label{dual-10-ex3-h0} \\
    S_1(q^4,q^{-4},1;q)&=\sum_{i\geq 0}\frac{q^{4i^2+8i+7}(-q^8,-q^{12};q^8)_i}{(q^8;q^8)_{2i+1}}=\frac{q^7}{1+q^4}\sum_{i\geq 0}\frac{q^{4i^2+8i}}{(q^4;q^4)_{2i+1}}\nonumber\\
    &=\frac{q^7J_8J_{80}}{(1+q^4)J_4J_{32,80}}. \label{dual-10-ex3-h1}
\end{align}

Substituting \eqref{dual-10-ex3-h0} and \eqref{dual-10-ex3-h1} into \eqref{dual-10-h}, we obtain \eqref{id-dual-10-3}.

(4) Let $(u,v,w)=(q^2,q^{-4},q^{4})$. By \eqref{Rogers3303} and \eqref{MSZ2.7} we have
\begin{align}
    S_0(q^2,q^{-4},q^4;q)&=\sum_{i\geq 0}\frac{q^{4i^2+4i}(-1,-q^8;q^8)_i}{(q^8;q^8)_{2i}}=\frac{J_{40}J_{80}^2}{J_8J_{8,80}J_{24,80}}, \label{dual-10-ex4-h0} \\
    S_1(q^2,q^{-4},q^4;q)&=\sum_{i\geq 0}\frac{q^{4i^2+8i+5}(-q^4,-q^{12};q^8)_i}{(q^8;q^8)_{2i+1}} \nonumber \\
     &=\frac{q^5J_{80}^6}{(1+q^4)J_{4,80}J_{12,80}J_{16,80}J_{28,80}J_{32,80}J_{36,80}}.\label{dual-10-ex4-h1}
\end{align}
Substituting \eqref{dual-10-ex4-h0} and \eqref{dual-10-ex4-h1} into \eqref{dual-10-h}, we  obtain \eqref{id-dual-10-4}.

(5) Let $(u,v,w)=(q^4,1,1)$. By \eqref{BMS2.17} and \eqref{S43} we have
\begin{align}
    S_0(q^4,1,1;q)&=\sum_{i\geq 0}\frac{q^{4i^2+8i}(-q^4;q^8)_i^2}{(q^8;q^8)_{2i}}=\frac{J_{40}^2J_{80}^5}{J_{12,80}J_{16,80}^2J_{20,80}^2J_{24,80}J_{28,80}},\label{dual-10-ex5-h0}\\
    S_1(q^4,1,1;q)&=\sum_{i\geq 0}\frac{q^{4i^2+12i+7}(-q^8;q^8)_i^2}{(q^8;q^8)_{2i+1}}=\frac{q^7J_{80}^3}{J_8J_{40}J_{24,80}}.\label{dual-10-ex5-h1}
\end{align}
Substituting \eqref{dual-10-ex5-h0} and \eqref{dual-10-ex5-h1} into \eqref{dual-10-h}, we obtain \eqref{id-dual-10-5}.
\end{proof}

\subsection{Dual of Example 11}
We list Zagier's Example 11 and its dual in Table \ref{tab-exam11}. 
{\small
\begin{table}[htbp]
\centering
\caption{Modular triples for Example 11 and their duals}\label{tab-exam11}
\begin{tabular}{c|ccccc}
  \hline
    \padedvphantom{I}{4ex}{4ex}
  $A$ &  \multicolumn{5}{c}{$\begin{pmatrix} 4 & 2 & -1 \\ 2 & 2 & -1 \\ -1 & -1 & 1 \end{pmatrix}$}  \\
  \hline
\padedvphantom{I}{4ex}{4ex}
$B$ & $\begin{pmatrix} 0 \\ 0 \\ 0 \end{pmatrix}$ & $\begin{pmatrix} 0 \\ 0 \\ 1/2 \end{pmatrix}$ &
$\begin{pmatrix} 1 \\ 0 \\ 0 \end{pmatrix}$ & $\begin{pmatrix} 2 \\ 1 \\ -1/2 \end{pmatrix}$ &
$\begin{pmatrix} 2 \\ 1 \\ 0 \end{pmatrix}$ \\
  ~~ $C$ & $-1/16$ & $1/24$ & $5/48$ & $3/8$ & $7/16$ \\
  \hline
\padedvphantom{I}{4ex}{4ex}
  $\mathcal{D}(A)$ &  \multicolumn{5}{c}{$\begin{pmatrix} 1/2 & -1/2 & 0 \\ -1/2 & 3/2 & 1 \\ 0 & 1 & 2 \end{pmatrix}$}  \\
  \hline
\padedvphantom{I}{4ex}{4ex}
$\mathcal{D}(B)$ & $\begin{pmatrix} 0 \\ 0 \\ 0 \end{pmatrix}$ &
$\begin{pmatrix} 0 \\ 1/2 \\ 1 \end{pmatrix}$ &
$\begin{pmatrix} 1/2 \\ -1/2 \\ 0 \end{pmatrix}$ &
$\begin{pmatrix} 1/2 \\ 0 \\ 0 \end{pmatrix}$ & 
$\begin{pmatrix} 1/2 \\ 1/2 \\ 1 \end{pmatrix}$ \\
  ~~ $\mathcal{D}(C)$ & $-1/16$ & $1/12$ & $1/48$ & $0$ & $3/16$ \\
  \hline
\end{tabular}
\end{table}
}

The modularity of Zagier's Example 11 has been confirmed by Wang \cite[Theorem 4.14]{Wang-rank3}.
As for its dual, Shi and Wang \cite[Theorem 1.3]{Shi-Wang} discovered the identities \eqref{dual-1}--\eqref{dual-2} to prove its modularity. Note that \eqref{dual-2} was proposed as a conjecture \cite[Conjecture 1.4]{Shi-Wang}. Here we prove it. 
\begin{proof}[Proof of Theorem \ref{thm-dual-2}]
Firstly, we prove that
\begin{align}\label{dual-ex11-eq1}
    F_1(q):=\sum_{m,n,r\geq 0}\frac{q^{4m^2+3n^2+r^2+4mn-2nr+4m+2n}}{(q^4;q^4)_m(q^4;q^4)_n(q^4;q^4)_r}\sum_{s\in\mathbb{Z}}q^{(2s+r+1-n)^2}=3q\frac{J_8J_{12}^3}{J_4^3}.
\end{align}

Using the rank reduction formula of rank 4 \cite[Theorem 1.2]{Shi-Wang}, we have
\begin{align}\label{proof-exam11-rank-4}
    &\sum_{i,j,k,l\ge 0}\frac{q^{\frac{i^2+(i-j)^2+(j-k)^2+(k-l)^2}{2}+i-j+k-\frac{l}{2}}}{(q;q)_i(q;q)_j(q;q)_k(q;q)_l}\nonumber\\&=\frac{(-1;q)_{\infty}}{(q;q)_{\infty}}\sum_{m,n,r\geq 0}\frac{q^{m^2+\frac{3}{4}n^2+\frac{1}{4}r^2+mn-\frac{1}{2}nr+m+\frac{1}{2}r}}{(q;q)_m(q;q)_n(q;q)_r}\sum_{s\in\mathbb{Z}}q^{(s-\frac{n-r}{2})^2+s}\nonumber\\
    &=\frac{q^{-\frac{1}{4}}(-1;q)_{\infty}}{(q;q)_{\infty}}\sum_{m,n,r\geq 0}\frac{q^{m^2+\frac{3}{4}n^2+\frac{1}{4}r^2+mn-\frac{1}{2}nr+m+\frac{1}{2}n}}{(q;q)_m(q;q)_n(q;q)_r}\sum_{s\in\mathbb{Z}}q^{(s+\frac{r+1-n}{2})^2} \nonumber \\
    &=\frac{q^{-\frac{1}{4}}(-1;q)_{\infty}}{(q;q)_{\infty}} F_1(q^{\frac{1}{4}}).
\end{align}
Suppose that $z_1z_2=1$. Recall the rank reduction formula of rank 3 \cite[Theorem 1.2]{Shi-Wang}: 
\begin{align}\label{rank-reducing-3}
    \sum_{i,j,k\geq 0}\frac{q^{\frac{i^2+(i-j)^2+(j-k)^2}{2}}z_1^iz_2^jz_3^k}{(q;q)_i(q;q)_j(q;q)_k}=\frac{(-q^{\frac{1}{2}}z_3;q)_{\infty}}{(q;q)_{\infty}}\sum_{n\geq 0}\frac{q^{\frac{n^2}{4}}(z_2z_3)^n}{(q;q)_n}\sum_{s\in\mathbb{Z}}q^{(s+\frac{n}{2})^2}z_1^{-s}.
\end{align}
We have
\begin{align}
    &\chi(z_1,z_2,z_3,z_4;q):=\sum_{i,j,k,l\ge 0}\frac{q^{\frac{i^2+(i-j)^2+(j-k)^2+(k-l)^2}{2}}z_1^iz_2^jz_3^kz_4^l}{(q;q)_i(q;q)_j(q;q)_k(q;q)_l}\nonumber\\
    &=\mathrm{CT}_y\Big[\sum_{i,j,k\geq 0}\frac{q^{\frac{i^2+(i-j)^2+(j-k)^2}{2}}z_1^iz_2^jz_3^ky^k}{(q;q)_i(q;q)_j(q;q)_k} \cdot\sum_{t\in\mathbb{Z}}q^{\frac{t^2}{2}}y^{-t}\sum_{l\geq 0}\frac{z_4^ly^{-l}}{(q;q)_l}\Big]\nonumber\\
    &=\mathrm{CT}_y\Big[\frac{(-q^{\frac{1}{2}}z_3y;q)_{\infty}}{(q;q)_{\infty}}\sum_{n\geq 0}\frac{q^{\frac{n^2}{4}}(z_2z_3y)^n}{(q;q)_n}\sum_{s\in\mathbb{Z}}q^{(s+\frac{n}{2})^2}z_1^{-s}\sum_{t\in\mathbb{Z}}q^{\frac{t^2}{2}}y^{-t}\sum_{l\geq 0}\frac{z_4^ly^{-l}}{(q;q)_l}\Big]\nonumber\\
    &=\frac{1}{(q;q)_{\infty}}\mathrm{CT}_y\Big[\sum_{m\geq 0}\frac{q^{\frac{m^2}{2}}z_3^my^m}{(q;q)_m}\sum_{n\geq 0}\frac{q^{\frac{n^2}{4}}(z_2z_3y)^n}{(q;q)_n}\sum_{s\in\mathbb{Z}}q^{(s+\frac{n}{2})^2}z_1^{-s}\sum_{t\in\mathbb{Z}}q^{\frac{t^2}{2}}y^{-t}\sum_{l\geq 0}\frac{z_4^ly^{-l}}{(q;q)_l}\Big]\nonumber\\
    &=\frac{1}{(q;q)_{\infty}}\sum_{m,n,r\geq 0}\frac{q^{m^2+\frac{3}{4}n^2+\frac{1}{2}r^2+mn-mr-nr}z_3^m(z_2z_3)^nz_4^r}{(q;q)_m(q;q)_n(q;q)_r}\sum_{s\in\mathbb{Z}}q^{(s+\frac{n}{2})^2}z_1^{-s}.\label{dual-ex11-reducingrank}
\end{align}
Setting $(z_1,z_2,z_3,z_4)=(q,q^{-1},q,q^{-\frac{1}{2}})$, we have
\begin{align}
    &\sum_{i,j,k,l\ge 0}\frac{q^{\frac{i^2+(i-j)^2+(j-k)^2+(k-l)^2}{2}+i-j+k-\frac{l}{2}}}{(q;q)_i(q;q)_j(q;q)_k(q;q)_l}\nonumber\\
    &=\frac{1}{(q;q)_{\infty}}\sum_{m,n,r\geq 0}\frac{q^{m^2+\frac{3}{4}n^2+\frac{1}{2}r^2+mn-mr-nr+m-\frac{r}{2}}}{(q;q)_m(q;q)_n(q;q)_r}\sum_{s\in\mathbb{Z}}q^{(s+\frac{n}{2})^2-s}\nonumber\\
    &=\frac{q^{-\frac{1}{4}}}{(q;q)_{\infty}}\sum_{m,n,r\geq 0}\frac{q^{m^2+\frac{3}{4}n^2+\frac{1}{2}r^2+mn-mr-nr+m+\frac{n}{2}-\frac{r}{2}}}{(q;q)_m(q;q)_n(q;q)_r}\sum_{s\in\mathbb{Z}}q^{(s+\frac{n-1}{2})^2}.\label{dual-11-part1-0}
\end{align}
Recall the identity \eqref{id-dual-8-3}:
\begin{align}\label{dual-11-part1-1}
   W(q)&:=\sum_{m,n,r\geq 0}\frac{q^{4m^2+3n^2+2r^2+4mn-4mr-4nr+4m+2n-2r}}{(q^4;q^4)_m(q^4;q^4)_n(q^4;q^4)_r} \nonumber \\
   &=\frac{2J_{24}^6J_{2,24}J_{6,24}J_{10,24}}{J_{1,24}J_{4,24}^3J_{5,24}J_{7,24}J_{8,24}J_{11,24}J_{12,24}}.
\end{align}
Replacing $q$ by $-q$, we obtain
\begin{align}\label{dual-11-part1-2}
   W(-q)&=\sum_{m,n,r\geq 0}\frac{(-1)^nq^{4m^2+3n^2+2r^2+4mn-4mr-4nr+4m+2n-2r}}{(q^4;q^4)_m(q^4;q^4)_n(q^4;q^4)_r} \nonumber \\
   &=\frac{2J_{1,24}J_{5,24}J_{6,24}J_{7,24}J_{11,24}}{J_{4,24}^3J_{8,24}J_{12,24}}.
\end{align}
Substituting \eqref{dual-11-part1-1} and \eqref{dual-11-part1-2} into \eqref{dual-11-part1-0}, we have
\begin{align}
    F_1(q)&=\frac{q(q^4;q^4)_{\infty}}{(-1;q^4)_{\infty}}\sum_{i,j,k,l\geq 0}\frac{q^{2i^2+2(i-j)^2+2(j-k)^2+2(k-l)^2+4i-4j+4k-2l}}{(q^4;q^4)_i(q^4;q^4)_j(q^4;q^4)_k(q^4;q^4)_l}  \quad \text{(by \eqref{proof-exam11-rank-4})}\nonumber\\
    &=\frac{1}{(-1;q^4)_{\infty}}\sum_{m,n,r\geq 0}\frac{q^{4m^2+3n^2+2r^2+4mn-4mr-4nr+4m+2n-2r}}{(q^4;q^4)_m(q^4;q^4)_n(q^4;q^4)_r}\sum_{s\in\mathbb{Z}}q^{(2s+n-1)^2}\nonumber\\
    &=\frac{1}{(-1;q^4)_{\infty}}\Big(a(q)\frac{W(q)-W(-q)}{2}+b(q) \frac{W(q)+W(-q)}{2}  \Big) \nonumber \\
    &=3q\frac{J_8J_{12}^3}{J_{4}^3}.\label{dual-11-3to4to3}
\end{align}
Here the for the last equality we used the Maple approach in \cite{Frye-Garvan} to verify theta function identities.

Secondly, we prove that
\begin{align}\label{dual-ex11-eq2}
    F_2(q):=\sum_{m,n,r\geq 0}\frac{q^{4m^2+3n^2+r^2+4mn-2nr+4m+2n}}{(q^4;q^4)_m(q^4;q^4)_n(q^4;q^4)_r}\sum_{s\in\mathbb{Z}}q^{(2s+r-n)^2}=\frac{J_{24}^6}{J_{2,24}^2J_{4,24}J_{10,24}^2}.
\end{align}
We first write the left hand side as
\begin{align}\label{add-F2-start}
    F_2(q)&=\sum_{m,n,r\geq 0}\frac{q^{4m^2+3n^2+r^2+4mn-2nr+4m+2n}}{(q^4;q^4)_m(q^4;q^4)_n(q^4;q^4)_r}\sum_{s\in\mathbb{Z}}q^{(2s-(2m-r+n))^2}\nonumber\\
    &=\sum_{\substack{m,n,r\geq 0 \\ s,t\in\mathbb{Z}\\ s+t=2m-r+n}}\frac{q^{4mr+2n^2+4m+2n+2s^2+2t^2}}{(q^4;q^4)_m(q^4;q^4)_n(q^4;q^4)_r}\nonumber\\
    &=\mathrm{CT}_{z}\Big[ \sum_{m,n,r\geq 0}\frac{q^{4mr+2n^2+4m+2n}z^{2m-r+n}}{(q^4;q^4)_m(q^4;q^4)_n(q^4;q^4)_r}\sum_{s,t\in\mathbb{Z}}q^{2s^2+2t^2}z^{-s-t} \Big].
\end{align}

Using  \eqref{Euler} and \eqref{q-binomial}, we have
\begin{align}\label{add-F2-mid}
    &\sum_{m,r\geq 0}\frac{q^{4mr+4m}z^{2m-r}}{(q^4;q^4)_m(q^4;q^4)_r}=\sum_{m\geq 0}\frac{q^{4m}z^{2m}}{(q^4;q^4)_m(q^{4m}z^{-1};q^4)_{\infty}}\nonumber\\
    &=\frac{1}{(z^{-1};q^4)_{\infty}}\sum_{m\geq 0}\frac{(z^{-1};q^4)_m}{(q^4;q^4)_m}(q^4z^2)^m=\frac{(q^4z;q^4)_{\infty}}{(z^{-1};q^4)_{\infty}(q^4z^2;q^4)_{\infty}}.
\end{align}
Substituting \eqref{add-F2-mid} into \eqref{add-F2-start}, we have
\begin{align}
    F_2(q)&=\mathrm{CT}_{z} \frac{(q^4z;q^4)_{\infty}(-q^4z;q^4)_{\infty}(q^4,-q^2z,-q^2z^{-1};q^4)_{\infty}^2}{(z^{-1};q^4)_{\infty}(q^4z^2;q^4)_{\infty}} \nonumber\\
    &=\mathrm{CT}_{z}\frac{(q^4,-q^2z,-q^2z^{-1};q^4)_{\infty}^2}{(z^{-1};q^4)_{\infty}(q^4z^2;q^8)_{\infty}} \nonumber\\
    &=\mathrm{CT}_{z} \sum_{m,n\geq 0}\frac{q^{4n}z^{2n-m}}{(q^4;q^4)_m(q^8;q^8)_n}\sum_{s,t\in\mathbb{Z}}q^{2s^2+2t^2}z^{-s-t}\nonumber\\
    &=\sum_{m,n\geq 0}\frac{q^{(2n-m)^2+4n}}{(q^4;q^4)_m(q^8;q^8)_n}\sum_{s\in\mathbb{Z}}q^{(2s+m)^2}\label{dual-11-ct}.
\end{align}

Now we define
\begin{align}
    Z_0(q):=\sum_{m,n\geq 0}\frac{q^{(2n-2m)^2+4n}}{(q^4;q^4)_{2m}(q^8;q^8)_n}, \quad   Z_1(q):=\sum_{m,n\geq 0}\frac{q^{(2n-2m-1)^2+4n}}{(q^4;q^4)_{2m+1}(q^8;q^8)_n}.
\end{align}
Using \eqref{Euler} and \eqref{Heine}, we have
\begin{align}
    Z_0(q)&=\sum_{m\geq 0}\frac{q^{4m^2}(-q^{8-8m};q^8)_{\infty}}{(q^4;q^4)_{2m}}=(-q^8;q^8)_{\infty}\sum_{m\geq 0}\frac{(-1;q^8)_{m}q^{4m}}{(q^4;q^4)_{2m}} \nonumber \\
&=\frac{(-q^4;q^4)_{\infty}}{(q^4;q^8)_{\infty}^2}\sum_{n\geq 0}\frac{(q^4;q^8)_n(-q^4)^{n^2}}{(-q^4;-q^4)_{2n}}. \label{dual-11-z0}
\end{align}
Substituting \eqref{S29} with $q$ replaced by $-q^4$ into \eqref{dual-11-z0}, we deduce that
\begin{align}\label{dual-11-z0-1}
    Z_0(q)=\frac{J_{48}^9}{J_{4,48}^2J_{8,48}^2J_{12,48}^2J_{16,48}J_{20,48}^2}.
\end{align}

Similarly, using \eqref{Euler} and \eqref{Heine}, we have
\begin{align}
    Z_1(q)&=\sum_{m\geq 0}\frac{q^{4m^2+4m+1}(-q^{4-8m};q^8)_{\infty}}{(q^4;q^4)_{2m+1}}=(-q^4;q^8)_{\infty}\sum_{m\geq 0}\frac{q^{4m+1}(-q^4;q^8)_m}{(q^4;q^4)_{2m+1}} \nonumber \\
    &=\frac{q(-q^4;q^4)_{\infty}}{(q^4;q^8)_{\infty}^2}\sum_{n\geq 0}\frac{(q^4;q^8)_n}{(q^{16};q^{16})_n}(-1)^nq^{4n^2+8n}.   \label{dual-11-z1}  
\end{align}

Substituting \eqref{Rama4211} with $q$ replaced by $-q^4$ into \eqref{dual-11-z1}, we deduce that
\begin{align}\label{dual-11-z1-1}
Z_1(q)=\frac{qJ_{48}^9}{J_{4,48}^3J_{12,48}^2J_{16,48}J_{20,48}^3}.    
\end{align}

Substituting \eqref{dual-11-z0-1} and \eqref{dual-11-z1-1} into \eqref{dual-11-ct}, we have
\begin{align}
F_2(q)=Z_0(q)\sum_{s\in\mathbb{Z}}q^{4s^2}+Z_1(q)\sum_{s\in\mathbb{Z}}q^{(2s+1)^2}=\frac{J_{24}^6}{J_{2,24}^2J_{4,24}J_{10,24}^2}.
\end{align}

Finally, we write the Nahm sum in \eqref{dual-2} as
\begin{align*}
  S(q):=\sum_{m,n,r\geq 0}\frac{q^{4m^2+3n^2+r^2+4mn-2nr+4m+2n}}{(q^4;q^4)_m(q^4;q^4)_n(q^4;q^4)_r}=\sum_{n\geq 0}c_{n}q^n
\end{align*}
and define
\begin{align*}
    &S_0(q):=\sum_{n\geq 0}c_{2n}q^{2n}, \quad     S_1(q):=\sum_{n\geq 0}c_{2n+1}q^{2n+1}.
\end{align*}
By \eqref{dual-ex11-eq1} and \eqref{dual-ex11-eq2}, we have
\begin{align}
    aS_0(q)+bS_1(q)&=\frac{J_{24}^6}{J_{2,24}^2J_{4,24}J_{10,24}^2}, \\
    bS_0(q)+aS_1(q)&=3q\frac{J_8J_{12}^3}{J_4^3}
\end{align}
where the functions $a,b$ are defined in \eqref{ab-defn}. 
It follows that
\begin{align*}
    &S(q)=S_0(q)+S_1(q)=\frac{1}{a+b}\Big(\frac{J_{24}^6}{J_{2,24}^2J_{4,24}J_{10,24}^2}+3q\frac{J_8J_{12}^3}{J_4^3}\Big) \nonumber\\
    &=\frac{J_{2,24}J_{24}^{9/2}J_{6,24}J_{10,24}}{J_{1,24}J_{4,24}^2J_{5,24}J_{7,24}J_{8,24}J_{11,24}J_{12,24}^{1/2}}. 
\end{align*}
Here the last equality can be proved easily using the Maple approach in \cite{Frye-Garvan}.
\end{proof}

\subsection{Dual of Example 12}\label{sec-exam12}
We list Zagier's Example 12 and its dual in Table \ref{tab-exam12}. 

{\small
\begin{table}[htbp]
\centering
\caption{Modular triples Example 12 and their duals}
    \label{tab-exam12}
\begin{tabular}{|c|cc|c|cc|}
\hline
\padedvphantom{I}{3.5ex}{3.5ex}
    $A$   & \multicolumn{2}{c|}{$\begin{pmatrix}  8 & 4 & 1 \\ 4 & 3 & 0 \\ 1 & 0 &1 \end{pmatrix}$}  & $\mathcal{D}(A)$
    & \multicolumn{2}{c|}{$\begin{pmatrix}
3/5 & -4/5 & -3/5 \\ -4/5 & 7/5 & 4/5 \\ -3/5  & 4/5 & 8/5
\end{pmatrix}$} \\
\hline 
\padedvphantom{I}{3.5ex}{3.5ex}
   $B$
     & $\begin{pmatrix} 0 \\ -1/2 \\ 1/2 \end{pmatrix}$ & $\begin{pmatrix} 2 \\ 1/2 \\ 1/2 \end{pmatrix}$
     &$\mathcal{D}(B)$
    & $\begin{pmatrix} 1/10 \\ -3/10 \\ 2/5 \end{pmatrix}$ & $\begin{pmatrix}
        1/2 \\ -1/2 \\ 0
    \end{pmatrix}$ \\
     \padedvphantom{I}{1ex}{1ex}
   $C$
    & $1/40$ & $9/40$ & $\mathcal{D}(C)$
    & $1/40$  & $1/40$ \\
    \hline
\end{tabular}
\end{table}
}

The modularity of Zagier's Example 12 has been confirmed by Wang \cite[Theorem 4.15]{Wang-rank3}. 
As for the dual, we are going to establish Theorem \ref{thm-dual-exam12} to justify its modularity. We first make some preparations.
\begin{lemma}\label{lem-SN}
For any $n\in\mathbb Z$ we have
\begin{align}\label{SN-evaluation}
 K_n(q):=\sum_{\substack{i,k\geq0\\i-k=n}}
 \frac{q^{k(k+1)/2}}{(q;q)_i(q;q)_k}=\frac{1}{(q;q)_\infty}
 \sum_{r\geq0}\frac{q^{\binom{n-2r}{2}}}{(q^2;q^2)_r}.
\end{align}
\end{lemma}
\begin{proof}
Using \eqref{Euler} and \eqref{Jacobi}, we have
\begin{align*}
&\sum_{n\in\mathbb Z}K_n(q)z^n=\sum_{i,k\geq0}
 \frac{q^{k(k+1)/2}z^{i-k}}{(q;q)_i(q;q)_k}=\frac{(-q/z;q)_\infty}{(z;q)_\infty} \nonumber \\
&=\frac{(-z;q)_\infty(-q/z;q)_\infty}{(z^2;q^2)_\infty} =\frac{1}{(q;q)_\infty}
 \Big(\sum_{r\geq 0}\frac{z^{2r}}{(q^2;q^2)_r}\Big)
 \Big(\sum_{m\in\mathbb Z}z^mq^{\binom m2}\Big)
 \nonumber \\
 &=\frac{1}{(q;q)_\infty}\sum_{n\in\mathbb Z}z^n
 \sum_{r\geq0}\frac{q^{\binom{n-2r}{2}}}{(q^2;q^2)_r}.
\end{align*}
Comparing the coefficients of $z^n$ on both sides, we obtain \eqref{SN-evaluation}.
\end{proof}

For $0\leq s\leq 3$ we define
\begin{align}\label{Theta-defn}
\Theta_s=\Theta_s(q):=\sum_{n\in\mathbb{Z}}q^{4n^2-2sn}=(-q^{4-2s},-q^{4+2s},q^8;q^8)_\infty.
\end{align}
For convenience, we denote the Nahm sums in the left side of \eqref{dual-12-1} and \eqref{dual-12-2} as $\mathcal{T}_1(q)$ and $\mathcal{T}_2(q)$, respectively. 

We now apply Lemma \ref{lem-SN} to give reductions for the Nahm sums $\mathcal{T}_1$ and $\mathcal{T}_2$.
\begin{prop}\label{prop:kernel}
For $s=0,1,2,3$ we define
\begin{align*}
\alpha_s(q)&:=
\sum_{\substack{r,j\ge0\\5r+2j+1\equiv s \!\!\!\pmod4}}
\frac{q^{(15r^2-20rj+10j^2+10r-10j+s^2-1)/4}}
{(q^{10};q^{10})_r(q^5;q^5)_j},\\
\beta_s(q)&:=
\sum_{\substack{r,j\ge0\\5r+2j\equiv s \!\!\!\pmod4}}
\frac{q^{(15r^2-20rj+10j^2+20r-10j+s^2)/4}}
{(q^{10};q^{10})_r(q^5;q^5)_j}.
\end{align*}
Then
\begin{equation}
\mathcal{T}_1(q)=\frac{1}{(q^5;q^5)_\infty}\sum_{s=0}^3\Theta_s(q)\alpha_s(q),
 \quad \mathcal{T}_2(q)=\frac{1}{(q^5;q^5)_\infty}\sum_{s=0}^3\Theta_s(q)\beta_s(q).
\label{T1T2-reduction}
\end{equation}
\end{prop}
\begin{proof}
We first discuss $\mathcal{T}_1$. We denote
\begin{align*}
    Q(i,j,k):=\tfrac{3}{2}i^2+\tfrac{7}{2}j^2+4k^2-4ij-3ik+4jk.
\end{align*}
Note that
\begin{align*}
 Q(i,j,k)+\tfrac{1}{2}i-\tfrac{3}{2}j+2k=\tfrac{3}{2}(i-k)^2+\big(\tfrac{1}{2}-4j\big)(i-k)+\tfrac{7}{2}j^2-\tfrac{3}{2}j+\tfrac{5}{2}k(k+1).
\end{align*}
Let $n=i-k$.  We have
\begin{align}
\mathcal{T}_1(q)&=\sum_{j\geq0}\frac{q^{(7j^2-3j)/2}}{(q^5;q^5)_j}
 \sum_{n\in\mathbb Z}q^{\frac{3}{2}n^2+(\frac{1}{2}-4j)n}K_n(q^5) \nonumber \\
 &=\frac{1}{(q^5;q^5)_\infty}
\sum_{r,j\geq0}\frac{q^{(7j^2-3j)/2+10r^2+5r}}
{(q^{10};q^{10})_r(q^5;q^5)_j}\sum_{n\in\mathbb Z}
q^{4n^2-2(5r+2j+1)n}. \quad \text{(by \eqref{SN-evaluation})}
\label{T1-reduction-start}
\end{align}
Let $c=5r+2j+1$, and write  $c=4h+s$ with $h\in \mathbb{Z}$ and 
$s\in\{0,1,2,3\}$.  Replacing $n$ by $n+h$, we obtain
\begin{equation}
\sum_{n\in\mathbb Z}q^{4n^2-2cn}
=q^{-(c^2-s^2)/4}\Theta_s(q).
\label{theta-translate}
\end{equation}
Substituting \eqref{theta-translate} into \eqref{T1-reduction-start}, we obtain the first identity in \eqref{T1T2-reduction}.

Next, we discuss $\mathcal{T}_2$. Note that
\begin{align*}
Q(i,j,k)+\tfrac{5}{2}i-\tfrac{5}{2}j=\tfrac{3}{2}(i-k)^2+\big(\tfrac{5}{2}-4j\big)(i-k)+\tfrac{7}{2}j^2-\tfrac{5}{2}j+\tfrac{5}{2}k(k+1).
 \end{align*}
Again let $n=i-k$. By \eqref{SN-evaluation} we have
\begin{align}
\mathcal{T}_2(q)&=\frac{1}{(q^5;q^5)_\infty}
\sum_{r,j\geq0}\frac{q^{(7j^2-5j)/2+10r^2+5r}}
{(q^{10};q^{10})_r(q^5;q^5)_j}\sum_{n\in\mathbb Z}
q^{4n^2-2(5r+2j)n}.
\end{align}
Writing $5r+2j=4h+s$ with $h\in \mathbb{Z}$ and 
$s\in\{0,1,2,3\}$, and then substituting \eqref{theta-translate} into it, we obtain the second identity in \eqref{T1T2-reduction}.
\end{proof}
We now have all ingredients to prove Theorem \ref{thm-dual-exam12}.
\begin{proof}[Proof of Theorem \ref{thm-dual-exam12}]
In $\alpha_s$, the congruence condition is equivalent to $r+2j+1\equiv s\pmod4$.  It has solutions only when
$r\equiv s-1\pmod2$, and this condition then determines uniquely a number $a_{r,s}\in \{0,1\}$ such that $j\equiv a_{r,s}$ (mod 2). We have
\begin{align}
\alpha_s(q)&=
\sum_{\substack{r \geq 0 \\ r\equiv s-1 \!\!\! \pmod{2}}} \frac{q^{\frac{1}{4}(5r^2+s^2-1)}}{(q^{10};q^{10})_r}\sum_{\substack{j\geq 0 \\ j\equiv a_{r,s}\!\!\!\pmod{2}}}
\frac{q^{5\binom{j-r}{2}}}
{(q^5;q^5)_j}  \nonumber \\
&=q^{(s^2-1)/4}(-q^5;q^5)_\infty
 \sum_{\substack{r\ge0\\r\equiv s-1 \!\!\!\pmod2}}
 \frac{q^{5r^2/4}}{(q^5;q^5)_r}.
\label{alpha-reduce}
\end{align}
Here for the second equality we used the following fact: for $a\in \{0,1\}$,
\begin{align}
&\sum_{\substack{j\ge0\\j\equiv a \!\!\!\pmod2}}
 \frac{q^{\binom{j-r}{2}}}{(q;q)_j}
=\frac{1}{2} q^{r(r+1)/2}
 \left((-q^{-r};q)_\infty+(-1)^a(q^{-r};q)_\infty\right)\nonumber \\
 &=(-q;q)_\infty(-q;q)_r.
\label{eq-partial}
\end{align}
Hence we have
\begin{equation}\label{alpha-product}
\begin{split}
    &\alpha_0(q)=q(-q^5;q^5)_\infty \mathcal{A}_1(q^5), \quad \alpha_1(q)=(-q^5;q^5)_\infty \mathcal{A}_0(q^5), \\
    &\alpha_2(q)=q^2(-q^5;q^5)_\infty \mathcal{A}_1(q^5), \quad \alpha_3(q)= q^2(-q^5;q^5)_\infty \mathcal{A}_0(q^5).
\end{split}
\end{equation}
Here the functions $\mathcal{A}_i(q)$ ($i=0,1$) are given in \eqref{S79}--\eqref{S94}.

Similarly, in $\beta_s$, the congruence condition $5r+2j\equiv s$ (mod 4) becomes
$r+2j\equiv s\pmod4$. Hence $r\equiv s\pmod2$ and this determines a unique $b_{r,s}\in \{0,1\}$ such that $j\equiv b_{r,s}$ (mod 2). Using \eqref{eq-partial} we have
\begin{align}
\beta_s(q)&=
\sum_{\substack{r \geq 0\\ r\equiv s \!\!\! \pmod{2}}} \frac{q^{(5r^2+10r+s^2)/4}}{(q^{10};q^{10})_r} \sum_{\substack{j\geq 0 \\ j\equiv b_{r,s}\!\!\! \pmod{2}}}
\frac{q^{5\binom{j-r}{2}}}
{(q^5;q^5)_j} \nonumber \\
&=q^{s^2/4}(-q^5;q^5)_\infty
 \sum_{\substack{r\ge0\\r\equiv s \!\!\! \pmod2}}
 \frac{q^{5r(r+2)/4}}{(q^5;q^5)_r}.
\label{beta-reduce}
\end{align}
Hence we have
\begin{equation}\label{beta-product}
\begin{split}
\beta_0(q)&=(-q^5;q^5)_\infty \mathcal{B}_1(q^5), \quad \beta_1(q)=q^{-1}(-q^5;q^5)_\infty \mathcal{B}_0(q^5), \\
\beta_2(q)&=q(-q^5;q^5)_\infty \mathcal{B}_1(q^5), \quad \beta_3(q)=q(-q^5;q^5)_\infty \mathcal{B}_0(q^5).
\end{split}
\end{equation}
Here the functions $\mathcal{B}_i(q)$ ($i=0,1$) are given in \eqref{S96}--\eqref{S99}.

Substituting \eqref{alpha-product} and \eqref{beta-product}) into \eqref{T1T2-reduction}, and using the fact $\Theta_3=q^{-2}\Theta_1$ which is easily deduced from \eqref{Theta-defn}, we obtain 
\begin{align}
\mathcal{T}_1(q)&=\frac{(-q^5;q^5)_\infty}{(q^5;q^5)_\infty}
 \Big(2\Theta_1 \mathcal{A}_0(q^5)+(q\Theta_0+q^2\Theta_2)\mathcal{A}_1(q^5)\Big),
\label{S1-reduction}\\
\mathcal{T}_2(q)&=\frac{(-q^5;q^5)_\infty}{(q^5;q^5)_\infty}
 \Big((\Theta_0+q\Theta_2)\mathcal{B}_1(q^5)+2q^{-1}\Theta_1 \mathcal{B}_0(q^5)\Big).
\label{S2-reduction}
\end{align}
It is now straightforward to prove \eqref{dual-12-1} and \eqref{dual-12-2} from the above formulas using the Maple approach in \cite{Frye-Garvan}.
\end{proof}

\begin{rem}
Theorem \ref{thm-dual-exam12} was proposed as a conjecture in the first version of this paper (arXiv:2607.23257v1). After finishing the current version, we noted that Goswami \cite{Goswami} has independently provided a different proof for it.
\end{rem}

\section{Some new modular rank three Nahm sums}\label{sec-rank3}
This section is devoted to presenting proofs for Theorems \ref{thm-ex9-new} and \ref{thm-dual-rank3} which provide some new modular Nahm sums. We record the modular triples in these theorems in Table \ref{tab-newexample}. 

{\small 
\begin{table}[htbp]
\centering
\caption{Modular triples $(A,B,C)$ in Theorem \ref{thm-ex9-new} and their duals}\label{tab-newexample}
\begin{tabular}{c|cccc}
  \hline
    \padedvphantom{I}{4ex}{4ex}
  $A$ &  \multicolumn{4}{c}{$\begin{pmatrix} 2 &1 & 0 \\ 1 & 3/2 & -1/2 \\ 0 &-1/2 & 1 \end{pmatrix}$}  \\
  \hline
\padedvphantom{I}{4ex}{4ex}
$B$ & $\begin{pmatrix} 1 \\ 1/2 \\ 0 \end{pmatrix}$ & $\begin{pmatrix} 0  \\ -1/2  \\ 0 \end{pmatrix}$ & $\begin{pmatrix} 0 \\ 0 \\ -1/2\end{pmatrix}$  & $\begin{pmatrix} 0 \\ 0 \\ 0 \end{pmatrix}$ \\
  ~~ $C$ & $ 1/9$ & $0$ & $1/36$ & $-1/18$ \\
  \hline
    \padedvphantom{I}{4ex}{4ex}
  $\mathcal{D}(A)$ &  \multicolumn{4}{c}{$\begin{pmatrix} 5/6 & -2/3 & -1/3 \\ -2/3 & 4/3 & 2/3 \\ -1/3 & 2/3 & 4/3 \end{pmatrix}$}  \\
  \hline
\padedvphantom{I}{4ex}{4ex}
$\mathcal{D}(B)$ & $\begin{pmatrix} 1/2 \\ 0 \\ 0 \end{pmatrix}$ & $\begin{pmatrix} 1/3  \\ -2/3  \\ -1/3 \end{pmatrix}$ & $\begin{pmatrix} 1/6 \\ -1/3 \\ -2/3 \end{pmatrix}$  & $\begin{pmatrix} 0 \\ 0 \\ 0 \end{pmatrix}$ \\
  ~~ $\mathcal{D}(C)$ & $1/72$ & $1/24$ & $1/72$ & $-5/72$ \\
  \hline
\end{tabular}
\end{table}
}

The proof of Theorem \ref{thm-ex9-new} is relatively straightforward based on the Bailey pair method. However, the proof of Theorem \ref{thm-dual-rank3} is quite involved and requires more ingredients. We first present several preliminary results that will streamline the final proof of Theorem \ref{thm-dual-rank3}.
   
\subsection{Auxiliary results}\label{sec-aux}
For $n\in\mathbb{Z}$, let $H_n:=\widetilde{H}_n:=0$ for $n<0$, and for $n\geq 0$ we define
\begin{equation}
 H_n:=\sum_{\substack{j,k\geq 0\\2j+k=n}}
 \frac{q^{\frac{1}{2}k^2}}{(q;q)_j(q;q)_k}, \quad \widetilde{H}_n:=\sum_{\substack{j,k\geq 0\\2j+k=n}}
 \frac{q^{\frac{1}{2}(k^2-k)}}{(q;q)_j(q;q)_k},        \label{H-def}
\end{equation}
\begin{lemma}\label{lem-finite}
For $n\geq 0$ we have
\begin{equation}
 H_n=q^{\frac{1}{2}n}\sum_{r\geq 0}
 \frac{q^{-r}}{(q^2;q^2)_r(q;q)_{n-2r}}.   \label{H-finite}
\end{equation}
For any $n\in \mathbb{Z}$ we have the recurrence relation:
\begin{equation}
 q^nH_n=H_n-q^{\frac{1}{2}}H_{n-1}-H_{n-2}+q^{\frac{1}{2}}H_{n-3}. \label{H-rec}
\end{equation}
Furthermore, we have
\begin{align}\label{wH-H-relation}
    \widetilde{H}_{n}=q^{-\frac{1}{2}n}(H_n-H_{n-2}).
\end{align}
\end{lemma}
\begin{proof}
We have 
\begin{align}
&H(z):=\sum_{n\geq 0} H_nz^n=\sum_{j,k\geq 0} \frac{q^{\frac{1}{2}k^2}z^{2j+k}}{(q;q)_j(q;q)_k}=\frac{(-q^{\frac{1}{2}}z;q)_\infty}{(z^2;q)_\infty} \nonumber \\
&=\frac{1}{(z^2;q^2)_\infty(q^{\frac{1}{2}}z;q)_\infty}=\sum_{r,k\geq 0} \frac{q^{\frac{1}{2}k}z^{2r+k}}{(q^2;q^2)_r(q;q)_k}. \label{H-gen}
\end{align}
Comparing the coefficients of $z^n$ on both sides, we obtain \eqref{H-finite}.

From \eqref{H-gen} we have
\begin{align*}
H(qz)=(1-q^{\frac{1}{2}}z)(1-z^2)H(z).
\end{align*}
Comparing the coefficients on both sides, we obtain \eqref{H-rec}.

Similarly, we have
\begin{align}
\widetilde{H}(z):=\sum_{n\geq 0}\widetilde{H}_nz^n=\frac{(-z;q)_\infty}{(z^2;q)_\infty}=(1-q^{-1}z^2)H(q^{-\frac{1}{2}}z).
\end{align}
Comparing the coefficients of $z^n$ on both sides, we obtain \eqref{wH-H-relation}.
\end{proof}

Throughout this section we denote 
\begin{equation}
 \Omega_a=\frac{J_{a,15}}{J_1},\quad \Lambda_0=\frac{1}{J_1}(J_{7,15}-qJ_{2,15}), \quad 
 \Lambda_1=\frac{1}{J_1}(J_{4,15}+qJ_{1,15}).  \label{Omega-Lambda}
\end{equation}

\begin{lemma}\label{lem-Andrews}
(Cf.\ \cite[Theorems 1--6]{Andrews1981}.) We have
\begin{align}
 \sum_{n,r\geq 0}
 \frac{q^{3n(n+1)/2-r}}
 {(q^2;q^2)_r(q;q)_{3n+1-2r}}&=\Lambda_1,   \label{And-1}\\
 \sum_{n,r\geq 0}
 \frac{q^{3n(n+1)/2-r}}
 {(q^2;q^2)_r(q;q)_{3n-2r}}&=\Lambda_0,        \label{And-2}\\
 \sum_{n\geq 1,r\geq 0}
 \frac{q^{n(3n-1)/2-r}}
 {(q^2;q^2)_r(q;q)_{3n-2r-1}}&=\Omega_6,                  \label{And-3}\\
 \sum_{n,r\geq 0}
 \frac{q^{n(3n+5)/2+1-r}}
 {(q^2;q^2)_r(q;q)_{3n-2r+1}}&=q\Omega_3,    \label{And-4}\\
 \sum_{n,r\geq 0}
 \frac{q^{n(3n+1)/2-r}}
 {(q^2;q^2)_r(q;q)_{3n-2r}}&=\Omega_6,     \label{And-5}\\
 \sum_{n\geq 1,r\geq 0}
 \frac{q^{n(3n+1)/2-r}}
 {(q^2;q^2)_r(q;q)_{3n-2r-1}}&=q\Omega_3.      \label{And-6}
\end{align}
\end{lemma}

For $n\in\mathbb{Z}$, we define three sequences that are fundamental to our proof:
\begin{align}\label{abc-def}
 a_n:=\sum_{m\in\mathbb{Z}}q^{\frac{3}{2}m^2}H_{n+3m},  ~~        
 b_n:=\sum_{m\in\mathbb{Z}}q^{\frac{3}{2}m^2+r}H_{n+3m+1}, ~~         
 c_n:=\sum_{m\in\mathbb{Z}}q^{\frac{3}{2}m^2-m}H_{n+3m-1}.    
\end{align}

\begin{lemma}\label{lem-rec}
Each of $a_n,b_n,c_n$ satisfies
\begin{equation}
 y_{n+3}=q^{\frac{1}{2}}y_{n+2}+y_{n+1}+(q^{n+\frac{3}{2}}-q^{\frac{1}{2}})y_n.     \label{eq-common-rec}
\end{equation}
Moreover,
\begin{align}
& a_0=\Lambda_0, ~~a_1=q^{\frac{1}{2}}\Lambda_1,~~a_2=\Lambda_0+q\Lambda_1, \label{a-initial}\\
&b_0=c_0=q^{\frac{1}{2}}\Omega_3, ~~b_1=c_1=\Omega_6,~~
 b_2=c_2=q^{\frac{1}{2}}(\Omega_3+\Omega_6).        \label{bc-initial}
\end{align}
In particular $b_n=c_n$ for every $n\geq 0$.
\end{lemma}

\begin{proof}
We consider the sequence
\begin{align*}
 y_n(s,t):=\sum_{m\in\mathbb{Z}}q^{\frac{3}{2}m^2+sm}H_{n+3m+t}.
\end{align*}
Applying \eqref{H-rec} with $n$ replaced by $n+3m+t+3$, multiplying both sides by $q^{\frac{3}{2}m^2+sm}$ and then summing over $m\in \mathbb{Z}$, and then replacing $m$ by $m+1$ in the term containing $q^n$, we deduce that
\begin{align}
    y_{n+3}(s,t)=q^{\frac{1}{2}}y_{n+2}(s,t)+y_{n+1}(s,t)+(q^{n+t+\frac{3}{2}-s}-q^{\frac{1}{2}})y_n(s,t).
\end{align}
Setting $(t,s)=(0,0),(1,1),(-1,-1)$, we obtain \eqref{eq-common-rec}.

Substituting \eqref{H-finite} into \eqref{abc-def}, from Lemma \ref{lem-Andrews} we obtain 
$$a_0=\Lambda_0,~~a_1=q^{\frac{1}{2}}\Lambda_1, ~~ b_0=c_0=q^{\frac{1}{2}}\Omega_3, ~~b_1=c_1=\Omega_6.$$
Applying \eqref{eq-common-rec} with $n=-1$ we obtain the values for $a_2$, $b_2$ and $c_2$. Since $b_n$ and $c_n$ share the same initial values and the same recurrence, we conclude that $b_n=c_n$.
\end{proof}

We are now going to solve the recurrence relation \eqref{eq-common-rec} so that we can find useful expressions for $a_n$, $b_n$ and $c_n$. Surprisingly, its solution is implicitly connected to the configuration sum studied by Andrews--Baxter--Forrester \cite{ABF}. For $1\le a,b,c\le4$ with $|b-c|=1$, we define the 4-state one-dimensional configuration sum \cite[(2.7)]{Warnaar1996}
\begin{equation}\label{ABF-defn}
X_n(a,b,c)=X_n(a,b,c;q):=
 \sum_{\substack{\sigma_0=a,\ \sigma_n=b,\ \sigma_{n+1}=c\\
 1\le\sigma_i\le4,\ |\sigma_{i+1}-\sigma_i|=1}}
 q^{\sum_{k=1}^n k|\sigma_{k+1}-\sigma_{k-1}|/4}.                  
\end{equation}
Note that the last term in the exponent of $q$, i.e., $n|\sigma_{n+1}-\sigma_{n-1}|/4$ is zero when $\sigma_{n-1}=\sigma_{n+1}$ and is $n/2$ when $|\sigma_{n+1}-\sigma_{n-1}|=2$. It follows that (see e.g,~\cite[(2.11)]{Warnaar1996})
\begin{align}\label{ABF-Xn-rec}
    X_n(a,b,b\pm 1)=X_{n-1}(a,b\pm 1,b)+q^{n/2}X_{n-1}(a,b\mp 1,b), \quad 3\leq 2b\pm 1\leq 2r-3.
\end{align}
We denote
\begin{align*}
 A_n(a)&=X_n(a,1,2),& B_n(a)&=X_n(a,2,1),& C_n(a)&=X_n(a,2,3),\\
 D_n(a)&=X_n(a,3,2),& E_n(a)&=X_n(a,3,4),& F_n(a)&=X_n(a,4,3).
\end{align*}
When no confusion can arise, we suppress the argument $a$. 
It follows from \eqref{ABF-Xn-rec} that
\begin{align}
 A_n&=B_{n-1}, \quad F_n=E_{n-1}, \label{ABCDEF-1}\\
 B_n&=A_{n-1}+q^{n/2}D_{n-1}, \quad E_n=F_{n-1}+q^{n/2}C_{n-1},     \label{ABCDEF-2}\\ 
 C_n&=D_{n-1}+q^{n/2}A_{n-1}, \quad D_n=C_{n-1}+q^{n/2}F_{n-1}.        \label{ABCDEF-3}
\end{align}

Now we define
\[
 S_n(a)=A_n(a)+F_n(a),\qquad T_n(a)=B_n(a)+E_n(a),\qquad U_n(a)=C_n(a)+D_n(a).
\]
From \eqref{ABCDEF-1}--\eqref{ABCDEF-3} we deduce that
\begin{align}\label{STU}
 S_n=T_{n-1},\qquad
 T_n=S_{n-1}+q^{n/2}U_{n-1},\qquad
 U_n=U_{n-1}+q^{n/2}S_{n-1}.
\end{align}
Eliminating $T_n,U_n$ yields
\begin{equation}\label{ABF-rec}
 S_n=q^{\frac{1}{2}}S_{n-1}+S_{n-2}
 +(q^{n-\frac{3}{2}}-q^{\frac{1}{2}})S_{n-3}.               
\end{equation}
Replacing $n$ by $n+3$, equation \eqref{ABF-rec} is exactly \eqref{eq-common-rec}.

From direct calculations using \eqref{ABF-defn}, we have
\begin{align*}
    S_0(1)=1, \quad S_1(1)=0, \quad S_2(1)=1; \quad S_0(2)=0, \quad S_1(2)=1, \quad S_2(2)=q^{1/2}.
\end{align*}
Since $a_n,b_n,c_n$ and $S_n$ satisfy the same recurrence relation \eqref{eq-common-rec}, comparing their initial values, we deduce that
\begin{align}\label{abc-S}
    a_n=\Lambda_0S_n(1)+q^{\frac{1}{2}}\Lambda_1S_n(2), \quad b_n=c_n=q^{\frac{1}{2}}\Omega_3S_n(1)+\Omega_6S_n(2).
\end{align}
From \eqref{ABF-defn} it is easy to see that $X_n(a,b,c)=0$ whenever $n\equiv c-a$ (mod 2). Hence
\begin{align}
A_{2m+1}(1)=0, \quad F_{2m}(1)=0, \quad A_{2m}(2)=0, \quad F_{2m+1}(2)=0.
\end{align}
It follows that
\begin{equation}\label{ab-X-1st}
\begin{split}
    a_{2m}&=\Lambda_0A_{2m}(1)+q^{\frac{1}{2}}\Lambda_1 F_{2m}(2)=\Lambda_0X_{2m}(1,1,2)+q^{\frac{1}{2}}\Lambda_1 X_{2m}(2,4,3),  \\
    a_{2m+1}&=\Lambda_0F_{2m+1}(1)+q^{\frac{1}{2}}\Lambda_1A_{2m+1}(2)=\Lambda_0X_{2m+1}(1,4,3)+q^{\frac{1}{2}}\Lambda_1 X_{2m+1}(2,1,2), \\
    b_{2m}&=q^{\frac{1}{2}}\Omega_3A_{2m}(1)+\Omega_6F_{2m}(2)=q^{\frac{1}{2}} \Omega_3X_{2m}(1,1,2)+\Omega_6X_{2m}(2,4,3), \\
    b_{2m+1}&=q^{\frac{1}{2}}\Omega_3F_{2m+1}(1)+\Omega_6A_{2m+1}(2)=q^{\frac{1}{2}}\Omega_3X_{2m+1}(1,4,3)+\Omega_6 X_{2m+1}(2,1,2).
\end{split}
\end{equation}

Setting $r=5$ in \cite[Theorem 2.3.1]{ABF} or \cite[(2.14)]{Warnaar1996}, we have
\begin{align}\label{X-binomial-exp}
X_n(a,b,c;q)
&=q^{(a-b)(a-c)/4}\sum_{j\in\mathbb{Z}}\Bigg\{
q^{20j^2+(5(b+c-1)/2-4a)j}
\qbinom{n}{(n+a-b)/2-5j}\nonumber \\
&\qquad -q^{20j^2+(5(b+c-1)/2+4a)j+a(b+c-1)/2}
\qbinom{n}{(n-a-b)/2-5j}\Bigg\}. 
\end{align}

It follows that
\begin{align}
 X_{2m}(1,1,2)&=\sum_{j\in\mathbb{Z}}\left\{q^{20j^2-j}\qbinom{2m}{m-5j}
 -q^{20j^2+9j+1}\qbinom{2m}{m-5j-1}\right\},    \label{X-even-112}\\
q^{\frac{1}{2}}X_{2m}(2,4,3)&=\sum_{j\in\mathbb{Z}}\left\{
 q^{20j^2+7j+1}\qbinom{2m}{m-5j-1}
 -q^{20j^2+17j+4}\qbinom{2m}{m-5j-2}\right\}, \label{X-even-243}\\
  q^{-\frac{1}{2}}X_{2m+1}(1,4,3)&=\sum_{j\in\mathbb{Z}}\left\{
 q^{20j^2+11j+1}\qbinom{2m+1}{m-5j-1}
 -q^{20j^2+21j+5}\qbinom{2m+1}{m-5j-2}\right\}, \label{X-odd-143}\\
X_{2m+1}(2,1,2)&=\sum_{j\in\mathbb{Z}}\left\{
 q^{20j^2+3j}\qbinom{2m+1}{m-5j}
 -q^{20j^2+13j+2}\qbinom{2m+1}{m-5j-1}\right\}. \label{X-odd-212}
\end{align}
\begin{lemma}\label{lem-X-BP-result}
We have
\begin{align}
    \sum_{m\geq 0} \frac{q^{m^2}}{(q;q)_{2m}}X_{2m}(1,1,2)&=\frac{\overline{J}_{44,90}-q^2\overline{J}_{26,90}}{J_1}, \label{eq-lem-X1-id-1}\\
    q^{\frac{1}{2}} \sum_{m\geq 0} \frac{q^{m^2}}{(q;q)_{2m}}X_{2m}(2,4,3)&=\frac{q^2\overline{J}_{28,90}-q^8\overline{J}_{8,90}}{J_1},  \label{eq-lem-X1-id-2}\\
    q^{-\frac{1}{2}}\sum_{m\geq 0} \frac{q^{m^2+m}}{(q;q)_{2m+1}}X_{2m+1}(1,4,3)&=\frac{q^3\overline{J}_{19,90}-q^{10}\overline{J}_{1,90}}{J_1},  \label{eq-lem-X1-id-3}\\
     \sum_{m\geq 0} \frac{q^{m^2+m}}{(q;q)_{2m+1}}X_{2m+1}(2,1,2)&=\frac{\overline{J}_{37,90}-q^4\overline{J}_{17,90}}{J_1}. \label{eq-lem-X1-id-4}
\end{align}
\end{lemma}
\begin{proof}
(1) From \eqref{X-even-112} and \eqref{X-even-243} we have
\begin{align}
&\frac{X_{2m}(1,1,2)}{(q;q)_{2m}}=\frac{1}{(q;q)_m^2}+\sum_{j>0} \frac{q^{20j^2-j}+q^{20j^2+j}}{(q;q)_{m-5j}(q;q)_{m+5j}}-\sum_{j\geq 0} \frac{q^{20j^2+9j+1}}{(q;q)_{m-5j-1}(q;q)_{m+5j+1}} \nonumber \\
&\quad -\sum_{j\geq 0} \frac{q^{20j^2+31j+12}}{(q;q)_{m-5j-4}(q;q)_{m+5j+4}}, \\
&q^{\frac{1}{2}}\frac{X_{2m}(2,4,3)}{(q;q)_{2m}}=\sum_{j\geq 0} \frac{q^{20j^2+7j+1}}{(q;q)_{m-5j-1}(q;q)_{m+5j+1}}+\sum_{j\geq 0} \frac{q^{20j^2+33j+14}}{(q;q)_{m-5j-4}(q;q)_{m+5j+4}} \nonumber \\
&-\sum_{j\geq 0} \frac{q^{20j^2+17j+4}}{(q;q)_{m-5j-2}(q;q)_{m+5j+2}} -\sum_{j\geq 0} \frac{q^{20j^2+23j+7}}{(q;q)_{m-5j-3}(q;q)_{m+5j+3}}.
\end{align}
Here we replaced $j$ by $-j-1$ when $j<0$. 
Therefore, both $\big(\alpha_m^{(1)},X_{2m}(1,1,2)/(q;q)_{2m}\big)$ and $\big(\alpha_m^{(2)}, q^{\frac{1}{2}}X_{2m}(2,4,3)/(q;q)_{2m}\big)$ are  Bailey pairs relative to 1 with $\alpha_m^{(k)}$ ($k=1,2$) given in Table \ref{tab-AB12}. Substituting these Bailey pairs into \eqref{bailey-s1}, we obtain \eqref{eq-lem-X1-id-1} and \eqref{eq-lem-X1-id-2}.

Similarly, we have
\begin{align}
  & q^{-\frac{1}{2}} \frac{X_{2m+1}(1,4,3)}{(q;q)_{2m+1}} =\sum_{j\geq 0} \frac{q^{20j^2+11j+1}}{(q;q)_{m-5j-1}(q;q)_{m+5j+2}}+\sum_{j\geq 0} \frac{q^{20j^2+29j+10}}{(q;q)_{m-5j-3}    (q;q)_{m+5j+4}}  \nonumber \\
    &\quad -\sum_{j\geq 0} \frac{q^{20j^2+19j+4}}{(q;q)_{m-5j-2}(q;q)_{m+5j+3}}-\sum_{j\geq 0} \frac{q^{20j^2+21j+5}}{(q;q)_{m-5j-2}(q;q)_{m+5j+3}}, \\
&\frac{X_{2m+1}(2,1,2)}{(q;q)_{2m+1}}=\sum_{j\geq 0} \frac{q^{20j^2+3j}}{(q;q)_{m-5j}(q;q)_{m+5j+1}}+\sum_{j\geq 0} \frac{q^{20j^2+37j+17}}{(q;q)_{m-5j-4}(q;q)_{m+5j+5}} \nonumber \\
&-\sum_{j\geq 0} \frac{q^{20j^2+13j+2}}{(q;q)_{m-5j-1}(q;q)_{m+5j+2}}-\sum_{j\geq 0} \frac{q^{20j^2+27j+9}}{(q;q)_{m-5j-3}(q;q)_{m+5j+4}}.
\end{align}
Therefore, both
$$\Big(\frac{\gamma_m^{(1)}}{1-q},q^{-\frac{1}{2}}\frac{X_{2m+1}(1,4,3)}{(q;q)_{2m+1}}\Big) ~~\text{and} ~~ \Big(\frac{\gamma_m^{(2)}}{1-q}, \frac{X_{2m+1}(2,1,2)}{(q;q)_{2m+1}}\Big)$$
are  Bailey pairs relative to $q$ with $\gamma_m^{(k)}$ ($k=1,2$) given in Table \ref{tab-AB12}.

\begin{table}[htbp]
\centering
\begingroup
\footnotesize
\renewcommand{\arraystretch}{1.4}
\setlength{\tabcolsep}{4pt}

\begin{adjustbox}{max width=\linewidth}
\begin{tabular}{c|cccc}
\hline
$n$
& $\alpha_m^{(1)}$ & $\alpha_m^{(2)}$ & $\gamma_m^{(1)}$ & $\gamma_m^{(2)}$
\\
\hline
$0$ & $1$ & $0$ & $0$ & $1$
\\
$5j$ & $q^{20j^2-j}+q^{20j^2+j}$ & $0$ & $0$ & $q^{20j^2+3j}$ 
\\[3pt]
$5j+1$ & $-q^{20j^2+9j+1}$ & $q^{20j^2+7j+1}$ & $q^{20j^2+11j+1}$ & $-q^{20j^2+13j+2}$
\\
$5j+2$ & $0$ & $-q^{20j^2+17j+4}$ & $-q^{20j^2+19j+4}-q^{20j^2+21j+5}$ & $0$
\\
$5j+3$ & $0$ & $-q^{20j^2+23j+7}$ & $q^{20j^2+29j+10}$ & $-q^{20j^2+27j+9}$
\\
$5j+4$ & $-q^{20j^2+31j+12}$ & $q^{20j^2+33j+14}$ & $0$ & $q^{20j^2+37j+17}$
\\
\hline
\end{tabular}
\end{adjustbox}

\endgroup
\caption{Values of $\alpha_m^{(k)}$ and $\gamma_m^{(k)}$ for $k=1,2$.}
\label{tab-AB12}
\end{table}

Substituting these Bailey pairs into \eqref{bailey-s1}, we obtain \eqref{eq-lem-X1-id-3} and \eqref{eq-lem-X1-id-4}.
\end{proof}

\begin{lemma}\label{lem-X-BP-results-2}
We have
\begin{align}
      \sum_{m\geq 0} \frac{q^{m^2+m}}{(q;q)_{2m}}X_{2m}(1,1,2)&=\frac{\overline{J}_{41,90}-q\overline{J}_{31,90}}{J_1},   \label{X-BP-result-5}\\
    q^{\frac{1}{2}} \sum_{m\geq 0} \frac{q^{m^2+m}}{(q;q)_{2m}}X_{2m}(2,4,3)&=\frac{q^3\overline{J}_{23,90}-q^6\overline{J}_{13,90}}{J_1}, \label{X-BP-result-6} \\
    q^{-\frac{1}{2}}\sum_{m\geq 0} \frac{q^{m^2+2m}}{(q;q)_{2m+1}}X_{2m+1}(1,4,3)&=\frac{q^4\overline{J}_{14,90}-q^8\overline{J}_{4,90}}{J_1}, \label{X-BP-result-7} \\
     \sum_{m\geq 0} \frac{q^{m^2+2m}}{(q;q)_{2m+1}}X_{2m+1}(2,1,2)&=\frac{\overline{J}_{32,90}-q^2\overline{J}_{22,90}}{J_1}. \label{X-BP-result-8}
\end{align}
\end{lemma}
In order to apply the Bailey pair method, we see that we need to produce Bailey pairs with  $\beta$ parts the same as the proof of Lemma \ref{lem-X-BP-result} but the parameters change from $1,q$ to $q,q^2$, respectively. We achieve this by rewriting \eqref{X-even-112}--\eqref{X-odd-212} in different forms. 
We first state a useful identity. Let $n,k\in\mathbb{Z}$, and let $u$ be arbitrary.  Then
\begin{align}
(1-q^{n+1})
\left\{
q^u\qbinom nk-q^{u+n+1-2k}\qbinom{n}{k-1}
\right\}=\left(q^u-q^{u+n+1-2k}\right)\qbinom{n+1}k.   \label{raising}
\end{align}
In fact, the left side can be written as
\begin{align}
&q^u(1-q^{n+1-k})\qbinom{n+1}k-q^{u+n+1-2k}(1-q^{n+2-k})\qbinom{n+1}{k-1} \nonumber \\
&=q^u(1-q^{n+1-k})\qbinom{n+1}k-q^{u+n+1-2k}(1-q^k)\qbinom{n+1}k.     
\end{align}
This is exactly the right side upon simplifications. 
\begin{proof}[Proof of Lemma \ref{lem-X-BP-results-2}]
For $k=m-s$ and $n=2m$ or $2m+1$, \eqref{raising} can be written as
\begin{align}
(1-q^{2m+1})
\left\{
q^u\qbinom{2m}{m-s}
-q^{u+2s+1}\qbinom{2m}{m-s-1}
\right\}=\left(q^u-q^{u+2s+1}\right)
\qbinom{2m+1}{m-s},   \label{even-raise} \\
(1-q^{2m+2})
\left\{
q^u\qbinom{2m+1}{m-s}
-q^{u+2s+2}\qbinom{2m+1}{m-s-1}
\right\}=\left(q^u-q^{u+2s+2}\right)
\qbinom{2m+2}{m-s}.   \label{odd-raise}
\end{align}
Applying \eqref{even-raise} to \eqref{X-even-112} and \eqref{X-even-243}, we obtain
\begin{align}
&(1-q^{2m+1})X_{2m}(1,1,2)=\sum_{j\in\mathbb{Z}}
\left(q^{20j^2-j}-q^{20j^2+9j+1}\right)
\qbinom{2m+1}{m-5j},   \label{X-even-112-shift} \\
&q^{1/2}(1-q^{2m+1})X_{2m}(2,4,3)
=\sum_{j\in\mathbb{Z}}
\left(q^{20j^2+7j+1}-q^{20j^2+17j+4}\right)
\qbinom{2m+1}{m-5j-1}.  \label{X-even-243-shift}
\end{align}
Applying \eqref{odd-raise} to \eqref{X-odd-143} and \eqref{X-odd-212}, we obtain
\begin{align}
&q^{-1/2}(1-q^{2m+2})X_{2m+1}(1,4,3)
=\sum_{j\in\mathbb{Z}}
\left(q^{20j^2+11j+1}-q^{20j^2+21j+5}\right)
\qbinom{2m+2}{m-5j-1}, \label{X-odd-143-shift} \\
&(1-q^{2m+2})X_{2m+1}(2,1,2)
=\sum_{j\in\mathbb{Z}}
\left(q^{20j^2+3j}-q^{20j^2+13j+2}\right)
\qbinom{2m+2}{m-5j}.    \label{X-odd-212-shift}
\end{align}
From \eqref{X-even-112-shift}--\eqref{X-even-243-shift} we obtain two Bailey pairs relative to $q$:
$$\Big(\frac{\widetilde{\alpha}_m^{(1)}}{1-q},\frac{X_{2m}(1,1,2)}{(q;q)_{2m}}\Big), ~~ \Big(\frac{\widetilde{\alpha}_m^{(2)}}{1-q}, q^{\frac{1}{2}}\frac{X_{2m}(2,4,3)}{(q;q)_{2m}}\Big).$$
From \eqref{X-odd-143-shift} and \eqref{X-odd-212-shift} we obtain two Bailey pairs relative to $q^2$: 
$$\Big(\frac{\widetilde{\gamma}_m^{(1)}}{(1-q)(1-q^2)},q^{-\frac{1}{2}}\frac{X_{2m+1}(1,4,3)}{(q;q)_{2m+1}}\Big), ~~ \Big(\frac{\widetilde{\gamma}_m^{(2)}}{(1-q)(1-q^2)}, \frac{X_{2m+1}(2,1,2)}{(q;q)_{2m+1}}\Big).$$ 
Here 
\begin{align}
    &\widetilde{\alpha}_m^{(1)}=\begin{cases}
        q^{20j^2-j}-q^{20j^2+9j+1} & m=5j, \\
        -q^{20j^2+31j+12}+q^{20j^2+41j+21} & m=5m+4, \\
        0 & \text{otherwise};
    \end{cases} \\
    &\widetilde{\alpha}_m^{(2)}=\begin{cases}
        q^{20j^2+7j+1}-q^{20j^2+17j+4} & m=5j+1,\\
        -q^{20j^2+23j+7}+q^{20j^2+33j+14} & m=5j+3, \\
        0 & \text{otherwise};
    \end{cases} \\
    &\widetilde{\gamma}_m^{(1)}=\begin{cases}
        q^{20j^2+11j+1}-q^{20j^2+21j+5} & m=5j+1,\\
        -q^{20j^2+19j+4}+q^{20j^2+29j+10} & m=5j+2, \\
        0 & \text{otherwise}; 
    \end{cases}\\
    &\widetilde{\gamma}_m^{(2)}=\begin{cases}
        q^{20j^2+3j}-q^{20j^2+13j+2} &m=5j,\\
        -q^{20j^2+27j+9}+q^{20j^2+37j+17} &m=5j+3, \\
        0 & \text{otherwise}.
    \end{cases}
\end{align}
Substituting the above Bailey paris into \eqref{bailey-s1}, we obtain \eqref{X-BP-result-5}--\eqref{X-BP-result-8}.
\end{proof}

We may rewrite \eqref{ab-X-1st} as
\begin{equation}\label{ab-X-2nd}
\begin{split}
    a_{2m+2}&=\Lambda_0 B_{2m+1}(1)+q^{\frac{1}{2}}\Lambda_1 E_{2m+1}(2)=\Lambda_0 X_{2m+1}(1,2,1)+q^{\frac{1}{2}}\Lambda_1 X_{2m+1}(2,3,4),  \\
    a_{2m+1}&=\Lambda_0 E_{2m}(1)+q^{\frac{1}{2}}\Lambda_1 B_{2m}(2)=\Lambda_0 X_{2m}(1,3,4)+q^{\frac{1}{2}}\Lambda_1 X_{2m}(2,2,1), \\
    b_{2m+2}&=q^{\frac{1}{2}}\Omega_3B_{2m+1}(1)+\Omega_6E_{2m+1}(2)=q^{\frac{1}{2}}\Omega_3X_{2m+1}(1,2,1)+\Omega_6X_{2m+1}(2,3,4), \\
    b_{2m+1}&=q^{\frac{1}{2}}\Omega_3E_{2m}(1)+\Omega_6B_{2m}(2)=q^{\frac{1}{2}}\Omega_3X_{2m}(1,3,4)+\Omega_6 X_{2m}(2,2,1).
\end{split}
\end{equation}
Now we establish some identities involving the sequences $X_n(a,b,c)$ in \eqref{ab-X-2nd}.

\begin{lemma}\label{lem-X-BP-results-3}
We have
\begin{align}
    q^{-\frac{3}{2}}\sum_{m\geq 0} \frac{q^{m^2}}{(q;q)_{2m}}X_{2m}(1,3,4)=\frac{q\overline{J}_{24,90}-q^7\overline{J}_{6,90}}{J_1},  \label{eq-lem-X-BP-3-id1}\\
    \sum_{m\geq 0} \frac{q^{m^2}}{(q;q)_{2m}}X_{2m}(2,2,1)=\frac{\overline{J}_{42,90}-q^6\overline{J}_{12,90}}{J_1},  \label{eq-lem-X-BP-3-id2} \\
    \sum_{m\geq 0} \frac{q^{m^2+m}}{(q;q)_{2m+1}}X_{2m+1}(1,2,1)=\frac{\overline{J}_{39,90}-q^3\overline{J}_{21,90}}{J_1},  \label{eq-lem-X-BP-3-id3}\\
    q^{-\frac{1}{2}}\sum_{m\geq 0}\frac{q^{m^2+m}}{(q;q)_{2m+1}}X_{2m+1}(2,3,4)=\frac{\overline{J}_{33,90}-q^9\overline{J}_{3,90}}{J_1}. \label{eq-lem-X-BP-3-id4}
\end{align}
\end{lemma}
\begin{proof}
From \eqref{X-binomial-exp} we have
\begin{align}
X_{2m}(1,3,4)
&=q^{3/2}\sum_{j\in\mathbb{Z}}\left\{
q^{20j^2+11j}\qbinom{2m}{m-5j-1}
-q^{20j^2+19j+3}\qbinom{2m}{m-5j-2}
\right\}, \label{X134}\\
X_{2m}(2,2,1)
&=\sum_{j\in\mathbb{Z}}\left\{
q^{20j^2-3j}\qbinom{2m}{m-5j}
-q^{20j^2+13j+2}\qbinom{2m}{m-5j-2}
\right\}, \label{X221}\\
X_{2m+1}(1,2,1)
&=\sum_{j\in\mathbb{Z}}\left\{
q^{20j^2+j}\qbinom{2m+1}{m-5j}
-q^{20j^2+9j+1}\qbinom{2m+1}{m-5j-1}
\right\}, \label{X121}\\
X_{2m+1}(2,3,4)
&=q^{1/2}\sum_{j\in\mathbb{Z}}\left\{
q^{20j^2+7j}\qbinom{2m+1}{m-5j}
-q^{20j^2+23j+6}\qbinom{2m+1}{m-5j-2}
\right\}. \label{X234}
\end{align}
Therefore, we have
\begin{align}
& \frac{q^{-3/2}X_{2m}(1,3,4)}{(q;q)_{2m}}
=\sum_{j\geq 0}
\frac{q^{20j^2+11j}}
     {(q;q)_{m-5j-1}(q;q)_{m+5j+1}}+
\sum_{j\geq 0}
\frac{q^{20j^2+29j+9}}
     {(q;q)_{m-5j-4}(q;q)_{m+5j+4}}\nonumber\\
&\quad -\sum_{j\geq 0}
\frac{q^{20j^2+19j+3}}
     {(q;q)_{m-5j-2}(q;q)_{m+5j+2}}
-\sum_{j\geq 0}
\frac{q^{20j^2+21j+4}}
     {(q;q)_{m-5j-3}(q;q)_{m+5j+3}}, \label{BP-134}\\
&\frac{X_{2m}(2,2,1)}{(q;q)_{2m}}
=\frac{1}{(q;q)_m^2}+\sum_{j>0}
\frac{q^{20j^2-3j}+q^{20j^2+3j}}
     {(q;q)_{m-5j}(q;q)_{m+5j}}\nonumber\\
& \quad  -\sum_{j\geq 0}
\frac{q^{20j^2+13j+2}}
     {(q;q)_{m-5j-2}(q;q)_{m+5j+2}}
-\sum_{j\geq 0}
\frac{q^{20j^2+27j+9}}
     {(q;q)_{m-5j-3}(q;q)_{m+5j+3}}, \label{BP-221}\\
&\frac{X_{2m+1}(1,2,1)}{(q;q)_{2m+1}}
=\sum_{j\geq 0}
\frac{q^{20j^2+j}}
     {(q;q)_{m-5j}(q;q)_{m+5j+1}}+
\sum_{j\geq 0}
\frac{q^{20j^2+39j+19}}
     {(q;q)_{m-5j-4}(q;q)_{m+5j+5}}\nonumber\\
&\quad  -\sum_{j\geq 0}
\frac{q^{20j^2+9j+1}}
     {(q;q)_{m-5j-1}(q;q)_{m+5j+2}}
-\sum_{j\geq 0}
\frac{q^{20j^2+31j+12}}
     {(q;q)_{m-5j-3}(q;q)_{m+5j+4}}, \label{BP-121}\\
&\frac{q^{-1/2}X_{2m+1}(2,3,4)}{(q;q)_{2m+1}}
=\sum_{j\geq 0}
\frac{q^{20j^2+7j}}
     {(q;q)_{m-5j}(q;q)_{m+5j+1}}+
\sum_{j\geq 0}
\frac{q^{20j^2+33j+13}}
     {(q;q)_{m-5j-4}(q;q)_{m+5j+5}}\nonumber\\
&\quad  -\sum_{j\geq 0}
\frac{q^{20j^2+23j+6}}
     {(q;q)_{m-5j-2}(q;q)_{m+5j+3}}
-\sum_{j\geq 0}
\frac{q^{20j^2+17j+3}}
     {(q;q)_{m-5j-2}(q;q)_{m+5j+3}}. \label{BP-234}
\end{align}
From \eqref{BP-134} and \eqref{BP-221} we obtain two Bailey pairs relative to 1:
$$\Big(\alpha_m^{(3)},q^{-3/2}\frac{X_{2m}(1,3,4)}{(q;q)_{2m}}\Big), \quad \Big(\alpha_m^{(4)},\frac{X_{2m}(2,2,1)}{(q;q)_{2m}}\Big).$$
From \eqref{BP-121} and \eqref{BP-234} we obtain two Bailey pairs relative to $q$: 
$$\Big(\frac{\delta_m^{(3)}}{1-q},\frac{X_{2m+1}(1,2,1)}{(q;q)_{2m+1}}\Big), \quad \Big(\frac{\delta_m^{(4)}}{1-q},q^{-{1}/{2}}\frac{X_{2m+1}(2,3,4)}{(q;q)_{2m+1}}\Big).$$
Here $\alpha_m^{(k)}$ and $\gamma_m^{(k)}$ ($k=3,4$) are recorded in Table \ref{tab-AB34}.
\begin{table}[htbp]
\centering
\begingroup
\footnotesize
\renewcommand{\arraystretch}{1.4}
\setlength{\tabcolsep}{4pt}

\begin{adjustbox}{max width=\linewidth}
\begin{tabular}{c|cccc}
\hline
$m$ & $\alpha_m^{(3)}$ & $\alpha_m^{(4)}$ & $\gamma_m^{(3)}$ & $\gamma_m^{(4)}$
\\
\hline
$0$ & $0$ & $1$ & $1$ & $1$
\\
$5j$ & $0$ & $q^{20j^2-3j}+q^{20j^2+3j}$ & $q^{20j^2+j}$ & $q^{20j^2+7j}$
\\[3pt]
$5j+1$ & $q^{20j^2+11j}$& $0$ & $-q^{20j^2+9j+1}$ & $0$
\\
$5j+2$ & $-q^{20j^2+19j+3}$ & $-q^{20j^2+13j+2}$ & $0$ & $-q^{20j^2+17j+3}-q^{20j^2+23j+6}$
\\
$5j+3$ & $-q^{20j^2+21j+4}$ & $-q^{20j^2+27j+9}$ & $-q^{20j^2+31j+12}$ & $0$
\\
$5j+4$ & $q^{20j^2+29j+9}$ & $0$ & $q^{20j^2+39j+19}$ & $q^{20j^2+33j+13}$
\\
\hline
\end{tabular}
\end{adjustbox}

\endgroup
\caption{Values of $\alpha_m^{(k)}$ and $\gamma_m^{(k)}$ for $k=3,4$.}
\label{tab-AB34}
\end{table}

Substituting these Bailey pairs into \eqref{bailey-s1}, we obtain \eqref{eq-lem-X-BP-3-id1}--\eqref{eq-lem-X-BP-3-id4}.
\end{proof}

For the sequences $X_n(a,b,c)$ appearing later in \eqref{s-t-X}), we establish some identities similar to Lemmas \ref{lem-X-BP-result}--\ref{lem-X-BP-results-3}.
\begin{lemma}\label{lem-X-BP-results-4}
We have
\begin{align}
  q^{\frac{1}{2}} \sum_{m\geq 0} \frac{q^{m^2}}{(q;q)_{2m}}X_{2m}(1,3,2)&=\frac{q^2\overline{J}_{29,90}-q^7\overline{J}_{11,90}}{J_1}, \\
  \sum_{m\geq 0} \frac{q^{m^2}}{(q;q)_{2m}}X_{2m}(2,2,3)&=\frac{\overline{J}_{43,90}-q^8\overline{J}_{7,90}}{J_1}, \\
  q^{\frac{1}{2}}\sum_{m\geq 0} \frac{q^{m^2+m}}{(q;q)_{2m+1}}X_{2m+1}(1,2,3)&=\frac{q\overline{J}_{34,90} -q^5\overline{J}_{16,90}}{J_1}, \\
  \sum_{n\geq 0} \frac{q^{m^2+m}}{(q;q)_{2m+1}}X_{2m+1}(2,3,2)&=\frac{\overline{J}_{38,90}-q^{10}\overline{J}_{2,90}}{J_1}.
\end{align}
\end{lemma}
\begin{proof}
From \eqref{X-binomial-exp} we have
\begin{align}
X_{2m}(1,3,2)
&=q^{1/2}\sum_{j\in\mathbb{Z}}\left\{
 q^{20j^2+6j}\qbinom{2m}{m-5j-1}
-q^{20j^2+14j+2}\qbinom{2m}{m-5j-2}
\right\},   \label{X132}\\
X_{2m}(2,2,3)
&=\sum_{j\in\mathbb{Z}}\left\{
 q^{20j^2+2j}\qbinom{2m}{m-5j}
-q^{20j^2+18j+4}\qbinom{2m}{m-5j-2}
\right\},      \label{X223}\\
X_{2m+1}(1,2,3)
&=q^{1/2}\sum_{j\in\mathbb{Z}}\left\{
 q^{20j^2+6j}\qbinom{2m+1}{m-5j}
-q^{20j^2+14j+2}\qbinom{2m+1}{m-5j-1}
\right\},   \label{X123}\\
X_{2m+1}(2,3,2)
&=\sum_{j\in\mathbb{Z}}\left\{
 q^{20j^2+2j}\qbinom{2m+1}{m-5j}
-q^{20j^2+18j+4}\qbinom{2m+1}{m-5j-2}
\right\}.   \label{X232}
\end{align}
We can rewrite them as
\begin{align}
&\frac{q^{1/2}X_{2m}(1,3,2)}{(q;q)_{2m}}
=\sum_{j\geq 0}
 \frac{q^{20j^2+6j+1}}
      {(q;q)_{m-5j-1}(q;q)_{m+5j+1}}
-\sum_{j\geq 0}
 \frac{q^{20j^2+14j+3}}
      {(q;q)_{m-5j-2}(q;q)_{m+5j+2}}\nonumber\\
&-\sum_{j\geq 0}
 \frac{q^{20j^2+26j+9}}
      {(q;q)_{m-5j-3}(q;q)_{m+5j+3}}
+\sum_{j\geq 0}
 \frac{q^{20j^2+34j+15}}
      {(q;q)_{m-5j-4}(q;q)_{m+5j+4}}, \label{X132-BP} \\
&\frac{X_{2m}(2,2,3)}{(q;q)_{2m}}
=\frac{1}{(q;q)_m^2}
+\sum_{j\geq 1}
 \frac{q^{20j^2-2j}+q^{20j^2+2j}}
      {(q;q)_{m-5j}(q;q)_{m+5j}}\nonumber\\
&-\sum_{j\geq 0}
 \frac{q^{20j^2+18j+4}}
      {(q;q)_{m-5j-2}(q;q)_{m+5j+2}}
-\sum_{j\geq 0}
 \frac{q^{20j^2+22j+6}}
      {(q;q)_{m-5j-3}(q;q)_{m+5j+3}}, \label{X223-BP} \\
&\frac{q^{1/2}X_{2m+1}(1,2,3)}{(q;q)_{2m+1}}
=\sum_{j\geq 0}
 \frac{q^{20j^2+6j+1}}
      {(q;q)_{m-5j}(q;q)_{m+5j+1}}
-\sum_{j\geq 0}
 \frac{q^{20j^2+14j+3}}
      {(q;q)_{m-5j-1}(q;q)_{m+5j+2}}\nonumber\\
&-\sum_{j\geq 0}
 \frac{q^{20j^2+26j+9}}
      {(q;q)_{m-5j-3}(q;q)_{m+5j+4}}
+\sum_{j\geq 0}
 \frac{q^{20j^2+34j+15}}
      {(q;q)_{m-5j-4}(q;q)_{m+5j+5}},  \label{X123-BP} \\
&\frac{X_{2m+1}(2,3,2)}{(q;q)_{2m+1}}
=\sum_{j\geq 0}
 \frac{q^{20j^2+2j}}
      {(q;q)_{m-5j}(q;q)_{m+5j+1}}\nonumber\\
&-\sum_{j\geq 0}
 \frac{q^{20j^2+18j+4}+q^{20j^2+22j+6}}
      {(q;q)_{m-5j-2}(q;q)_{m+5j+3}}
+\sum_{j\geq 0}
 \frac{q^{20j^2+38j+18}}
      {(q;q)_{m-5j-4}(q;q)_{m+5j+5}}.  \label{X232-BP}
\end{align}
From \eqref{X132-BP} and \eqref{X223-BP} we obtain two Bailey pairs relative to $1$:
$$\Big(\alpha_m^{(5)},q^{1/2}\frac{X_{2m}(1,3,2)}{(q;q)_{2m}}\Big), \quad \Big(\alpha_m^{(6)}, \frac{X_{2m}(2,2,3)}{(q;q)_{2m}}\Big).$$
From \eqref{X123-BP} and \eqref{X232-BP} we obtain two Bailey pairs  relative to $q$:
$$\Big(\frac{\gamma_m^{(5)}}{1-q},q^{1/2}\frac{X_{2m+1}(1,2,3)}{(q;q)_{2m+1}}\Big), \quad 
\Big(\frac{\gamma_m^{(6)}}{1-q}, \frac{X_{2m+1}(2,3,2)}{(q;q)_{2m+1}}\Big).
$$
Here $\alpha_m^{(k)}$ and $\gamma_m^{(k)}$ ($k=5,6)$ are recorded in Table \ref{tab-AB56}.

\begin{table}[htbp]
\centering
\begingroup
\footnotesize
\renewcommand{\arraystretch}{1.4}
\setlength{\tabcolsep}{4pt}

\begin{adjustbox}{max width=\linewidth}
\begin{tabular}{c|cccc}
\hline
$m$ & $\alpha_m^{(5)}$ & $\alpha_m^{(6)}$ & $\gamma_m^{(5)}$ & $\gamma_m^{(6)}$
\\
\hline
$0$ & $0$ & $1$ & $q$ & $1$
\\
$5j$ & $0$ & $q^{20j^2-2j}+q^{20j^2+2j}$ & $q^{20j^2+6j+1}$ & $q^{20j^2+2j}$
\\[3pt]
$5j+1$ & $q^{20j^2+6j+1}$ & $0$ & $-q^{20j^2+14j+3}$ & $0$
\\

$5j+2$ & $-q^{20j^2+14j+3}$ & $-q^{20j^2+18j+4}$ & $0$ & $-q^{20j^2+18j+4}-q^{20j^2+22j+6}$
\\

$5j+3$ & $-q^{20j^2+26j+9}$ & $-q^{20j^2+22j+6}$ & $-q^{20j^2+26j+9}$ & $0$
\\

$5j+4$ & $q^{20j^2+34j+15}$ & $0$ & $q^{20j^2+34j+15}$ & $q^{20j^2+38j+18}$
\\
\hline
\end{tabular}
\end{adjustbox}

\endgroup
\caption{Values of $\alpha_m^{(k)}$ and $\gamma_m^{(k)}$ for $k=5,6$.}
\label{tab-AB56}
\end{table}

Substituting these Bailey pairs into \eqref{bailey-s1}, we obtain the desired identities.
\end{proof}

\subsection{Proofs of Theorems \ref{thm-ex9-new} and \ref{thm-dual-rank3}}
We first prove Theorem \ref{thm-ex9-new} which does not rely on any results in Section \ref{sec-aux}.
\begin{proof}[Proof of Theorem \ref{thm-ex9-new}]
Recall the function $T(u,v,w;q)$ defined in \eqref{T-defn}. Using \eqref{Euler} we have
\begin{align}
    T(q^4,q^2,1;q)&=\sum_{i,j\geq 0}\frac{q^{4i^2+4ij+3j^2+4i+2j}(-q^{2-2j};q^4)_{\infty}}{(q^4;q^4)_i(q^4;q^4)_j}=T_{1,0}(q)+T_{1,1}(q), \label{proof-T-1}\\
    T(1,q^{-2},1;q)&=\sum_{i,j\geq 0}\frac{q^{4i^2+4ij+3j^2-2j}(-q^{2-2j};q^4)_{\infty}}{(q^4;q^4)_i(q^4;q^4)_j}=T_{2,0}(q)+T_{2,1}(q), \label{proof-T-2} \\
    T(1,1,q^{-2};q)&=\sum_{i,j\geq 0}\frac{q^{4i^2+4ij+3j^2}(-q^{-2j};q^4)_{\infty}}{(q^4;q^4)_i(q^4;q^4)_j}=T_{3,0}(q)+X_{3,1}(q), \label{proof-T-3}\\
    T(1,1,1;q)&=\sum_{i,j\geq 0}\frac{q^{4i^2+4ij+3j^2}(-q^{2-2j};q^4)_{\infty}}{(q^4;q^4)_i(q^4;q^4)_j}=T_{4,0}(q)+T_{4,1}(q). \label{proof-T-4}
\end{align}
Here
\begin{align}
    T_{1,0}(q)&=(-q^2;q^4)_{\infty}\sum_{i,j\geq 0}\frac{q^{4i^2+8ij+10j^2+4i+4j}(-q^2;q^4)_j}{(q^4;q^4)_i(q^4;q^4)_{2j}}, \\
    T_{1,1}(q)&=(-1;q^4)_{\infty}\sum_{i,j\geq 0}\frac{q^{4i^2+8ij+10j^2+8i+14j+5}(-q^4;q^4)_j}{(q^4;q^4)_i(q^4;q^4)_{2j+1}}, \\
    T_{2,0}(q)&=(-q^2;q^4)_{\infty}\sum_{i,j\geq 0}\frac{q^{4i^2+8ij+10j^2-4j}(-q^2;q^4)_j}{(q^4;q^4)_i(q^4;q^4)_{2j}}, \\
    T_{2,1}(q)&=(-1;q^4)_{\infty}\sum_{i,j\geq 0}\frac{q^{4i^2+8ij+10j^2+4i+6j+1}(-q^4;q^4)_j}{(q^4;q^4)_i(q^4;q^4)_{2j+1}}, \\
    T_{3,0}(q)&=(-1;q^4)_{\infty}\sum_{i,j\geq 0}\frac{q^{4i^2+8ij+10j^2-2j}(-q^4;q^4)_j}{(q^4;q^4)_i(q^4;q^4)_{2j}}, \\
    T_{3,1}(q)&=(-q^2;q^4)_{\infty}\sum_{i,j\geq 0}\frac{q^{4i^2+8ij+10j^2+4i+8j+1}(-q^2;q^4)_{j+1}}{(q^4;q^4)_{i}(q^4;q^4)_{2j+1}}, \\
    T_{4,0}(q)&=(-q^2;q^4)_{\infty}\sum_{i,j\geq 0}\frac{q^{4i^2+8ij+10j^2}(-q^2;q^4)_j}{(q^4;q^4)_i(q^4;q^4)_{2j}},\\
    T_{4,1}(q)&=(-1;q^4)_{\infty}\sum_{i,j\geq 0}\frac{q^{4i^2+8ij+10j^2+4i+10j+3}(-q^4;q^4)_j}{(q^4;q^4)_i(q^4;q^4)_{2j+1}}.
\end{align}
Note that
\begin{align}\label{add-X10-proof}
T_{1,0}(q)=(-q^2;q^4)_\infty \sum_{n=0}^\infty q^{4n^2+4n} \sum_{j=0}^n \frac{q^{6j^2}(-q^2;q^4)_j}{(q^4;q^4)_{n-j}(q^4;q^4)_{2j}}.
\end{align}
Substituting the Bailey pair \eqref{dual-9-bp1} with $q^{1/2}$ replaced by $-q^{1/2}$ into \eqref{dual-9-bailey's lemma}, we deduce that
\begin{align}\label{add-X10-proof-new-mid}
&\sum_{n=0}^\infty q^{n^2+n}\sum_{j=0}^n \frac{q^{\frac{3}{2}j^2}}{(q;q)_{n-j}(q^{\frac{1}{2}};q)_j(q^2;q^2)_j} \nonumber \\
&=\frac{1}{(q;q)_\infty}\sum_{n=0}^\infty (-1)^{\frac{1}{2}n(n-1)}(1-q^{2n+1})q^{\frac{9}{4}n^2+\frac{5}{4}n} \nonumber \\
&=\frac{1}{(q;q)_\infty}\sum_{n=-\infty}^\infty (-1)^{\frac{1}{2}n(n-1)}q^{\frac{9}{4}n^2+\frac{5}{4}n} \nonumber \\
&=\frac{1}{(q;q)_\infty}\sum_{n=-\infty}^\infty \Big((-1)^nq^{9n^2+\frac{5}{2}n}-(-1)^nq^{9n^2-\frac{13}{2}n+1}  \Big) \nonumber \\
&=\frac{1}{J_1}\big(J_{\frac{13}{2},18}-qJ_{\frac{5}{2},18}\big)=\frac{J_{1,18}J_{7,18}J_{8,18}J_{9,18}}{J_1J_{\frac{7}{2},18}J_{\frac{9}{2},18}J_{\frac{11}{2},18}}.
\end{align}
Here the last equality can be proved directly using the Maple approach in \cite{Frye-Garvan}.

Substituting \eqref{add-X10-proof-new-mid} with $q$ replaced by $q^4$ into \eqref{add-X10-proof}, we deduce that
\begin{align}
     T_{1,0}(q)&=\frac{J_4J_{4,72}J_{28,72}J_{32,72}J_{36,72}}{J_2J_8J_{14,72}J_{18,72}J_{22,72}}. \label{X10-result}
\end{align}

Similarly, substituting the Bailey pairs \eqref{dual-9-bp5}--\eqref{dual-9-bp2} into \eqref{dual-9-bailey's lemma}, after simplifications we have
\begin{align}
    T_{1,1}(q)&=\frac{2q^5J_8J_{8,72}}{J_4^2}, \label{X1-result}\\
    T_{2,0}(q)&=\frac{J_4J_{24}J_{12,24}}{J_2J_8J_{6,24}}, \quad 
    T_{2,1}(q)=\frac{2qJ_8J_{24}}{J_4^2}, \label{X2-result}\\
    T_{3,0}(q)&=\frac{2J_8J_{32,72}}{J_4^2}, \quad 
    T_{3,1}(q)=\frac{qJ_4J_{4,72}J_{16,72}J_{20,72}J_{36,72}}{J_2J_8J_{2,72}J_{18,72}J_{34,72}}, \label{X3-result}\\
    T_{4,0}(q)&=\frac{J_4J_{8,72}J_{20,72}J_{28,72}J_{36,72}}{J_2J_8J_{10,72}J_{18,72}J_{26,72}}, \quad 
    T_{4,1}(q)=\frac{2q^3J_8J_{16,72}}{J_4^2}. \label{X4-result}
\end{align}
Substituting \eqref{X10-result}--\eqref{X4-result} into \eqref{proof-T-1}--\eqref{proof-T-4}, we obtain \eqref{eq-thm-T-1}--\eqref{eq-thm-T-4}.
\end{proof}

Next, we are going to prove Theorem \ref{thm-dual-rank3}. We first explain how we discovered it. 
Let
\begin{align}
 F(u,v,w)=F(u,v,w;q):=\sum_{i,j,k\geq 0}\frac{q^{5i^2+8j^2+8k^2-8ij-4ik+8jk}u^iv^jw^k}{(q^{12};q^{12})_i(q^{12};q^{12})_j(q^{12};q^{12})_k}.
\end{align}
The dual Nahm sums are given by
\begin{align}
&F^{(1)}(q):=F(q^6,1,1), \quad F^{(2)}(q):=F(q^4,q^{-8},q^{-4}), \label{dual-F1F2} \\
&F^{(3)}(q):=F(q^2,q^{-4},q^{-8}), \quad F^{(4)}(q):=F(1,1,1). \label{dual-F3F4}
\end{align}
In order to find product representations for them, we find that it is helpful to consider their 3-dissections: for $i=1,2,3,4$,
\begin{align}\label{3-dissection}
    F^{(i)}(q)=S_0^{(i)}(q^3)+qS_1^{(i)}(q^3)+q^2S_2^{(i)}(q^3), \quad S_0^{(i)}(q), S_1^{(i)}(q),S_2^{(i)}(q)\in \mathbb{Z}[[q]].
\end{align}
Let 
\begin{align*}
    Q(i,j,k):=5i^2+8j^2+8k^2-8ij-4ik+8jk, \quad m=i+j-k.
\end{align*}
Note that
\begin{align}
&Q(i,j,k)\equiv Q(i,j,k)+6i\equiv 2m^2 \equiv \begin{cases}
2, & m\equiv 1,2\pmod{3}, \\
0, & m\equiv 0 \pmod{3};
\end{cases} \label{exp-mod3-1} \\
&Q(i,j,k)+4i-8j-4k \equiv 2m^2+m \equiv
\begin{cases}
0, & m\equiv 0, 1 \pmod{3}, \\
1, & m\equiv 2 \pmod{3};
\end{cases}  \label{exp-mod3-3} \\
&Q(i,j,k)+2i-4j-8k \equiv 2m^2+2m \equiv
\begin{cases}
0, & m\equiv 0, 2 \pmod{3}, \\
1, & m\equiv 1 \pmod{3}.
\end{cases}  \label{exp-mod3-2} 
\end{align}
From \eqref{exp-mod3-1}--\eqref{exp-mod3-3} and the definition we deduce that
\begin{align}\label{S-zero}
    S_1^{(1)}(q)=S_2^{(2)}(q)=S_2^{(3)}(q)=S_1^{(4)}(q)=0.
\end{align}
Consequently, only two components in the 3-dissection \eqref{3-dissection} are nonzero. Maple computations suggest that each of these components can be expressed either as a single infinite product or as the sum of two infinite products. This leads us to the discovery of Theorem \ref{thm-dual-rank3}. However, its proof is much more complicated and we need the results in Section \ref{sec-aux}. The key observation is 
\begin{equation}
Q(i,j,k)=3i^2+6k^2+2(2j-i+k)^2.     \label{key-square}
\end{equation}
\begin{proof}[Proof of Theorem \ref{thm-dual-rank3}]
We shall first prove \eqref{eq-rank3-dual-2} and \eqref{eq-rank3-dual-4} together,  and then we prove \eqref{eq-rank3-dual-1} and \eqref{eq-rank3-dual-3}.

(1) For $s=0,1$ we have
\begin{align}
    F^{(4-2s)}(q)=\sum_{i\geq 0} \frac{q^{\frac{1}{4}i^2+\frac{1}{2}si}}{(q;q)_i}\sum_{j,k\geq 0} \frac{q^{\frac{1}{2}k^2+\frac{1}{6}(2j+k-i)^2}}{(q;q)_j(q;q)_k}.
\end{align}
Hence 
\begin{align}
    &F^{(4-2s)}(q)=\sum_{\sigma \in \{-1,0,1\}} \sum_{i\geq 0} \frac{q^{\frac{1}{4}i^2+\frac{1}{2}si}}{(q;q)_i}\sum_{\substack{j,k\geq 0 \\ 2j+k-i\equiv \sigma \!\!\!\pmod{3}}} \frac{q^{\frac{1}{2}k^2+\frac{1}{6}(2j+k-i)^2}}{(q;q)_j(q;q)_k} \nonumber \\
    &=\sum_{\sigma \in \{-1,0,1\}} \sum_{i\geq 0} \frac{q^{\frac{1}{4}i^2+\frac{1}{2}si}}{(q;q)_i} \sum_{r\in \mathbb{Z}} q^{\frac{1}{6}(3r+\sigma)^2} \sum_{2j+k=3r+\sigma+i} \frac{q^{\frac{1}{2}k^2}}{(q;q)_j(q;q)_k} \nonumber \\
    &=\sum_{\sigma\in \{-1,0,1\}} q^{\frac{1}{6}\sigma^2}\sum_{i\geq 0} \frac{q^{\frac{1}{4}i^2+\frac{1}{2}si}}{(q;q)_i} \sum_{r\in \mathbb{Z}} q^{\frac{3}{2}r^2+\sigma r} H_{3r+\sigma+i} \nonumber \\
    &=\sum_{i\geq 0} \frac{q^{\frac{1}{4}i^2+\frac{1}{2}si}}{(q;q)_i}\Big(a_i+q^{\frac{1}{6}}(b_i+c_i)\Big)  \nonumber \\
    &=\sum_{i\geq 0} \frac{q^{\frac{1}{4}i^2+\frac{1}{2}si}}{(q;q)_i}\Big(a_i+2q^{\frac{1}{6}}b_i\Big) \quad \text{(by Lemma \ref{lem-rec})} \nonumber \\
   &=\sum_{m\geq 0} \frac{q^{m^2+sm}}{(q;q)_{2m}}\big(a_{2m}+2q^{\frac{1}{6}}b_{2m}\big)+ q^{\frac{2s+1}{4}}\sum_{m\geq 0} \frac{q^{m^2+(s+1)m}}{(q;q)_{2m+1}}\big(a_{2m+1}+2q^{\frac{1}{6}}b_{2m+1}\big).
\end{align}
Here for the last step we split the sum according to the parity of $i$.

Using \eqref{ab-X-1st}, we deduce that
\begin{align}
  &  F^{(4-2s)}(q) =(\Lambda_0+2q^{\frac{2}{3}}\Omega_3)\sum_{m\geq 0} \frac{q^{m^2+sm}}{(q;q)_{2m}} X_{2m}(1,1,2) \nonumber \\
  &\quad +(q^{\frac{1}{2}}\Lambda_1+2q^{\frac{1}{6}}\Omega_6)\sum_{m\geq 0} \frac{q^{m^2+sm}}{(q;q)_{2m}}X_{2m}(2,4,3) \nonumber \\
    &\quad +q^{\frac{2s+1}{4}}\big(\Lambda_0+2q^{\frac{2}{3}}\Omega_3\big)\sum_{m\geq 0} \frac{q^{m^2+(s+1)m}}{(q;q)_{2m+1}}X_{2m+1}(1,4,3) \nonumber \\
 &\quad   +q^{\frac{2s+1}{4}}\big(q^{\frac{1}{2}}\Lambda_1+2q^{\frac{1}{6}}\Omega_6\big) \sum_{m\geq 0} \frac{q^{m^2+(s+1)m}}{(q;q)_{2m+1}}X_{2m+1}(2,1,2).
\end{align}
For $s=0,1$, substituting Lemmas \ref{lem-X-BP-result} and \ref{lem-X-BP-results-2} into it, we obtain expressions for $F^{(4)}(q)$ and $F^{(2)}(q)$ in terms of $J_m$ and $J_{a,m}$.

(2) We have
\begin{align}
    &F^{(1)}(q)=\sum_{i\geq 0} \frac{q^{\frac{1}{4}i^2}}{(q;q)_i}\sum_{j,k\geq 0} \frac{q^{\frac{1}{2}k^2+\frac{1}{6}(2j+k-i)^2-\frac{1}{3}(2j+k-i)}}{(q;q)_j(q;q)_k} \nonumber \\
     &=\sum_{\sigma\in \{0,1,2\}} \sum_{i\geq 0} \frac{q^{\frac{1}{4}i^2}}{(q;q)_i} \sum_{r\in \mathbb{Z}} q^{\frac{1}{6}(3r+\sigma)^2-\frac{1}{3}(3r+\sigma)} \sum_{2j+k=3r+\sigma+i} \frac{q^{\frac{1}{2}k^2}}{(q;q)_j(q;q)_k} \nonumber \\
    &=\sum_{\sigma\in \{0,1,2\}} q^{\frac{1}{6}\sigma^2-\frac{1}{3}\sigma}\sum_{i\geq 0} \frac{q^{\frac{1}{4}i^2}}{(q;q)_i} \sum_{r\in \mathbb{Z}} q^{\frac{3}{2}r^2+(\sigma-1)r} H_{3r+\sigma+i} \nonumber \\
    &=\sum_{i\geq 0}\frac{q^{\frac{1}{4}i^2}}{(q;q)_i}\big(c_{i+1}+q^{-\frac{1}{6}}a_{i+1}+b_{i+1}\big) \nonumber \\
    &=\sum_{i\geq 0}\frac{q^{\frac{1}{4}i^2}}{(q;q)_i}\big(q^{-\frac{1}{6}}a_{i+1}+2b_{i+1}\big)  \quad \text{(by Lemma \ref{lem-rec})} \nonumber \\
    &=\sum_{m\geq 0} \frac{q^{m^2}}{(q;q)_{2m}}\Big(q^{-\frac{1}{6}}a_{2m+1}+2b_{2m+1}\Big) \nonumber \\
    &\quad +q^{\frac{1}{4}}\sum_{m\geq 0} \frac{q^{m^2+m}}{(q;q)_{2m+1}}\Big(q^{-\frac{1}{6}}a_{2m+2}+2b_{2m+2}\Big). \label{F1-reduce-start}
\end{align}
Here for the last step we split the sum according to the parity of $i$.

Substituting \eqref{ab-X-2nd} into \eqref{F1-reduce-start}, we have
\begin{align}
    &F^{(1)}(q)=(q^{-\frac{1}{6}}\Lambda_0+2q^{\frac{1}{2}}\Omega_3)\sum_{m\geq 0} \frac{q^{m^2}}{(q;q)_{2m}}X_{2m}(1,3,4) \nonumber \\
    &\quad +(q^{\frac{1}{3}}\Lambda_1+2\Omega_6)\sum_{m\geq 0} \frac{q^{m^2}}{(q;q)_{2m}}X_{2m}(2,2,1) \nonumber \\
    &\quad +q^{\frac{1}{4}}(q^{-\frac{1}{6}}\Lambda_0+2q^{\frac{1}{2}}\Omega_3)\sum_{m\geq 0} \frac{q^{m^2+m}}{(q;q)_{2m+1}}X_{2m+1}(1,2,1) \nonumber \\
& \quad   +q^{\frac{1}{4}}(q^{\frac{1}{3}}\Lambda_1+2\Omega_6)\sum_{m\geq 0} \frac{q^{m^2+m}}{(q;q)_{2m+1}}X_{2m+1}(2,3,4).
\end{align}
Substituting Lemma \ref{lem-X-BP-results-3} into it, we arrived at an expression of $F^{(1)}(q)$ in terms of $J_m$ and $J_{a,m}$.  

(3) We have
\begin{align}
    &F^{(3)}(q)=\sum_{i\geq 0} \frac{q^{\frac{1}{4}i^2}}{(q;q)_i}\sum_{j,k\geq 0} \frac{q^{\frac{1}{2}k^2-\frac{1}{2}k+\frac{1}{6}(2j+k-i)^2-\frac{1}{6}(2j+k-i)}}{(q;q)_j(q;q)_k} \nonumber \\
   &=\sum_{\sigma \in \{1,2,3\}} \sum_{i\geq 0} \frac{q^{\frac{1}{4}i^2}}{(q;q)_i} \sum_{r\in \mathbb{Z}} q^{\frac{1}{6}(3r+\sigma)^2-\frac{1}{6}(3r+\sigma)} \sum_{2j+k=3r+\sigma+i} \frac{q^{\frac{1}{2}(k^2-k)}}{(q;q)_j(q;q)_k} \nonumber \\
   &=\sum_{\sigma\in \{1,2,3\}}q^{\frac{1}{6}(\sigma^2-\sigma)}\sum_{i\geq 0} \frac{q^{\frac{1}{4}i^2}}{(q;q)_i} \sum_{r\in \mathbb{Z}} q^{\frac{3}{2}r^2+(\sigma-\frac{1}{2})r} \widetilde{H}_{i+3r+\sigma} \nonumber \\
   &=\sum_{\sigma\in \{1,2,3\}}q^{\frac{1}{6}\sigma^2-\frac{2}{3}\sigma}\sum_{i\geq 0} \frac{q^{\frac{1}{4}i^2-\frac{1}{2}i}}{(q;q)_i} \sum_{r\in \mathbb{Z}} q^{\frac{3}{2}r^2+(\sigma-2)r} (H_{i+3r+\sigma}-H_{i+3r+\sigma-2})  \quad \text{(by \eqref{wH-H-relation})}\nonumber \\
   &=q^{-\frac{1}{2}}\sum_{i\geq 0} \frac{q^{\frac{1}{4}i^2-\frac{1}{2}i}}{(q;q)_i}\big(c_{i+2}-c_i)+q^{-\frac{2}{3}}\sum_{i\geq 0} \frac{q^{\frac{1}{4}i^2-\frac{1}{2}i}}{(q;q)_i}\big(a_{i+2}-a_i\big)+q^{-\frac{1}{2}}\sum_{i\geq 0}\frac{q^{\frac{1}{4}i^2-\frac{1}{2}i}}{(q;q)_i}\big(b_{i+2}-b_i\big) \nonumber \\
   &=q^{-\frac{2}{3}}\sum_{i\geq 0} \frac{q^{\frac{1}{4}i^2-\frac{1}{2}i}}{(q;q)_i}\big(a_{i+2}-a_i\big)+2q^{-\frac{1}{2}}\sum_{i\geq 0} \frac{q^{\frac{1}{4}i^2-\frac{1}{2}i}}{(q;q)_i}\big(b_{i+2}-b_i\big) \quad \text{(by Lemma \ref{lem-rec})} \nonumber \\
   &=q^{-\frac{1}{6}}\sum_{i\geq 0} \frac{q^{\frac{1}{4}i^2}}{(q;q)_i}s_i+2\sum_{i\geq 0} \frac{q^{\frac{1}{4}i^2}}{(q;q)_i}t_i \nonumber \\
    &=\sum_{m\geq 0} \frac{q^{m^2}}{(q;q)_{2m}}(q^{-\frac{1}{6}}s_{2m}+2t_{2m})+q^{\frac{1}{4}}\sum_{m\geq 0} \frac{q^{m^2+m}}{(q;q)_{2m+1}}(q^{-\frac{1}{6}}s_{2m+1}+2t_{2m+1}), \label{F3-reduce}
\end{align}
where $s_n:=q^{-\frac{n+1}{2}}(a_{n+2}-a_n)$ and $t_n:=q^{-\frac{n+1}{2}}(b_{n+2}-b_n)$.

From \eqref{STU} we deduce that
\begin{align}
    S_{n+2}-S_n=q^{\frac{n+1}{2}}U_n.
\end{align}
Therefore, by \eqref{abc-S} we obtain
\begin{align}
    s_n=\Lambda_0 U_n(1)+q^{\frac{1}{2}}\Lambda_1 U_n(2), \quad t_n=q^{\frac{1}{2}}\Omega_3 U_n(1)+\Omega_6 U_n(2).
\end{align}
It follows that
\begin{equation}\label{s-t-X}
\begin{split}
    s_{2m}&=\Lambda_0 X_{2m}(1,3,2)+q^{\frac{1}{2}}\Lambda_1 X_{2m}(2,2,3), \\
    s_{2m+1}&=\Lambda_0 X_{2m+1}(1,2,3)+q^{\frac{1}{2}}\Lambda_1 X_{2m+1}(2,3,2), \\
    t_{2m}&=q^{\frac{1}{2}}\Omega_3 X_{2m}(1,3,2)+\Omega_6 X_{2m}(2,2,3), \\
    t_{2m+1}&=q^{\frac{1}{2}}\Omega_3 X_{2m+1}(1,2,3)+\Omega_6 X_{2m+1}(2,3,2).
\end{split}
\end{equation}
Substituting them into \eqref{F3-reduce}, we obtain
\begin{align}
    &F^{(3)}(q)=\big(q^{-\frac{1}{6}}\Lambda_0+2q^{\frac{1}{2}}\Omega_3\big)\sum_{m\geq 0} \frac{q^{m^2}}{(q;q)_{2m}}X_{2m}(1,3,2) \nonumber \\
&\quad +\big(q^{\frac{1}{3}}\Lambda_1+2\Omega_6   \big)\sum_{m\geq 0} \frac{q^{m^2}}{(q;q)_{2m}}X_{2m}(2,2,3)
\nonumber \\
    &\quad +q^{\frac{1}{4}}\big(q^{-\frac{1}{6}}\Lambda_0+2q^{\frac{1}{2}}\Omega_3\big)\sum_{m\geq 0} \frac{q^{m^2+m}}{(q;q)_{2m+1}}X_{2m+1}(1,2,3) \nonumber \\
   &\quad +q^{\frac{1}{4}}\big(q^{\frac{1}{3}}\Lambda_1+2\Omega_6\big)\frac{q^{m^2+m}}{(q;q)_{2m+1}}X_{2m+1}(2,3,2). 
\end{align}
Substituting Lemma \ref{lem-X-BP-results-4} into it,   we arrived at an expression of $F^{(3)}(q)$ in terms of $J_m$ and $J_{a,m}$.

By the Maple approach in \cite{Frye-Garvan}, it is easy to check that the the above representations for $F^{(k)}(q)$ ($k=1,2,3,4$) agree with the expressions in \eqref{eq-rank3-dual-1}--\eqref{eq-rank3-dual-4}.
\end{proof}


\section{Proofs of some conjectural Nahm sum identities}\label{sec-proof}

In this section, we prove some conjectural identities  proposed by Cao--Wang \cite{Cao-Wang-rank4}, Li \cite{Li-arXiv} and Shi--Wang \cite{Shi-Wang}. They were stated in Section \ref{sec-intro} as Theorems~\ref{thm-CSW-conj}--\ref{thm-CW-conj3.3} and \eqref{eq-Li-1}--\eqref{eq-Li-2}.

\subsection{Proof of some conjectural identities on tadpole Nahm sums}
Based on the proof in the dual of Example 11 and the identities for the dual of Example 8 (see Theorem \ref{thm-dual-8}), we now prove the conjectural identities on rank four tadpole Nahm sums in \eqref{id-SW} and \eqref{id-CW-1}--\eqref{id-CW-3}.
\begin{proof}[Proof of Theorem \ref{thm-CSW-conj}]
Recall the identity \eqref{dual-ex11-reducingrank}. The identities \eqref{id-SW} and \eqref{id-CW-1}--\eqref{id-CW-3} correspond to four assignments of $(z_1,z_2,z_3,z_4)$ which we will discuss separately.

(1) Let $(z_1,z_2,z_3,z_4)=(q,q^{-1},q,q^{-\frac{1}{2}})$. In this case, \eqref{id-SW} follows from \eqref{dual-11-3to4to3}.

(2) Let $(z_1,z_2,z_3,z_4)=(q^2,q^{-2},q,q^{-\frac{1}{2}})$.
From \eqref{dual-ex11-reducingrank} we have
\begin{align}\label{chi-1}
    \chi(q^2,q^{-2},q,q^{-\frac{1}{2}};q)=\frac{R_1(q)}{(q;q)_\infty}
\end{align}
where
\begin{align}\label{R1-start}
    R_1(q)&:=\sum_{m,n,r\geq 0}\frac{q^{m^2+\frac{3}{4}n^2+\frac{1}{2}r^2+mn-mr-nr+m-n-\frac{r}{2}}}{(q;q)_m(q;q)_n(q;q)_r}\sum_{s\in\mathbb{Z}}q^{(s+\frac{n}{2})^2-2s}\nonumber\\
    &=\sum_{m,n,r\geq 0}\frac{q^{m^2+\frac{3}{4}n^2+\frac{1}{2}r^2+mn-mr-nr+m-n-\frac{r}{2}}}{(q;q)_m(q;q)_n(q;q)_r}\sum_{s\in\mathbb{Z}}q^{(s-1+\frac{n}{2})^2-1+n}\nonumber\\
    &=q^{-1}\sum_{m,n,r\geq 0}\frac{q^{m^2+\frac{3}{4}n^2+\frac{1}{2}r^2+mn-mr-nr+m-\frac{r}{2}}}{(q;q)_m(q;q)_n(q;q)_r}\sum_{s\in\mathbb{Z}}q^{(s+\frac{n}{2})^2}.
\end{align}
By \eqref{id-dual-8-1} we have
\begin{align}
    &S_1(q):=\sum_{m,n,r\geq 0}\frac{q^{4m^2+3n^2+2r^2+4mn-4mr-4nr+4m-2r}}{(q^4;q^4)_m(q^4;q^4)_n(q^4;q^4)_r}=2q^{-1}\frac{J_2^5J_8^2}{J_1^2J_4^5}, \label{tadpole-S1-result}\\
   & S_1(-q)=\sum_{m,n,r\geq 0}\frac{(-1)^nq^{4m^2+3n^2+2r^2+4mn-4mr-4nr+4m-2r}}{(q^4;q^4)_m(q^4;q^4)_n(q^4;q^4)_r}=-2q^{-1}\frac{J_{1,8}^2J_{2,8}J_{3,8}^2}{J_4^2J_8^3}. \label{tadpole-S1-minus-result}
\end{align}
From \eqref{R1-start} we have
\begin{align}\label{R1-split}
    q^4R_1(q^4)=\Big(\frac{S_1(q)+S_1(-q)}{2}\Big)\sum_{s\in \mathbb{Z}}q^{4s^2}+\Big(\frac{S_1(q)-S_1(-q)}{2}\Big)\sum_{s\in \mathbb{Z}}q^{(2s+1)^2}.
\end{align}
Substituting \eqref{tadpole-S1-result} and \eqref{tadpole-S1-minus-result} into \eqref{R1-split}, we obtain \eqref{id-CW-1} from \eqref{chi-1}.

(3) Let $(z_1,z_2,z_3,z_4)=(q^{\frac{3}{2}}, q^{-\frac{3}{2}}, q^{\frac{1}{2}},1)$. From \eqref{dual-ex11-reducingrank} we have
\begin{align}\label{chi-2}
\chi(q^{\frac{3}{2}}, q^{-\frac{3}{2}}, q^{\frac{1}{2}},1;q)=\frac{R_2(q)}{(q;q)_\infty}
\end{align}
where
\begin{align}\label{R2-start}
    R_2(q)&:=\sum_{m,n,r\geq 0}\frac{q^{m^2+\frac{3}{4}n^2+\frac{1}{2}r^2+mn-mr-nr+\frac{m}{2}-n}}{(q;q)_m(q;q)_n(q;q)_r}\sum_{s\in\mathbb{Z}}q^{(s+\frac{n}{2})^2-\frac{3s}{2}}\nonumber\\
    &=\sum_{m,n,r\geq 0}\frac{q^{m^2+\frac{3}{4}n^2+\frac{1}{2}r^2+mn-mr-nr+\frac{m}{2}-n}}{(q;q)_m(q;q)_n(q;q)_r}\sum_{s\in\mathbb{Z}}q^{(s+\frac{n}{2}-\frac{3}{4})^2-\frac{9}{16}+\frac{3n}{4}}\nonumber\\
    &=q^{-\frac{9}{16}}\sum_{m,n,r\geq 0}\frac{q^{m^2+\frac{3}{4}n^2+\frac{1}{2}r^2+mn-mr-nr+\frac{m}{2}-\frac{n}{4}}}{(q;q)_m(q;q)_n(q;q)_r}\sum_{s\in\mathbb{Z}}q^{(s-\frac{1}{4})^2}.
\end{align}
By \eqref{id-dual-8-2} we have 
\begin{align}\label{tadpole-S2-result}
    S_2(q)&:=\sum_{m,n,r\geq 0}\frac{q^{4m^2+3n^2+2r^2+4mn-4mr-4nr+2m-n}}{(q^4;q^4)_m(q^4;q^4)_n(q^4;q^4)_r}=\frac{J_4^2}{J_2^2}.
\end{align}
From \eqref{R2-start} we have
\begin{align}\label{R2-split}
    q^{\frac{9}{16}}R_2(q)=S_2(q^{\frac{1}{4}})\sum_{s\in\mathbb{Z}}q^{(s-\frac{1}{4})^2}.
\end{align}
Substituting \eqref{tadpole-S2-result} into \eqref{R2-split}, we obtain \eqref{id-CW-2} from \eqref{chi-2}.

(4) Let $(z_1,z_2,z_3,z_4)=(q,q^{-1},1,q^{\frac{1}{2}})$. From \eqref{dual-ex11-reducingrank} we have
\begin{align}\label{chi-3}
\chi(q,q^{-1},1,q^{\frac{1}{2}};q)=\frac{R_3(q)}{(q;q)_\infty}   
\end{align}
where
\begin{align}\label{R3-start}
    R_3(q)&:=\sum_{m,n,r\geq 0}\frac{q^{m^2+\frac{3}{4}n^2+\frac{1}{2}r^2+mn-mr-nr-n+\frac{r}{2}}}{(q;q)_m(q;q)_n(q;q)_r}\sum_{s\in\mathbb{Z}}q^{(s+\frac{n}{2})^2-s}\nonumber\\
    &=\sum_{m,n,r\geq 0}\frac{q^{m^2+\frac{3}{4}n^2+\frac{1}{2}r^2+mn-mr-nr-n+\frac{r}{2}}}{(q;q)_m(q;q)_n(q;q)_r}\sum_{s\in\mathbb{Z}}q^{(s+\frac{n-1}{2})^2-\frac{1}{4}+\frac{n}{2}}\nonumber\\
    &=q^{-\frac{1}{4}}\sum_{m,n,r\geq 0}\frac{q^{m^2+\frac{3}{4}n^2+\frac{1}{2}r^2+mn-mr-nr-\frac{n}{2}+\frac{r}{2}}}{(q;q)_m(q;q)_n(q;q)_r}\sum_{s\in\mathbb{Z}}q^{(s+\frac{n-1}{2})^2}.
\end{align}
By \eqref{id-dual-8-7} we have 
\begin{align}
   & S_3(q):=\sum_{m,n,r\geq 0}\frac{q^{4m^2+3n^2+2r^2+4mn-4mr-4nr-2n+2r}}{(q^4;q^4)_m(q^4;q^4)_n(q^4;q^4)_r}=\frac{J_2^5J_8^2}{J_1^2J_4^5},  \label{tadpole-S3-result}\\
   &S_3(-q)=\sum_{m,n,r\geq 0}\frac{(-1)^nq^{4m^2+3n^2+2r^2+4mn-4mr-4nr-2n+2r}}{(q^4;q^4)_m(q^4;q^4)_n(q^4;q^4)_r}=\frac{J_{1,8}^2J_{2,8}J_{3,8}^2}{J_4^2J_8^3}. \label{tadplole-S3-minus-result}
\end{align}
From \eqref{R3-start} we have
\begin{align}\label{R3-split}
    qR_3(q^4)=\Big(\frac{S_3(q)-S_3(-q)}{2}\Big)\sum_{s\in \mathbb{Z}}q^{4s^2}+\Big(\frac{S_3(q)+S_3(-q)}{2}\Big)\sum_{s\in \mathbb{Z}}q^{(2s+1)^2}.
\end{align}
Substituting \eqref{tadpole-S3-result} and \eqref{tadplole-S3-minus-result} into \eqref{R3-split}, we obtain \eqref{id-CW-3} from \eqref{chi-3}.
\end{proof}

\subsection{Proofs of some conjectures of Cao--Wang and Li}

We are now going to prove Theorem \ref{thm-CW-conj3.3}.
Recall the function $N(u,v,w;q)$ defined in \eqref{eq-N-defn}. 
Firstly, we split the series $N(u,v,w;q)$ into two parts by the parity of $k$:
\begin{align}
    N(u,v,w;q)=N_0(u,v,w;q)+N_1(u,v,w;q).
\end{align}
Here
\begin{align}\label{eq-N-parity}
    N_a(u,v,w;q):=\sum_{\substack{i,j,k,l\geq 0 \\ k\equiv a\!\!\! \pmod{2}}}\frac{q^{2i^2+2j^2+k^2+l^2-2ij-2jk-kl}u^{i-j}v^kw^l}{(q^2;q^2)_i(q^2;q^2)_j(q^2;q^2)_k(q^2;q^2)_l}.
\end{align}

Using \eqref{Euler} to sum over $l$, we have
\begin{align}
    &N_0(u,v,w;q)=\sum_{i,j,k,l\geq 0}\frac{q^{2i^2+2j^2+4k^2+l^2-2ij-4jk-2kl}u^{i-j}v^{2k}w^l}{(q^2;q^2)_i(q^2;q^2)_j(q^2;q^2)_{2k}(q^2;q^2)_l}\nonumber\\
    &=\sum_{i,j,k\geq 0}\frac{q^{2i^2+2j^2+4k^2-2ij-4jk}u^{i-j}v^{2k}}{(q^2;q^2)_i(q^2;q^2)_j(q^2;q^2)_{2k}}\cdot (-q^{1-2k}w;q^2)_{\infty},\\
    &N_1(u,v,w;q)=\sum_{i,j,k,l\geq 0}\frac{q^{2i^2+2j^2+4k^2+l^2-2ij-4jk-2kl-2j+4k-l+1}u^{i-j}v^{2k+1}w^l}{(q^2;q^2)_i(q^2;q^2)_j(q^2;q^2)_{2k+1}(q^2;q^2)_l}\nonumber\\
    &=\sum_{i,j,k\geq 0}\frac{q^{2i^2+2j^2+4k^2-2ij-4jk-2j+4k+1}u^{i-j}v^{2k+1}}{(q^2;q^2)_i(q^2;q^2)_j(q^2;q^2)_{2k+1}}\cdot (-q^{-2k}w;q^2)_{\infty}.
\end{align}

To simplify $N_0(u,v,w;q)$ and $N_1(u,v,w;q)$, we give the following lemma:
\begin{lemma}
    We have
    \begin{align}
        N_0(u,v,w;q)&=\frac{(-qw;q^2)_{\infty}}{(q^2;q^2)_{\infty}}\sum_{\substack{k\geq 0, t\in\mathbb{Z}\\0\leq i\leq 2k}}\frac{q^{3k^2+2i^2-4ki+2t^2-2ti}(v^2w)^ku^{t-i}(-qw^{-1};q^2)_k}{(q^2;q^2)_i(q^2;q^2)_{2k-i}}, \label{add-N0-exp}\\
        N_1(u,v,w;q)&=\frac{(-w;q^2)_{\infty}}{(q^2;q^2)_{\infty}}\sum_{\substack{k\geq 0, t\in\mathbb{Z}\\ 0\leq i\leq2k+1}}\frac{q^{3k^2+2i^2-4ki+2t^2-2ti+3k-2i+1}v^{2k+1}w^k(-q^2w^{-1};q^2)_k}{(q^2;q^2)_i(q^2;q^2)_{2k+1-i}}.\label{add-N1-exp}
    \end{align}
\end{lemma}
\begin{proof}
By the fact that
\begin{align}
    (-q^{1-2k}w;q^2)_{\infty}&=(-qw;q^2)_{\infty}(-qw^{-1};q^2)_kw^kq^{-k^2}, \label{finite-product-1}\\
    (-q^{-2k}w;q^2)_{\infty}&=(-w;q^2)_{\infty}(-q^2w^{-1};q^2)_kw^kq^{-k^2-k}, \label{finite-product-2}
\end{align}
we have
\begin{align}
    &N_0(u,v,w;q)=(-qw;q^2)_{\infty}\sum_{i,j,k\geq 0}\frac{q^{2i^2+2j^2+3k^2-2ij-4jk}u^{i-j}v^{2k}w^k(-qw^{-1};q^2)_k}{(q^2;q^2)_i(q^2;q^2)_j(q^2;q^2)_{2k}},\\
    &N_1(u,v,w;q)= (-w;q^2)_{\infty}\nonumber\\
    & \quad \times \sum_{i,j,k\geq 0}\frac{q^{2i^2+2j^2+3k^2-2ij-4jk-2j+3k+1}u^{i-j}v^{2k+1}w^k(-q^2w^{-1};q^2)_k}{(q^2;q^2)_i(q^2;q^2)_j(q^2;q^2)_{2k+1}}.
\end{align}
Note that the quadratic form in the power of $q$ can be written as
\begin{align}
    2i^2+2j^2+3k^2-2ij-4jk=(i-k)^2+(j-k)^2+(i-j+k)^2.
\end{align}
We have
\begin{align}\label{A0-step1}
    &N_0(u,v,w;q) \nonumber \\
    &=(-qw;q^2)_{\infty}\mathrm{CT}_z\Big[ \sum_{\substack{i,j,k\geq 0\\ s\in\mathbb{Z}}}\frac{q^{(i-k)^2+(j-k)^2+s^2}u^{i-j}v^{2k}w^k(-qw^{-1};q^2)_kz^{i-j+k-s}}{(q^2;q^2)_i(q^2;q^2)_j(q^2;q^2)_{2k}} \Big]\nonumber\\
    &=(-qw;q^2)_{\infty}\mathrm{CT}_z\big[ \sum_{s\in\mathbb{Z}}q^{s^2}z^{-s}\sum_{k\geq 0}\frac{q^{2k^2}(v^2wz)^k(-qw^{-1};q^2)_k}{(q^2;q^2)_{2k}}\nonumber\\
    &\quad \times (-q^{-2k+1}uz,-q^{-2k+1}(uz)^{-1};q^2)_{\infty} \big]\nonumber\\
    &=\frac{(-qw;q^2)_{\infty}}{(q^2;q^2)_{\infty}}\mathrm{CT}_z\big[ \sum_{s\in\mathbb{Z}}q^{s^2}z^{-s}\sum_{k\geq 0}\frac{q^{2k^2}(v^2wz)^k(-qw^{-1};q^2)_k}{(q^2;q^2)_{2k}}\cdot (-q^{-2k+1}(uz)^{-1};q^2)_{2k}\nonumber\\
    &\quad \times (q^2,-q^{-2k+1}uz,-q^{2k+1}(uz)^{-1};q^2)_{\infty} \big].
\end{align}

By \eqref{q-binomial-finite} and \eqref{Jacobi}, we have
\begin{align}
&(-q^{-2k+1}(uz)^{-1};q^2)_{2k}=\sum_{i=0}^{2k}\qbinom{2k}{i}_{q^2}q^{i^2-2ki}(uz)^{-i},\label{qbinomial-A0}\\
    &(q^2,-q^{-2k+1}uz,-q^{2k+1}(uz)^{-1};q^2)_{\infty}=\sum_{t\in\mathbb{Z}}q^{t^2-2kt}(uz)^t.\label{Jacobi-A0}
\end{align}
Substituting \eqref{qbinomial-A0} and \eqref{Jacobi-A0} into \eqref{A0-step1}, we have
\begin{align*} 
    &N_0(u,v,w;q)=\frac{(-qw;q^2)_{\infty}}{(q^2;q^2)_{\infty}}\mathrm{CT}_z\Big[ \sum_{s\in\mathbb{Z}}q^{s^2}z^{-s}\sum_{k\geq 0, t\in\mathbb{Z}}\sum_{i=0}^{2k}\frac{q^{2k^2}(v^2wz)^k(-qw^{-1};q^2)_k}{(q^2;q^2)_{2k}}\nonumber\\
    &\qquad \times q^{t^2-2kt}(uz)^t\cdot \qbinom{2k}{i}_{q^2}q^{i^2-2ki}(uz)^{-i} \Big]\nonumber\\
    &=\frac{(-qw;q^2)_{\infty}}{(q^2;q^2)_{\infty}}\sum_{\substack{k\geq 0, t\in\mathbb{Z}\\0\leq i\leq 2k}}q^{(k+t-i)^2}\cdot\frac{q^{2k^2}(v^2w)^k(-qw^{-1};q^2)_k}{(q^2;q^2)_{2k}}\cdot q^{t^2-2kt}u^t\cdot \qbinom{2k}{i}_{q^2}q^{i^2-2ki}u^{-i}\nonumber\\
    &=\frac{(-qw;q^2)_{\infty}}{(q^2;q^2)_{\infty}}\sum_{\substack{k\geq 0, t\in\mathbb{Z}\\0\leq i\leq 2k}}\frac{q^{3k^2+2i^2-4ki+2t^2-2ti}(v^2w)^ku^{t-i}(-qw^{-1};q^2)_k}{(q^2;q^2)_i(q^2;q^2)_{2k-i}}.
\end{align*}

Similarly, we have that
\begin{align}\label{A1-step1}
    &N_1(u,v,w;q)=(-w;q^2)_{\infty}\nonumber\\
    & \quad \times\mathrm{CT}_z\Big[ \sum_{\substack{i,j,k\geq0\\s\in\mathbb{Z} }}\frac{q^{(i-k)^2+(j-k)^2+s^2-2j+3k+1}u^{i-j}v^{2k+1}w^k(-q^2w^{-1};q^2)_kz^{i-j+k-s}}{(q^2;q^2)_i(q^2;q^2)_j(q^2;q^2)_{2k+1}} \Big]\nonumber\\
    &=(-w;q^2)_{\infty}\mathrm{CT}_z\Big[ \sum_{s\in\mathbb{Z}}q^{s^2}z^{-s} \sum_{k\geq 0}\frac{q^{2k^2+3k+1}v^{2k+1}w^kz^k(-q^2w^{-1};q^2)_k}{(q^2;q^2)_{2k+1}}\nonumber\\
    & \quad \times (-q^{-2k+1}uz,-q^{-2k-1}(uz)^{-1};q^2)_{\infty}\Big] \nonumber\\
    &=\frac{(-w;q^2)_{\infty}}{(q^2;q^2)_{\infty}}\mathrm{CT}_z\Big[ \sum_{s\in\mathbb{Z}}q^{s^2}z^{-s} \sum_{k\geq 0}\frac{q^{2k^2+3k+1}v^{2k+1}w^kz^k(-q^2w^{-1};q^2)_k}{(q^2;q^2)_{2k+1}}\nonumber\\
    & \quad \times (q^2,-q^{-2k+1}uz,-q^{2k+1}(uz)^{-1};q^2)_{\infty}\cdot (-q^{-2k-1}(uz)^{-1};q^2)_{2k+1}\Big].
\end{align}

By \eqref{q-binomial-finite}, we have
\begin{align}\label{qbinomial-A1}
    (-q^{-2k-1}(uz)^{-1};q^2)_{2k+1}=\sum_{i=0}^{2k+1}\qbinom{2k+1}{i}_{q^2}q^{i^2-2i-2ki}(uz)^{-i}.
\end{align}

Substituting \eqref{Jacobi-A0} and \eqref{qbinomial-A1} into \eqref{A1-step1}, we have

\begin{align*} 
    &N_1(u,v,w;q)=\frac{(-w;q^2)_{\infty}}{(q^2;q^2)_{\infty}}\mathrm{CT}_z\Big[ \sum_{s\in\mathbb{Z}}q^{s^2}z^{-s} \sum_{k\geq 0}\frac{q^{2k^2+3k+1}v^{2k+1}w^kz^k(-q^2w^{-1};q^2)_k}{(q^2;q^2)_{2k+1}}\nonumber\\
    & \quad \times \sum_{t\in\mathbb{Z}} q^{t^2-2kt}(uz)^t \sum_{i=0}^{2k+1} \qbinom{2k+1}{i}_{q^2}q^{i^2-2i-2ki}(uz)^{-i} \Big]\nonumber\\
    &=\frac{(-w;q^2)_{\infty}}{(q^2;q^2)_{\infty}}\sum_{\substack{k\geq 0,  t\in\mathbb{Z}\\ 0\leq i\leq2k+1}}q^{(k+t-i)^2}\cdot \frac{q^{2k^2+3k+1}v^{2k+1}w^k(-q^2w^{-1};q^2)_k}{(q^2;q^2)_{2k+1}}\cdot q^{t^2-2kt}u^t\nonumber\\
    & \quad \times \qbinom{2k+1}{i}_{q^2}q^{i^2-2i-2ki}u^{-i}\nonumber\\
    &=\frac{(-w;q^2)_{\infty}}{(q^2;q^2)_{\infty}}\sum_{\substack{k\geq 0,  t\in\mathbb{Z}\\ 0\leq i\leq2k+1}}\frac{q^{3k^2+2i^2-4ki+2t^2-2ti+3k-2i+1}u^{t-i}v^{2k+1}w^k(-q^2w^{-1};q^2)_k}{(q^2;q^2)_i(q^2;q^2)_{2k+1-i}}. \qedhere
\end{align*}
\end{proof}

\begin{proof}[Proof of Theorem \ref{thm-CW-conj3.3}]
We may assume that $u=q^c$ with $c\in \mathbb{Q}$. We have
\begin{align}\label{completing1}
    \sum_{t\in\mathbb{Z}}q^{2t^2-2ti}u^t&=\sum_{t\in\mathbb{Z}}q^{2t^2-2ti+ct}=\sum_{t\in\mathbb{Z}}q^{2(t+\frac{c}{4}-\frac{i}{2})^2-\frac{1}{2}(\frac{c}{2}-i)^2}.
\end{align}
Substituting \eqref{completing1} into \eqref{add-N0-exp} and \eqref{add-N1-exp}, we have
\begin{align}
    &N_0(u,v,w;q)=\frac{(-qw;q^2)_{\infty}}{(q^2;q^2)_{\infty}}\sum_{\substack{k\geq 0\\ 0\leq i\leq 2k}}\frac{q^{3k^2+2i^2-4ki-ci-\frac{1}{2}(\frac{c}{2}-i)^2}(v^2w)^k(-qw^{-1};q^2)_k}{(q^2;q^2)_i(q^2;q^2)_{2k-i}}\nonumber\\
    &\quad \times\sum_{t\in\mathbb{Z}}q^{2(t+\frac{c}{4}-\frac{i}{2})^2},\\
    &N_1(u,v,w;q)=\frac{(-w;q^2)_{\infty}}{(q^2;q^2)_{\infty}}\sum_{\substack{k\geq 0\\ 0\leq i\leq 2k+1}}\frac{q^{3k^2+2i^2-4ki+3k-2i+1-ci-\frac{1}{2}(\frac{c}{2}-i)^2}v^{2k+1}w^k}{(q^2;q^2)_i(q^2;q^2)_{2k+1-i}}\nonumber\\
    &\quad \times (-q^2w^{-1};q^2)_k \sum_{t\in\mathbb{Z}}q^{2(t+\frac{c}{4}-\frac{i}{2})^2}.
\end{align}
It is clear that the summation over $t$ only depends on the parity of $i$. Then it is necessary to separate the series $N_0$ and $N_1$ into two parts respectively. We write
\begin{align}
    N_0(u,v,w;q)&=N_{0,0}(u,v,w;q)+N_{0,1}(u,v,w;q),\\
    N_1(u,v,w;q)&=N_{1,0}(u,v,w;q)+N_{1,1}(u,v,w;q).
\end{align}
Here
\begin{align}
    &N_{0,0}(u,v,w;q):=\frac{(-qw;q^2)_{\infty}}{(q^2;q^2)_{\infty}}\sum_{k\geq 0}\sum_{i=0}^k\frac{q^{3k^2+6i^2-8ki-ci-\frac{c^2}{8}}(v^2w)^k(-qw^{-1};q^2)_k}{(q^2;q^2)_{2i}(q^2;q^2)_{2k-2i}}\nonumber\\
    &\quad \times \sum_{t\in\mathbb{Z}}q^{2(t+\frac{c}{4})^2},\label{A00-0}\\
    &N_{0,1}(u,v,w;q):=\frac{(-qw;q^2)_{\infty}}{(q^2;q^2)_{\infty}}\sum_{k\geq 0}\sum_{i=0}^{k-1}\frac{q^{3k^2+6i^2-8ki-4k-ci+6i-\frac{c^2}{8}-\frac{c}{2}+\frac{3}{2}}(v^2w)^k}{(q^2;q^2)_{2i+1}(q^2;q^2)_{2k-2i-1}}\nonumber\\
    &\quad \times (-qw^{-1};q^2)_k \sum_{t\in\mathbb{Z}}q^{2(t+\frac{c}{4}+\frac{1}{2})^2},\label{A01-0}\\
    &N_{1,0}(u,v,w;q):=\frac{(-w;q^2)_{\infty}}{(q^2;q^2)_{\infty}}\sum_{k\geq 0}\sum_{i=0}^k\frac{q^{3k^2+6i^2-8ki+3k-4i-ci+1-\frac{c^2}{8}}v^{2k+1}w^k}{(q^2;q^2)_{2i}(q^2;q^2)_{2k+1-2i}}\nonumber\\
    &\quad \times (-q^2w^{-1};q^2)_k \sum_{t\in\mathbb{Z}}q^{2(t+\frac{c}{4})^2},\label{A10-0}\\
    &N_{1,1}(u,v,w;q):=\frac{(-w;q^2)_{\infty}}{(q^2;q^2)_{\infty}}\sum_{k\geq 0}\sum_{i=0}^{k}\frac{q^{3k^2+6i^2-8ki-k-ci+2i-\frac{c^2}{8}-\frac{c}{2}+\frac{1}{2}}v^{2k+1}w^k}{(q^2;q^2)_{2i+1}(q^2;q^2)_{2k-2i}}\nonumber\\
    &\quad \times (-q^2w^{-1};q^2)_k \sum_{t\in\mathbb{Z}}q^{2(t+\frac{c}{4}+\frac{1}{2})^2}\label{A11-0}.
\end{align}
Setting $k=i+j$ in \eqref{A00-0}, \eqref{A10-0}, \eqref{A11-0} and $k=i+j+1$ in \eqref{A01-0}, we have
\begin{align}
    N_{0,0}(u,v,w;q)&=\frac{(-qw;q^2)_{\infty}}{(q^2;q^2)_{\infty}}\sum_{t\in\mathbb{Z}}q^{2(t+\frac{c}{4})^2}\sum_{i,j\geq 0}\frac{q^{i^2+3j^2-2ij-ci-\frac{c^2}{8}}(v^2w)^{i+j}}{(q^2;q^2)_{2i}(q^2;q^2)_{2j}}\nonumber\\
    &\quad \times (-qw^{-1};q^2)_{i+j},\label{A00}\\
    N_{0,1}(u,v,w;q)&=\frac{(-qw;q^2)_{\infty}}{(q^2;q^2)_{\infty}}\sum_{t\in\mathbb{Z}}q^{2(t+\frac{c}{4}+\frac{1}{2})^2}\sum_{i,j\geq 0}\frac{q^{i^2+3j^2-2ij-ci+2j-\frac{c^2}{8}-\frac{c}{2}+\frac{1}{2}}(v^2w)^{i+j+1}}{(q^2;q^2)_{2i+1}(q^2;q^2)_{2j+1}}\nonumber\\
    &\quad \times (-qw^{-1};q^2)_{i+j+1},\\
    N_{1,0}(u,v,w;q)&=\frac{(-w;q^2)_{\infty}}{(q^2;q^2)_{\infty}}\sum_{t\in\mathbb{Z}}q^{2(t+\frac{c}{4})^2}\sum_{i,j\geq 0}\frac{q^{i^2+3j^2-2ij-i-ci+3j+1-\frac{c^2}{8}}v^{2i+2j+1}w^{i+j}}{(q^2;q^2)_{2i}(q^2;q^2)_{2j+1}}\nonumber\\
    &\quad \times (-q^2w^{-1};q^2)_{i+j},\\
    N_{1,1}(u,v,w;q)&=\frac{(-w;q^2)_{\infty}}{(q^2;q^2)_{\infty}}\sum_{t\in\mathbb{Z}}q^{2(t+\frac{c}{4}+\frac{1}{2})^2}\sum_{i,j\geq 0}\frac{q^{i^2+3j^2-2ij+i-ci-j-\frac{c^2}{8}-\frac{c}{2}+\frac{1}{2}}v^{2i+2j+1}w^{i+j}}{(q^2;q^2)_{2i+1}(q^2;q^2)_{2j}}\nonumber\\
    &\quad \times (-q^2w^{-1};q^2)_{i+j}.\label{A11}
\end{align}

Using \eqref{finite-product-1} and \eqref{finite-product-2} with $k=i+j$ and $k=i+j+1$,  we have
\begin{align}
    &N_{0,0}(u,v,w;q)=\frac{\sum_{t\in\mathbb{Z}}q^{2(t+\frac{c}{4})^2}}{(q^2;q^2)_{\infty}}\sum_{i,j\geq 0}\frac{q^{2i^2+4j^2-ci-\frac{c^2}{8}}v^{2i+2j}}{(q^2;q^2)_{2i}(q^2;q^2)_{2j}} (-q^{1-2(i+j)}w;q^2)_{\infty}\nonumber\\
    &=\frac{q^{-\frac{c^2}{8}}\sum_{t\in\mathbb{Z}}q^{2(t+\frac{c}{4})^2}}{(q^2;q^2)_{\infty}}\sum_{i,j,k\geq 0}\frac{q^{2i^2+4j^2+k^2-2ik-2jk}u^{-i}v^{2i+2j}w^k}{(q^2;q^2)_{2i}(q^2;q^2)_{2j}(q^2;q^2)_k} \nonumber \\
    &=\frac{q^{-\frac{c^2}{8}}\sum_{t\in\mathbb{Z}}q^{2(t+\frac{c}{4})^2}}{(q^2;q^2)_{\infty}}\Big(M_{0,0,0}(u^{-\frac{1}{2}}v,w,v;q^{\frac{1}{2}})+M_{0,1,0}(u^{-\frac{1}{2}}v,w,v;q^{\frac{1}{2}})\Big) ,\label{M-1}\\
    &N_{0,1}(u,v,w;q)=\frac{\sum_{t\in\mathbb{Z}}q^{2(t+\frac{c}{4}+\frac{1}{2})^2}}{(q^2;q^2)_{\infty}}\sum_{i,j\geq 0}\frac{q^{2i^2+4j^2-ci+2i+4j-\frac{c^2}{8}-\frac{c}{2}+\frac{3}{2}}v^{2i+2j+2}}{(q^2;q^2)_{2i+1}(q^2;q^2)_{2j+1}}\nonumber\\
    &\quad \times (-q^{1-2(i+j+1)}w;q^2)_{\infty}\nonumber\\
    &=\frac{q^{-\frac{c^2}{8}}\sum_{t\in\mathbb{Z}}q^{2(t+\frac{c}{4}+\frac{1}{2})^2}}{(q^2;q^2)_{\infty}}\sum_{i,j,k\geq 0}\frac{q^{2i^2+4j^2+k^2-2ik-2jk+2i+4j-2k+\frac{3}{2}}u^{-i-\frac{1}{2}}v^{2i+2j+2}w^k}{(q^2;q^2)_{2i+1}(q^2;q^2)_{2j+1}(q^2;q^2)_k} \nonumber \\
    &=\frac{q^{-\frac{c^2}{8}}\sum_{t\in\mathbb{Z}}q^{2(t+\frac{c}{4}+\frac{1}{2})^2}}{(q^2;q^2)_{\infty}}\Big(M_{1,0,1}(u^{-\frac{1}{2}}v,w,v;q^{\frac{1}{2}})+M_{1,1,1}(u^{-\frac{1}{2}}v,w,v;q^{\frac{1}{2}})\Big). \label{M-2}
\end{align}
Here for the second equality in \eqref{M-1} and \eqref{M-2}  we used \eqref{Euler} and the fact $u=q^c$. For the last equality we used \eqref{def-Mxyz} with the variables $j$ and $k$ interchanged.

Similarly, using \eqref{finite-product-2} with $k=i+j$, we have
\begin{align}
    &N_{1,0}(u,v,w;q)=\frac{\sum_{t\in\mathbb{Z}}q^{2(t+\frac{c}{4})^2}}{(q^2;q^2)_{\infty}}\sum_{i,j\geq 0}\frac{q^{2i^2+4j^2-ci+4j+1-\frac{c^2}{8}}v^{2i+2j+1}}{(q^2;q^2)_{2i}(q^2;q^2)_{2j+1}} (-q^{-2(i+j)}w;q^2)_{\infty}\nonumber\\
    &=\frac{q^{-\frac{c^2}{8}}\sum_{t\in\mathbb{Z}}q^{2(t+\frac{c}{4})^2}}{(q^2;q^2)_{\infty}}\sum_{i,j,k\geq 0}\frac{q^{2i^2+4j^2+k^2-2ik-2jk+4j-k+1}u^{-i}v^{2i+2j+1}w^k}{(q^2;q^2)_{2i}(q^2;q^2)_{2j+1}(q^2;q^2)_{k}} \nonumber \\
    &=\frac{q^{-\frac{c^2}{8}}\sum_{t\in\mathbb{Z}}q^{2(t+\frac{c}{4})^2}}{(q^2;q^2)_{\infty}}\Big(M_{0,0,1}(u^{-\frac{1}{2}}v,w,v;q^{\frac{1}{2}})+M_{0,1,1}(u^{-\frac{1}{2}}v,w,v;q^{\frac{1}{2}})\Big),  \label{M-3}\\
    &N_{1,1}(u,v,w;q)=\frac{\sum_{t\in\mathbb{Z}}q^{2(t+\frac{c}{4}+\frac{1}{2})^2}}{(q^2;q^2)_{\infty}}\sum_{i,j\geq 0}\frac{q^{2i^2+4j^2+2i-ci-\frac{c^2}{8}-\frac{c}{2}+\frac{1}{2}}v^{2i+2j+1}}{(q^2;q^2)_{2i+1}(q^2;q^2)_{2j}}\nonumber\\
    &\quad \times (-q^{-2(i+j)}w;q^2)_{\infty}\nonumber\\
    &=\frac{q^{-\frac{c^2}{8}}\sum_{t\in\mathbb{Z}}q^{2(t+\frac{c}{4}+\frac{1}{2})^2}}{(q^2;q^2)_{\infty}}\sum_{i,j,k\geq 0}\frac{q^{2i^2+4j^2+k^2-2ik-2jk+2i-k+\frac{1}{2}}u^{-i-\frac{1}{2}}v^{2i+2j+1}w^k}{(q^2;q^2)_{2i+1}(q^2;q^2)_{2j}(q^2;q^2)_k} \nonumber \\
    &= \frac{q^{-\frac{c^2}{8}}\sum_{t\in\mathbb{Z}}q^{2(t+\frac{c}{4}+\frac{1}{2})^2}}{(q^2;q^2)_{\infty}}\Big(M_{1,0,0}(u^{-\frac{1}{2}}v,w,v;q^{\frac{1}{2}})+M_{1,1,0}(u^{-\frac{1}{2}}v,w,v;q^{\frac{1}{2}})\Big). \label{M-4}
\end{align}

Adding \eqref{M-1}--\eqref{M-4} together, we obtain with $u=q^c$ that
\begin{align}
    &N(u,v,w;q)=N_{0,0}(u,v,w;q)+N_{0,1}(u,v,w;q)+N_{1,0}(u,v,w;q)+N_{1,1}(u,v,w;q)\nonumber\\
    &=\frac{q^{-\frac{c^2}{8}}}{(q^2;q^2)_{\infty}}\Big(\sum_{t\in\mathbb{Z}}q^{2(t+\frac{c}{4})^2}\big(M_{0,0,0}(u^{-\frac{1}{2}}v,w,v;q^{\frac{1}{2}})+M_{0,1,0}(u^{-\frac{1}{2}}v,w,v;q^{\frac{1}{2}})\nonumber\\
    &\quad +M_{0,0,1}(u^{-\frac{1}{2}}v,w,v;q^{\frac{1}{2}})+M_{0,1,1}(u^{-\frac{1}{2}}v,w,v;q^{\frac{1}{2}})\big) \nonumber \\
    &\quad +\sum_{t\in\mathbb{Z}}q^{2(t+\frac{c+2}{4})^2}\big( M_{1,0,1}(u^{-\frac{1}{2}}v,w,v;q^{\frac{1}{2}})+M_{1,1,1}(u^{-\frac{1}{2}}v,w,v;q^{\frac{1}{2}})\nonumber\\
    &\quad +M_{1,0,0}(u^{-\frac{1}{2}}v,w,v;q^{\frac{1}{2}})+M_{1,1,0}(u^{-\frac{1}{2}}v,w,v;q^{\frac{1}{2}}) \big)\Big).
\end{align}

Setting $(u,v,w)=(1,1,1)$, $(q^{-2},1,q^{-1})$, $(q^{-4},q^{-1},1)$, $(q^{-2},1,1)$ respectively, and then replacing $q$ by $q^2$, we deduce that
\begin{align}
  N^{(h)}(q^2)&=\frac{q^{c_h}}{J_4}\Big(a(q)\big(A^{(h)}(q)+3C^{(1)}(q)\big)+b(q)\big(3B^{(h)}(q)+D^{(h)}(q)\big)\Big). \label{Nh-ABCD}
\end{align}
Here the values of $c_h$ are given in Theorem \ref{thm-CW-conj3.3}. 

Note that 
\begin{align}\label{ab-phi}
    a(q)=\frac{1}{2}\big(\varphi(q)+\varphi(-q)\big), \quad b(q)=\frac{1}{2}\big(\varphi(q)-\varphi(-q)\big).
\end{align}
Replacing $q$ by $-q$ in the identity \eqref{M-ABCD}, we deduce that
\begin{align}\label{Mh-negative}
    M^{(h)}(-q)=(-1)^{h-1}\big(A^{(h)}(q)-3B^{(h)}(q)+3C^{(h)}(q)-D^{(h)}(q)\big).
\end{align}
Substituting \eqref{ab-phi} into \eqref{Nh-ABCD} and then using \eqref{M-ABCD} and \eqref{Mh-negative}, we obtain \eqref{eq-thm-Nh}.

Finally, substituting  \eqref{eq-thm-Mh} into \eqref{eq-thm-Nh} with $h=2$, we can express $N^{(2)}(q)$ in terms of $J_m$ and $J_{a,m}$. Using the Maple approach in \cite{Frye-Garvan}, it is easy to prove that \eqref{eq-CW-conj3.3} holds.
\end{proof}

We end this paper by giving a proof of Li's conjecture \cite[Conjecture 5.2]{Li-arXiv}.
\begin{corollary}\label{cor-Li-conj}
The conjectural identities \eqref{eq-Li-1} and \eqref{eq-Li-2} hold.
\end{corollary}
\begin{proof}
We denote the left side of \eqref{eq-Li-1} and \eqref{eq-Li-2} as $L_1(q)$ and $L_2(q)$, respectively. 
From \eqref{M-defn-general} and \eqref{M-defn} it is easy to see that 
\begin{align}
M^{(2)}(q)=L_1(q^2)+qL_2(q^2).
\end{align}
Recall a 2-dissection formula from \cite[(2.18)]{Yao-Xia}:
\begin{align}
    \frac{J_3}{J_1}=\frac{J_4J_6J_{16}J_{24}^2}{J_2^2J_8J_{12}J_{48}}+q\frac{J_6J_8^2J_{48}}{J_2^2J_{16}J_{24}}.
\end{align}
Substituting it into \eqref{M2-simple} and then extracting the even and odd power terms, we obtain \eqref{eq-Li-1} and \eqref{eq-Li-2}, respectively.
\end{proof}

\subsection*{Acknowledgements}
This work was supported by the National Key R\&D Program of China (Grant No.\ 2024YFA1014500).

\subsection*{Declaration of AI use}
Theorems \ref{thm-dual-exam12} and \ref{thm-dual-rank3} were proposed as conjectures in the first version of this paper. ChatGPT (OpenAI, GPT-5.6 Sol) was then used to explore possible proof strategies and it helped us to prove these theorems. It also helped to find the first approach to Theorem \ref{thm-dual-exam9-1st} and improve the presentation of the manuscript. All AI-generated suggestions were independently verified, refined, and rewritten by the authors. The authors take full responsibility for the content of this work.

\end{document}